\documentclass[12pt]{amsart}
\usepackage[utf8]{inputenc}
\usepackage{amsthm}
\usepackage{amsmath}
\usepackage[margin=1.1in]{geometry}
\usepackage{amssymb}
\usepackage{amscd}
\usepackage{mathrsfs}
\usepackage{mathtools}
\usepackage{amsfonts}
\usepackage{enumitem}
\usepackage{needspace}

\usepackage{tikz-cd}
\usepackage{tabularx}
\usepackage[colorlinks=true,linkcolor=blue, citecolor=cyan]{hyperref}
\usepackage[title]{appendix}
\usepackage[T1]{fontenc}

\newcommand{\coh}{\mathrm{H}}
\newcommand{\Gal}{\mathrm{Gal}}

\newcommand{\Spec}{\mathrm{Spec}}
\newcommand{\Pic}{\mathrm{Pic}}

\newcommand{\Hom}{\textrm{Hom}}

\newcommand{\Perv}{\textrm{\textbf{Perv}}}

\newcommand{\etcoh}{\textrm{H}_{\textrm{\'et}}}

\newcommand{\g}{\mathfrak{g}}
\newcommand{\gres}{\widetilde{\mathfrak{g}}}
\newcommand{\h}{\mathfrak{h}}
\newcommand{\Oring}{\mathcal{O}}
\newcommand{\borel}{\mathfrak{b}}

\newcommand{\Spf}{\textrm{Spf}}

\newcommand{\etabar}{\overline{\eta}}

\theoremstyle{plain}
\newtheorem{Thm}{Theorem}
\newtheorem{Lemma}[Thm]{Lemma}
\newtheorem{Cor}[Thm]{Corollary}
\newtheorem{Prop}[Thm]{Proposition}
\theoremstyle{definition}
\newtheorem{Def}[Thm]{Definition}
\theoremstyle{definition}
\newtheorem{ex}[Thm]{Example}
\theoremstyle{definition}

\theoremstyle{definition}
\newtheorem{Rem}[Thm]{Remark}
\newtheorem{Ass}[Thm]{Assumption}
\newtheorem{Constr}[Thm]{Construction}
\theoremstyle{definition}

\newtheoremstyle{break}
  {\topsep}{\topsep}%
  {\itshape}{}%
  {\bfseries}{}%
  {\newline}{}%
\theoremstyle{break}
\theoremstyle{definition}

\newtheorem{?}{Problem}
\theoremstyle{definition}
\newtheorem*{assumptionstar}{Assumption $(*)$}

\numberwithin{Thm}{section}
\numberwithin{equation}{section}

\hypersetup{
  linkcolor=blue,
  citecolor=blue!60!black,
  urlcolor=red!60!black
}
\usepackage{etoolbox}

\newcommand{\stacksref}[1]{%
  \cite[\href{https://stacks.math.columbia.edu/tag/#1}
    {Tag~#1}]{stacks-project}%
}

\makeatletter
\patchcmd{\@startsection}
  {\@sect{#1}{\@m}}
  {\@sect{#1}{\@M}}
  {}{}
\makeatother

\subjclass[2020]{Primary 14D05, 14J17;
  Secondary 14B07, 14F20, 14G20}

\keywords{monodromy, rational singularities,
  mixed characteristic, simultaneous resolutions, nearby cycles,
  good reduction}

\title{Monodromy of rational singularities on surfaces in mixed characteristic}
\author{Jason Kountouridis}
\date{}

\begin{document}

\begin{abstract}
    We study the monodromy on the $\ell$-adic cohomology of smooth proper surfaces over complete discrete valuation fields of mixed characteristic $(0,p)$, with algebraically closed residue field and $p>5$. When the surfaces admit proper models whose special fiber has at worst a rational double point, we show that the monodromy action of inertia is finite, depends on an \'etale neighborhood of the singularity and its nontrivial part coincides with a restriction of the reflection representation of a Weyl group naturally associated with the singularity. We extend this description for degenerations into general rational singularities whose deformations lie in the Artin component: the monodromy is described by restricting the reflection representation of the Weyl group associated with the exceptional $(-2)$-curves in the minimal resolution of the special fiber. Moreover we determine minimal finite extensions over which such surfaces have good reduction. The proof relies on a mixed-characteristic extension of Grothendieck--Springer simultaneous resolutions and an identification of monodromy and Springer representations via relatively perverse nearby cycles.
\end{abstract}

\maketitle
\begingroup
\small
\tableofcontents
\endgroup

\section{Introduction}

Let $\Lambda$ be a henselian discrete valuation ring of mixed characteristic $(0,p)$, by which we mean that $K=\textrm{Frac}(\Lambda)$ has characteristic zero and residue field $k$ is perfect of characteristic $p>0$. Given a smooth proper variety $X/K$ with a proper flat model $\mathcal{X}/\Lambda$, a central problem is to relate the geometry of the special fiber $\mathcal{X}_k$ to the action of inertia $I_K$ on $\coh_{\textrm{\'et}}^i(X_{\overline{K}},\mathbb{Q}_{\ell})$ for $\ell \neq p$. Here $I_K = \ker(\Gal_K \twoheadrightarrow \Gal_k)$ and its action on cohomology is referred to as the \emph{$\ell$-adic monodromy} action.

If $\mathcal{X}$ is smooth over $\Lambda$ then by smooth proper base-change we have that the $I_K$-action is trivial, i.e.\ all $\coh_{\textrm{\'et}}^i(X_{\overline{K}},\mathbb{Q}_{\ell})$ are unramified $\Gal_K$-representations. However, unramified cohomology does not imply that $X$ admits a smooth proper model after extension. For special classes of varieties, trivial monodromy action is equivalent to good reduction: if $X$ is an abelian variety, then by the N\'eron--Ogg--Shafarevich criterion $X$ admits a smooth proper model over $\Lambda$ if and only if $\etcoh^1(X_{\overline{K}},\mathbb{Q}_{\ell})$ is unramified (\cite[Thm.\ 1]{serre1968good}). 

For general varieties, Grothendieck's local monodromy theorem asserts that $I_K$ acts quasi-unipotently, i.e.\ there is an open subgroup $U \subset I_K$ so that $g\in U$ acts as $\textrm{exp}(t_{\ell}(g)N)$ where $t_{\ell}$ is the tame Kummer character and $N$ is the monodromy operator \cite[Expos\'e I, Th.\ 1.2]{grothendieck1973groupes}. When $X$ has strictly semistable reduction over $\Lambda$, $I_K$ acts unipotently and Rapoport--Zink have given a geometric description of $N$ via nearby cycles and a spectral sequence whose terms are the cohomology of intersections of components of the special fiber (\cite{rapoport1982lokale}). 

The study of nearby cycles has also been fruitful in the non-semistable case when $\mathcal{X}_k$ has ordinary double point singularities: in \cite{deligne2006groupes} the $\ell$-adic monodromy action is computed via Picard--Lefschetz formulas using vanishing cycle classes. These results were subsequently used in the study of Lefschetz pencils and Deligne's proof of the Weil conjectures (\cite[\S4.3 and \S5.8]{Deligne1974Conjecture}). 

A more recent arithmetic application of monodromy concerns good reductions of K3 surfaces: for a K3 surface $X$ and assuming the existence of a potential Kulikov model, Liedtke--Matsumoto show that $\etcoh^2(X_{\overline{K}},\mathbb{Q}_{\ell})$ unramified implies the existence of a projective model $\mathcal{X}/\Lambda$ with $\mathcal{X}_k$ having at worst rational double point singularities, and $X$ has good reduction after a finite \emph{unramified} extension of $K$ (\cite[Thm.\ 1.3]{liedtke2018good}). These results motivate the direct description of $\ell$-adic monodromy arising from degenerations to varieties with a prescribed set of singularities, and the arithmetic implications on good reduction of $X$.

In this paper we describe the monodromy action on $\coh^2(X_{\overline{K}},\mathbb{Q}_{\ell})$ for surface degenerations with rational double points (of which ordinary double points are the simplest case), and more generally with rational singularities whose deformations lie in their Artin component. In the spirit of \cite{deligne2006groupes} and \cite{illusie1994autour}, we use nearby cycles and more recent tools from perverse sheaves and Springer theory to determine the intrinsic $I_K$-action on $\etcoh^2(X_{\overline{K}}, \mathbb{Q}_{\ell})$ and detect good reduction criteria in this local setting.

\subsection{Main results.} Suppose now that $\Lambda$ is complete and residue field $k$ is algebraically closed. Let $X/K$ be a smooth proper surface with a proper flat model $\mathcal{X}/\Lambda$ so that $\mathcal{X}_k$ has a unique isolated rational double point singularity $x$. 

Rational double points on surfaces are characterized as normal singularities whose exceptional divisor in the minimal resolution is a tree of rational curves of self-intersection $-2$ (see Definition \ref{def: rational sings and rdps} for details). Such singularities are classically also known as Kleinian, Du Val or simple singularities, see \cite{durfee1979fifteen} for a thorough survey of their literature.

There is an important link between rational double points and $ADE$ Dynkin diagrams: if $E_1, \cdots, E_r$ denote the exceptional curves in the minimal resolution $\widetilde{\mathcal{X}}_k$, then under the intersection pairing on $\Pic(\widetilde{\mathcal{X}}_k)$, $\Phi= \bigoplus_i\mathbb{Z}[E_i]$ forms an irreducible root lattice of Dynkin type $A$, $D$ or $E$, with simple roots $[E_i]$. Under suitable restrictions on $p$, the Dynkin diagram of $\Phi$ (equivalently the dual graph of $x$) determines the formal isomorphism class of the singularity (see Section \ref{subsection-rational-double-point-definitions} for more; $p>5$ here suffices). Our first theorem is as follows:

\begin{Thm}[Theorem \ref{thm: main theorem about rdps}]\label{thm: main thm in intro}Let $p>5$ and $\mathcal{X}/\Lambda$ a proper flat surface degeneration with $X = \mathcal{X}_K$ a smooth proper surface and $\mathcal{X}_k$ having a unique rational double point $x$. Let $\Phi$ be its associated root lattice, $V_{\ell} = \Phi^{\vee}\otimes_{\mathbb{Z}}\mathbb{Q}_{\ell}$ and $W$ the Weyl group of $\Phi$.
\begin{itemize}
    \item[\emph{(i)}] The monodromy $I_K$-action on $\coh^i(X_{\overline{K}},\mathbb{Q}_{\ell})$ is trivial for $i\neq 2$. For $i=2$ let $L$ be the fixed field of the kernel of the $I_K$-action, then we have an $I_K$-equivariant splitting
    \begin{equation}\coh^2(X_{\overline{K}},\mathbb{Q}_{\ell}) \simeq \coh^2(\mathcal{X}_k,\mathbb{Q}_{\ell}) \oplus V_{\ell}(-1)\end{equation}
    The monodromy $I_K$-action is trivial on the first summand, and on the second it factors through $\Gal(L/K)$, where it is identified with the restriction of the natural $W$-action on $V_\ell$ along an embedding $\Gal(L/K) \hookrightarrow W$.

    \item[\emph{(ii)}] $L$ is the minimal Galois extension over which $X_K$ has good reduction in the category of algebraic spaces. Moreover, the isomorphism class of $I_K$-representation $V_\ell$ and the extension $L/K$ depend only on a formal affine neighborhood of $(\mathcal{X},x)$.
    \end{itemize}   
\end{Thm}

\begin{Rem} As a consequence, if $X$ has a rational double point reduction, then its $\ell$-adic monodromy is finite (equivalently $N=0$ for the monodromy operator). By (ii), $X$ has potential good reduction, but the smooth algebraic space model need not be a scheme, see Remark \ref{remark on algebraic space for good reduction}.\end{Rem}

We can extend Theorem \ref{thm: main thm in intro} to degenerations where $\mathcal{X}_k$ has an isolated rational surface singularity, under suitable restrictions on the model $\mathcal{X}/\Lambda$. Rational singularities are characterized by the vanishing of higher pushforwards along the minimal resolution: if $\pi: \widetilde{\mathcal{X}}_k \to \mathcal{X}_k$ is a minimal resolution then $\textrm{R}\pi_*\Oring_{\widetilde{\mathcal{X}}_k} \simeq \Oring_{\mathcal{X}_k}$. Rational double points form a subclass of rational singularities -- in general, the exceptional divisor over a rational singularity is a tree of rational curves of self-intersection $\leq -2$, with equality holding when the singularity is a rational double point.

The restriction on $\mathcal{X}$ is that the induced deformation of the singularity of $\mathcal{X}_k$ lies within a special component of the deformation space called the \emph{Artin component}. Such deformations are simultaneously resolvable in mixed characteristic by the methods of \cite{artin1974algebraic} and we may apply the arguments of Theorem \ref{thm: main thm in intro} to obtain the following.

\begin{Thm}[Theorem \ref{thm: rational sing monodromy}]\label{thm: 2nd thm in intro} Let $p>5$ and $\mathcal{X}/\Lambda$ be a proper flat surface degeneration with $X=\mathcal{X}_K$ a smooth proper surface and $\mathcal{X}_k$ having a unique rational singularity so that $\mathcal{X}$ induces a deformation within its Artin component. Let $\Phi$ be the exceptional curve lattice inside $\Pic(\widetilde{X}_k)$ and let $\Psi \subset \Phi$ be the root sub-lattice generated by the exceptional curves of self-intersection $-2$, $W = W(\Psi)$ the associated Weyl group. Set $V_{\ell} = \Phi^{\vee}\otimes_{\mathbb{Z}}\mathbb{Q}_{\ell}$ and $N_{\ell} = \Psi^{\vee}\otimes_{\mathbb{Z}}\mathbb{Q}_{\ell}$, and view $N_{\ell} \subset V_{\ell}$ via the intersection pairing. Then the monodromy $I_K$-action on $\coh^i(X_{\overline{K}},\mathbb{Q}_{\ell})$ is trivial for $i\neq 2$ and for $i=2$ we have an $I_K$-equivariant splitting
    \[\coh^2(X_{\overline{K}},\mathbb{Q}_{\ell}) \simeq \coh^2(\mathcal{X}_k,\mathbb{Q}_{\ell})\oplus V_{\ell}(-1)\]
    The $I_K$-action is trivial on the first summand, and on the second summand it acts on $V_{\ell} = N_{\ell} \oplus N_{\ell}^{\perp} $ via restriction along a homomorphism $\rho: I_K \to W$, where $W$ acts on $N_{\ell}$ via reflections and on $N_{\ell}^{\perp}$ it acts trivially. Moreover if $L$ is the fixed field of $\ker \rho$ then $L/K$ is the minimal finite Galois extension over which $X$ has good reduction. 
\end{Thm}

The Artin component condition is automatic for all rational double points, so Theorem \ref{thm: 2nd thm in intro} is a natural extension of Theorem \ref{thm: main thm in intro}. We remark however that for deformations of rational singularities outside the Artin component we do not know whether the $\ell$-adic cohomological monodromy is finite; for projective surface degenerations over $\mathbb{C}$, the answer is positive using tools from Hodge theory (see \cite{KerrLazaSaito2022Deformation}, Cor.\ 1).

Our results naturally extend to $\mathcal{X}_k$ having finitely many rational singularities (cf.\ Remark \ref{remark: variants of main rdp theorem}). One observation is that if $X$ admits a semistable model over $\Lambda$ in addition to a rational double point model, then it has good reduction over $K$ since a finite unipotent monodromy action is trivial.

Another observation is that if $p \nmid |W|$, then the good reduction extension $L/K$ is tame and thus $\Gal(L/K)$ is generated by a single element $T$; we can describe the monodromy $I_K$-action via this \emph{monodromy operator} $T$. In this setting we recover a result of Liedtke--Matsumoto on tame extensions of K3 surfaces and good reduction (\cite[Cor.\ 1.5]{liedtke2018good}): if $X$ is a K3 surface with a rational double point model and $p\geq 23$, then since the maximal number of exceptional curves is bounded by $21$ one gets that $p\nmid |W|$ and $X$ has good reduction after a tame extension.

The simplest example of Theorem \ref{thm: main thm in intro} is the case of ordinary double point reduction, which has an elementary explanation.

\begin{ex}[Ordinary double points] Let $\pi\in\Lambda$ be a uniformizer and consider the family of quadrics \[X_n=\{z^2+xy = \pi^nw^2\} \subset \mathbb{P}_{\Lambda}^3, \;\; n \geq 1\]

The special fiber $(X_n)_k = \{z^2+xy=0\}$ has an ordinary double point and the generic fiber $(X_n)_K$ is $K$-isomorphic to $(X_2)_K$ or $(X_1)_K$ depending on the parity of $n$. 

Set $t = \pi^{n/2}\in \overline{K}$, $L=K(t)$, then $(X_n)_L \simeq \mathbb{P}_L^1 \times \mathbb{P}_L^1$. Let $a,b \in \coh^2((X_n)_{\overline{K}},\mathbb{Q}_{\ell}(1))$ be the classes of the two lines $\{x=0, z=\pm tw\}$ in $(X_n)_L$. Inertia fixes $a+b$ and acts on $a-b$ via $\chi_n$, where $\chi_n(g) = g(t)/t$. Thus as an $I_K$-representation we have

\[\coh^2((X_n)_{\overline{K}},\mathbb{Q}_{\ell}) \simeq \mathbb{Q}_{\ell}(-1)\oplus (\mathbb{Q}_{\ell}(-1)\otimes \chi_n)\] 

Set $\delta = a-b$. Since $(\delta, \delta)=-2$ we can describe the $I_K$-action as
\[g(v) = v + \frac{1-\chi_n(g)}{2}(v,\delta)\delta, \;\;\;g \in I_K\]

This recovers the Picard--Lefschetz formulas in \cite[Expos\'e XV, \S3.2.1-3.2.3]{deligne2006groupes}.

If $n$ is even then $t \in K$ and the monodromy action on $\coh^2((X_n)_{\overline{K}},\mathbb{Q}_{\ell})$ is trivial; in fact $(X_n)_K \simeq \mathbb{P}^1_K\times\mathbb{P}^1_K$ is already a split quadric and has good reduction over $K$ with associated smooth proper model $\mathbb{P}^1_{\Lambda}\times\mathbb{P}^1_{\Lambda}$. 

If $n$ is odd, the minimal field where $(X_n)_K$ attains good reduction is $L=K(\sqrt{\pi})$. The character $\chi_n$ is tamely ramified when $p>2$ and wildly ramified when $p=2$. \end{ex}

\subsection{Proof outline.} We give a sketch of the main steps in proving Theorem \ref{thm: main thm in intro}. The model $\mathcal{X}/\Lambda$ induces a formal deformation $\widehat{\Oring}_{\mathcal{X},x}$ of the singularity $x$, and a nearby cycles argument shows that the monodromy action on $\coh^2(X_{\overline{K}},\mathbb{Q}_{\ell})$ can be computed from a formal neighborhood of the singularity.

When the characteristic $p$ is `good' (cf.\ Section \ref{subsection-rational-double-point-definitions}) there is a connection between rational double points $x$ and simple simply-laced Lie algebras $\g$ via identifying the dual graph of the minimal resolution of $x$ with the Dynkin diagram of $\g$. Work of many people (\cite{brieskorn1970singular}, \cite{tyurina1970resolution}, \cite{slodowy1980four}) has expanded on this relationship and realized rational double points as the generic singularities in slices transversal to the nilpotent orbit inside the nilpotent cone $\mathcal{N} \subset \g$; such a realization uniquely pins down the singularity class in `good' characteristic. 

By results of Springer and Borho--MacPherson (\cite{springer1976trigonometric}, \cite{borho}), the Grothendieck--Springer resolution $\pi: \gres \to \g$ may be viewed as a simultaneous resolution of the fibers of the adjoint quotient $\g \to \h/\!\!/W$ after a Weyl base-change $\h \to \h/\!\!/W$. Restricting to appropriate slices, one can simultaneously resolve rational double points and their deformations (\cite{Slodowy1980}).

We develop a mixed-characteristic incarnation of Grothendieck--Springer theory and obtain via an appropriate pullback a simultaneous resolution for our deformation $\widehat{\Oring}_{\mathcal{X},x}$; we thus get a smooth proper algebraic space $\widetilde{\mathcal{X}}$ with generic fiber $X_L$ where $L/K$ is the extension defined by the pullback of the Weyl cover.

The formalism of nearby cycles along with the relative perversity of the universally locally acyclic sheaf $\textrm{R}\pi_*\mathbb{Q}_{\ell}[\dim \g]$ (see \cite{hansen2021relative}) allows us to identify the monodromy action on $\coh^2(X_{\overline{K}}, \mathbb{Q}_{\ell})$ with a suitable restriction of the Weyl group action on the cohomology of the fibers of $\pi: \gres \to \g$, also known as a Springer representation. This gives the required description of monodromy in Theorem \ref{thm: main thm in intro}.

For the general rational singularities considered in Theorem \ref{thm: 2nd thm in intro}, we can partially simultaneously resolve them after a finite extension and then descend to another model $\mathcal{X}'$ over $\Lambda$ whose special fiber has only rational double point singularities. We can therefore reduce the monodromy calculation to the case of Theorem \ref{thm: main thm in intro}.

We remark that this paper is a natural extension of the ideas in \cite{Slodowy1980}, which hold over fields of characteristic zero or $p \gg 0$, to the relative setting over a mixed characteristic base. Parts of this work also relate to Artin's simultaneous resolution functors (\cite{artin1974algebraic}), which were later used to extract Weyl covers as Galois covers and study the integral cohomology of partial flag varieties (\cite{shepherd2021weyl}) even in bad characteristic. The simultaneous resolution of general rational singularities over $\mathbb{C}$ has also been considered by Wahl in \cite{Wah79}. The key point for us is that good characteristic allows us to give a concrete description of Galois monodromy actions as Springer actions of the Weyl group.

\subsection*{Organization of the paper.} Section \ref{section-RDPS-in-mixed-char-and-deformations} contains some background on rational singularities, simultaneous resolutions and deformations of both rational double points and their resolutions. Section \ref{chevalley-section} contains the Lie-theoretic prerequisites for the proof of Theorems \ref{thm: main thm in intro} and \ref{thm: 2nd thm in intro}, with a focus on the Grothendieck--Springer resolution in mixed characteristic. Section \ref{section-integral-slices} uses Grothendieck--Springer theory to give a Lie-theoretic description of deformations and simultaneous resolutions of rational double points in mixed characteristic. Section \ref{section-monodromy-weyl-action} contains the proof of Theorem \ref{thm: main thm in intro} along with technical ingredients such as nearby cycles and relative perverse sheaves. Finally Section \ref{section: extension to rational sings} extends our results to suitable rational singularities and proves Theorem \ref{thm: 2nd thm in intro}.

\subsection*{Acknowledgements.} This paper is based on my PhD thesis, whose results have since been expanded on and improved, and so I am indebted to my advisor Matthew Emerton for his ideas and extensive comments on this work. I would also like to thank Alexander Beilinson and Eduard Looijenga for helpful conversations and suggestions, and Yifan Wei for pointing out a mistake in a previous version and suggesting a simple alternative solution.

\subsection*{Notations.}\label{notations} Throughout the paper we will assume the following conventions.

\begin{itemize}
    \item[\textbullet] $\Lambda$ denotes a complete discrete valuation ring of mixed characteristic $(0,p)$, i.e.\ the fraction field $K=\textrm{Frac}(\Lambda)$ has characteristic $0$ and the residue field $k$ has characteristic $p$. Unless stated otherwise we assume $\overline{k} = k$. Typical examples include $\Lambda=W(k)$ or a ramified extension $W(k)[T]/(T^n-p)$.
    \item[\textbullet] By a \emph{surface degeneration} $\mathcal{X}/\Lambda$ we mean a flat proper scheme $\mathcal{X} \to \Spec(\Lambda)$ with generic fiber $X=\mathcal{X}_K$ a smooth proper surface and special fiber $\mathcal{X}_k$ having a \emph{unique} isolated singularity, or finitely many isolated singularities of a prescribed type.
    \item[\textbullet] $G_K = \Gal_K \coloneqq \Gal(\overline{K}/K)$ is the absolute Galois group of $K$. In general we have an exact sequence \[1\to I_K \to G_K \to G_k \to 1\]
    which defines inertia $I_K$; when $k$ is algebraically closed, $I_K=G_K$.
    \item[\textbullet] Unless stated otherwise, all cohomologies $\coh^i(-, \mathbb{Q}_{\ell})$ are computed with respect to the small \'etale site.
\end{itemize}

 \section{Deformations and simultaneous resolutions of rational singularities}\label{section-RDPS-in-mixed-char-and-deformations} We collect here some preparatory results on rational double points (Section \ref{subsection-rational-double-point-definitions}), simultaneous resolutions (Section \ref{section-simultaneous-resolutions}) and the deformation theory of isolated rational singularities (Sections \ref{subsection-forma-deformations-of-sings}-\ref{section: artin components}).

\subsection{Rational double points.}\label{subsection-rational-double-point-definitions} Throughout, we let $k$ be an algebraically closed field of characteristic $p$.
\begin{Def}\label{def: rational sings and rdps} Suppose $X/k$ is an algebraic surface with isolated singularity $x \in X(k)$.
\begin{enumerate}
    \item[(i)] $x$ is a \emph{rational} singularity if for its minimal resolution\footnote{Such a resolution always exists for surfaces in characteristic $p$, see e.g. \cite{lipman1978desingularization}.} $f:\widetilde{X} \to X$ we have \[\textrm{R}f_*\Oring_{\widetilde{X}} \simeq \Oring_X\]
    \item[(ii)] $x$ is a \emph{rational double point} (henceforth an RDP) if it is normal and its minimal resolution $f: \widetilde{X} \to X$ has as reduced exceptional divisor a finite union of smooth rational curves with self intersection $-2$.
\end{enumerate}
    
\end{Def}

Rational singularities are normal since $f_*\Oring_{
\widetilde{X}} = \Oring_X$, and the definitions do not depend on $f$. Rational double points in particular are equivalently defined as rational Gorenstein surface singularities.

A crucial fact about RDPs is that in appropriate characteristics (see Definition \ref{definition-of-good-primes} below) they are \textit{taut}, i.e. their completed local ring is determined up to isomorphism by the dual graph of their minimal resolution. This is classical over $\mathbb{C}$, and has been extended to positive characteristic by \cite{artin1977coverings}:

\begin{Thm}[\cite{artin1966isolated}, \cite{artin1977coverings}]\label{classification-of-rdps}
    Let $X$ be an algebraic surface over $k$.
    \begin{enumerate}[label=\emph{(\roman*)}]
    \item  If $E$ is a connected curve on $X$ with smooth rational components $E_i$ so that $X$ is smooth along $E$, $E_i^2 = -2$ and their intersection matrix is negative-definite, then the only possible dual graphs for $E$ are the Dynkin diagrams $A_n$, $D_n \,(n \geq 4)$, $E_6, E_7$ and $E_8$.
    \item If $x \in X(k)$ is an RDP and $\textrm{\emph{char}}(k)$ is good \emph{(}see \emph{Definition \ref{definition-of-good-primes}}\emph{)} then up to analytic isomorphism $\widehat{\Oring}_{X,x} \simeq k[\![x,y,z]\!]/(f(x,y,z))$, where $f(x,y,z)$ and the dual graph of the minimal resolution of $x$ are given by the table below: \begin{center}\begin{tabularx}{0.8\textwidth} { 
  | >{\raggedright\arraybackslash}X 
  | >{\centering\arraybackslash}X 
  | >{\raggedleft\arraybackslash}X | }
 \hline 
 $f(x,y,z)$ & \emph{Dual graph} \\
 \hline
 $xz + y^{n+1}$  & $A_n$  \\
 
 $z^2 + x^2y + y^{n-1}$ & $D_n$ \\
 
 $z^2 + x^3 + y^4$ & $E_6$ \\
 
 $z^2 +x^3 + xy^3$ & $E_7$ \\
 
 $z^2 + x^3 + y^5$ & $E_8$ \\
\hline
\end{tabularx}
\end{center}  
    \end{enumerate}
\end{Thm}
A polynomial $f(x,y,z)$ defining an RDP is called a `normal form'. We note that each ADE Dynkin diagram determines a simple simply-laced root lattice $\Phi$, and we say an RDP has `Dynkin type $\Phi$' when its dual graph is the Dynkin diagram of $\Phi$. 

Since the notion of `good characteristic' will be ubiquitous in this paper, we record the definition here.

\begin{Def}\label{definition-of-good-primes} Given a semisimple simply-laced root lattice $\Phi$ of a Lie algebra of $ADE$ type over a characteristic $p$ field, we say $p$ is:
\begin{enumerate}[label=(\roman*)]

\item \textit{good} with respect to $\Phi$ if $p \neq 2$ when $\Phi$ has any $D_n$ components, $p \neq 2,3$ when $\Phi$ has any $E_6, E_7$ components and $p \neq 2,3,5$ when $\Phi$ has any $E_8$ components (note any $p$ is good for $A_n$ components).

\item \textit{very good} with respect to $\Phi$ if it is good and for any $A_n$ component we have $p \nmid n+1$.

\item \textit{sufficiently good} with respect to $\Phi$ if $p \nmid \lvert W(\Phi)\rvert$ i.e. it does not divide the order of the associated Weyl group of $\Phi$.\end{enumerate}
\end{Def}

For semisimple root lattices, sufficiently good $\Rightarrow$ very good $\Rightarrow$ good. Via Theorem \ref{classification-of-rdps}, we say $p$ is good (or very good, or sufficiently good) for an RDP singularity if it is so for the singularity's Dynkin type. However, Theorem \ref{classification-of-rdps} (ii) is no longer true when $p$ is \textit{not} good, because the singularities need not be taut and different types of RDPs can have the same dual graph; see (\cite{artin1977coverings}) for details.


\subsection{Simultaneous resolutions.}\label{section-simultaneous-resolutions} We next describe a particularly rare notion of resolving singularities of schemes in families. In order to be consistent with \cite{artin1974algebraic}, we enlarge the category of schemes to include algebraic spaces.

\begin{Def}\label{definition-of-simultaneous-resolution} Let $f: \mathcal{X}\longrightarrow S$ be a finite-type morphism of separated algebraic spaces. A \textit{simultaneous resolution over $S$} is a proper $S$-morphism $\pi: \widetilde{\mathcal{X}} \to \mathcal{X}$ so that $f \circ \pi$ is smooth and for every geometric point $\overline{s} \to S$, the induced map $\pi_s: \widetilde{\mathcal{X}}_{\overline{s}} \to \mathcal{X}_{\overline{s}}$ is a minimal resolution.
\end{Def}

Given a finite surjection $\psi: T \to S$, the simultaneous resolution of $f: \mathcal{X} \to S$ ``after basechange along $\psi$'' is a simultaneous resolution $\pi: \widetilde{\mathcal{X}} \to \mathcal{X} \times_S T$ over $T$; the induced map $\widetilde{\mathcal{X}} \to S$ will not be smooth in general when $\psi$ is ramified. When $S = \Spec(k)$, we recover the usual notion of a minimal resolution $\pi: \widetilde{\mathcal{X}} \to \mathcal{X}$ of algebraic spaces over $k$ (\cite[\href{https://stacks.math.columbia.edu/tag/0BHV}{Tag 0BHV}]{stacks-project}).

\begin{ex} Simultaneous resolutions rarely exist. Assume $p \neq 3$ and suppose \[X = \{x^3+y^3+z^3 = tw^3\} \subset \mathbb{P}_k^3 \times \mathbb{A}_t^1\] with $f$ being the projection to $\mathbb{A}_t^1$. The fibers over $t\neq0$ are smooth cubic surfaces with $\chi(\Oring) = 1$ while $X_0$ is the projective cone over elliptic curve $\{x^3+y^3+z^3=0\}$. The blowup of the isolated vertex singularity yields a $\mathbb{P}^1$-bundle $\widetilde{X_0}$ over the elliptic curve, whence $\chi(\Oring_{\widetilde{X_0}})=0$. Any other resolution of $X_0$ also has $\chi(\Oring) = 0$ as it is a birational invariant of smooth proper surfaces (see e.g. \cite[Prop. V.3.4]{hartshorne2013algebraic}). 

Hence, a simultaneous resolution of $X/\mathbb{A}_t^1$ cannot exist after finite base-change; note after a degree 3 basechange the smooth part of the family becomes isotrivial, and yet the obstruction to simultaneously resolving the family remains.
\end{ex}

Simultaneous resolutions have been shown to exist for surface families with RDPs, both in the complex-analytic and the algebraic setting:

\begin{Thm}[\cite{brieskorn1966resolution},\cite{tyurina1970resolution}, \cite{artin1974algebraic}]\label{thm: simultaneous rdp resolution} Let $S = \Spec(R)$ for $R$ an excellent normal strictly henselian local ring, and $f: X \to S$ a finitely presented flat morphism whose geometric fibers are normal surfaces with at most rational double points. Assume the relative singular locus of $f$ is finite over $S$. Then $f$ admits a simultaneous resolution after a finite basechange $\psi: T \to S$, where $T$ is the spectrum of a normal henselian local ring and $\psi$ can be chosen to be \'etale over the locus of $S$ where $f$ is smooth. In other words we have a commutative diagram
\begin{center}
    \begin{tikzcd}
        \widetilde{X} \arrow[dr, "\widetilde{f}"] \arrow[r, "\pi_1"] & X\times_ST\arrow[r, "\pi_2"]\arrow[d]  \arrow[dr, phantom, "\square"] & X \arrow[d, "f"] \\ & T \arrow[r, "\psi"] & S
    \end{tikzcd}
\end{center}
where $\widetilde{f}$ is smooth, $\pi_1$ is proper birational and $\pi_2$ is finite.
 
\end{Thm}


We note that the strictly henselian formulation of Theorem \ref{thm: simultaneous rdp resolution} gives a local statement for families over a general excellent noetherian normal base scheme $S$: if $f\colon X \to S$ is a family as in Theorem \ref{thm: simultaneous rdp resolution}, then fixing a geometric point $s \to S$ and taking $S' = \Spec\left(\Oring_{S,s}^{\textrm{sh}}\right), X'= X\times_S S'$, we may apply Theorem \ref{thm: simultaneous rdp resolution} to the induced family $X' \to S'$ and get a simultaneous resolution after a finite basechange on $S'$; this is the setting treated in \cite{artin1974algebraic}.

\begin{Def} In the notation of Theorem \ref{thm: simultaneous rdp resolution}, let $\Sigma := \textrm{Sing}(X/S)$ be the relative singular locus of $f$. The \emph{discriminant locus of $f$} is $\Delta := f(\Sigma)_{\textrm{red}}$, a closed subscheme of $S$ equipped with the reduced induced subscheme structure.
    
\end{Def}

\begin{Lemma}
    Suppose $R$ is a regular local ring and $f: X \to S = \Spec(R)$ is a flat finitely presented map with geometrically normal surfaces as fibers, so that the generic fiber is smooth, the relative singular locus $\Sigma$ of $f$ is finite over $S$ and, \'etale-locally near $\Sigma$, $X$ is a hypersurface in a smooth scheme over $S$ of relative dimension 3. Then the discriminant locus $\Delta$ is an effective Cartier divisor.
\end{Lemma}

\begin{proof} We may assume that locally $X = V(f)$ inside $Y = \Spec(R[x,y,z]_{\mathfrak{q}}^{\textrm{sh}})$ where $\mathfrak{q} = (\mathfrak{m}_R,x,y,z)$, so that $\Sigma = \{f = f_x=f_y=f_z=0\}$. Let $\eta$ be the generic point of an irreducible component of $\Sigma$ with $\mathfrak{p} = f(\eta)$, which is not the zero ideal as the generic fiber is smooth. Note $\eta$ is a closed point of fiber $Y_{\mathfrak{p}}$ as $\Sigma$ is finite over $S$, and we have $\dim\Oring_{Y, \eta} = \textrm{ht}(\mathfrak{p})+3 \leq 4$ by Krull's height theorem, so $\textrm{ht}(\mathfrak{p})=1$ and $\Delta$ is pure of codimension 1, hence a Cartier divisor as $R$ is regular.\end{proof}

\begin{ex}\label{example-of-simult-resolution-for-ODP} Let $S = \Spec(W(k))$ have residue characteristic $p>2$ and consider maps
\[f_1: X_1 = V(xy-z^2+p^2) \to S, \;\; f_2: X_2 = V(xy-z^2+p) \to S\]
Both are flat families with an $A_1$-singularity $\{xy-z^2=0\}$ as their special fiber. Let
\[Y^- = \textrm{Bl}_{(x,z-p)}(X_1), \;\; Y^+=\textrm{Bl}_{(x,z+p)}(X_1)\]
then $Y^{\pm}$ are distinct simultaneous resolutions of $X_1$: they are smooth $S$-schemes, over the generic fiber they are isomorphic to $X_1$ and over the special fiber they are both the minimal resolution of $(X_1)_k$. $Y^-,Y+$ are related by a flop $Y^+ \dashrightarrow Y^-$, obtained by blowing up the exceptional $\mathbb{P}_k^1$ of $Y^+$ to an exceptional divisor $\mathbb{P}_k^1 \times \mathbb{P}_k^1$ and then contracting the other ruling (this is the \emph{arithmetic} Atiyah flop).

Family $f_2$ does not admit a simultaneous resolution over $S$. If $p>2$ then pullback along tamely ramified cover $T = \Spec(W(k)[p^{1/2}]) \to S$ yields a family over $T$ of the same type as $f_1$, which now has two simultaneous resolutions.

When $p=2$ we can take the wild extension $W(k)[q]/(q^2-2)$ to simultaneously resolve $f_2$. $Y^{\pm}$ are still related by a crepant flop, but it is not the Atiyah flop; blowing up the exceptional curves now yields Hirzebruch surface $\mathbb{F}_2$ instead of $\mathbb{P}_k^1 \times \mathbb{P}_k^1$.
    
\end{ex}

\begin{Rem}\label{artin-example-of-non-scheme-resolution} In general, one needs to consider simultaneous resolutions in the category of algebraic spaces due to global obstructions. Artin discusses an example of a K3 surface degenerating to a nodal quartic in \cite[\S1]{artin1974algebraic}: \'etale-locally at the node, we obtain two scheme-theoretic small resolutions after a base-change of degree at most $2$ as in Example \ref{example-of-simult-resolution-for-ODP}, but the global simultaneous resolutions are algebraic spaces that are \emph{not} schemes (cf.\ Remark \ref{remark on algebraic space for good reduction} too).
\end{Rem}

\subsection{Formal deformations of singularities.}\label{subsection-forma-deformations-of-sings} In this section we describe versal and miniversal deformations of isolated 2-dimensional hypersurface singularities whose completed local rings are
\begin{equation}\label{eq: presentation for rdp singularity}R = \frac{k[\![x,y,z]\!]}{f(x,y,z)}, \;\; k = \overline{k}\end{equation}
with the goal of characterizing miniversal deformations of rational double points in mixed characteristic.

\begin{Def}\label{def: formal local deformation} Fix a complete discrete valuation ring $\Lambda$ with residue field $k$\footnote{In practice $\Lambda$ will be $W(k)$ or a ramified extension of it.} and let $\textrm{Art}_{\Lambda,k} \subset \textrm{Comp}_{\Lambda,k}$ denote respectively the categories of Artinian local $\Lambda$-algebras and complete noetherian local $\Lambda$-algebras of residue field $k$. 

\begin{enumerate}
    \item[(i)] For $A \in \textrm{Comp}_{\Lambda,k}$, a \emph{formal local deformation of $R$} over $A$ is a complete flat local noetherian $A$-algebra $B$ with an identification $B/\mathfrak{m}_AB \simeq R$. The set of isomorphism classes of such $B$ is denoted $\textrm{Def}_R(A)$.
    \item[(ii)] For $A \in \textrm{Comp}_{\Lambda,k}$, $B \in \textrm{Def}_R(A)$, an \emph{algebraic realization} of $B$ is a flat morphism $f: X \to \Spec(A)$ of finite type such that for some $x \in X_k$ we have an identification of $A$-algebras $\widehat{\Oring}_{X,x} \simeq B$.
\end{enumerate}
    
\end{Def}

When $A \in \textrm{Art}_{\Lambda,k}$, we recover the usual notion of infinitesimal deformations of $R$, while for general $A$ we get a compatible system of $A_n$-algebras $B_n$ so that $B_n$ is an infinitesimal deformation of $R$ over $A_n$ and $B = \varprojlim B_n$; here $A_n = A/\mathfrak{m}_A^{n+1}$ and $B_n = B/\mathfrak{m}_A^{n+1}B$.

\begin{Lemma}\label{lemma: deformation of hypersurface sing}
    Every formal local deformation $B$ of $R$ over $A \in \textrm{Comp}_{\Lambda,k}$ has the form
    \[B \simeq \frac{A[\![x,y,z]\!]}{F(x,y,z)}, \;\;\;F \equiv f \mod \mathfrak{m}_A\]
\end{Lemma}
\begin{proof}
    Lifting coordinates $x,y,z$ yields a surjection $A[\![x,y,z]\!] \twoheadrightarrow B$, with kernel $I$. Then since $B$ is a flat $A$-algebra we have an exact sequence
    \[ 0 \to I/\mathfrak{m}_AI \to k[\![x,y,z]\!] \to R \to 0\]
    after tensoring by $\otimes_Ak$. Let $F \in I$ so that $\overline{F} \in I/\mathfrak{m}_AI$ is $f$, then by Nakayama's lemma we have $I=(F)$.
\end{proof}

\begin{Def}\label{def: deformations over dual numbers} Put $D=k[\varepsilon]/\varepsilon^2$. The \emph{tangent space of deformation functor} $\textrm{Def}_R$ is\[\textrm{Def}_R(D) = \left\{
(B,i)\ \middle|\
\begin{array}{l}
B \text{ flat complete Noetherian local } D\text{-algebra},\\
i : B/\varepsilon B \xrightarrow{\sim} R
\end{array}
\right\}\big/\sim\]
and it parametrizes 1st-order deformations of $R$. Any $B \in \textrm{Def}_R(D)$ has the form
\[B = B(g) \simeq \frac{D[\![x,y,z]\!]}{f+\varepsilon g}, \;\; g \in k[\![x,y,z]\!]\]
Thus an equivalent definition of the tangent space of $\textrm{Def}_R$ is \[\textrm{Def}_R(D)\simeq T^1(R)\coloneqq k[\![x,y,z]\!]/(f,f_x,f_y,f_z), \;\;(B(g),i) \mapsto g\]
The finite dimensional $k$-algebra is called the \emph{Tjurina algebra} of singularity $R$.
\end{Def}

    

\begin{Rem} The Tjurina algebra should be distinguished from the \emph{Milnor algebra} $k[\![x,y,z]\!]/(f_x,f_y,f_z)$ associated to the presentation of the singularity $R$. When $f(x,y,z)$ is weighted homogeneous of degree prime to $p$, one has that $f \in (f_x, f_y, f_z)$ and so the two algebras coincide, but they differ e.g.\ for $A_{p-1}$ singularities.  
\end{Rem}

Now let $B \in \textrm{Def}_R(A)$ be a formal local deformation of $R$ over $A$ and set
\[ T_{\Lambda}(A) = \textrm{Hom}_k\left(\frac{\mathfrak{m}_A}{\mathfrak{m}_A^2+\mathfrak{m}_{\Lambda}A}, k\right) \simeq \Hom_{\Lambda}^{\textrm{loc}}\left(A, D\right)\]
to be the (Zariski) tangent space of the `mod $\mathfrak{m}_{\Lambda}$' fiber $\Spec(A/\mathfrak{m}_{\Lambda}A)$ of $\Spec(A)$ (cf. \cite[\S B.1]{talpo2010deformation}).

\begin{Def}[\cite{talpo2010deformation}, \S6]\label{def: kodaira--spencer} For a formal local deformation $B \in \textrm{Def}_R(A)$ of $R$, its \emph{Kodaira--Spencer} map is
\[\textrm{KS}_B : T_{\Lambda}(A) \to \textrm{Def}_R(D),\]
\[(v: A \to D) \longmapsto [B \widehat{\otimes}_A D]\]
where $D$ is viewed as an $A$-algebra through vector $v$.
\end{Def}

By Lemma \ref{lemma: deformation of hypersurface sing}, we can write $B = A[\![x,y,z]\!]/(F)$ and pick $v\in T_{\Lambda}A$ sending $v(F) = f + \varepsilon g_v$. Then $\textrm{KS}_B(v) = [B(g_v)]$ in the notation of Definition \ref{def: deformations over dual numbers}. Intuitively, the Kodaira--Spencer map records infinitesimal changes of the deforming singularity along a tangent vector on the parameter space $A$.

\begin{Def}[\cite{talpo2010deformation}, \S6]\label{def: versal and miniversal deformations} Given $B \in \textrm{Def}_R(A)$ a formal local deformation of $R$, and a homomorphism $\phi: A \to C$, set $B_{\phi} = B \widehat{\otimes}_A C$.

\begin{enumerate}
    \item[(i)] We say $B$ is a \emph{versal} deformation if the following lifting property holds: given a surjection $\widetilde{C} \twoheadrightarrow C$ in $\textrm{Art}_{\Lambda,k}$ and $\widetilde{B} \in \textrm{Def}_R(\widetilde{C})$ so that $B_{\phi} \simeq \widetilde{B} \widehat{\otimes}_{\widetilde{C}}C$, there exists a homomorphism $\psi: A \to \widetilde{C}$ lifting $\phi$ and an isomorphism $B_{\psi} \simeq \widetilde{B}$ compatible with $B_{\phi} \simeq \widetilde{B} \widehat{\otimes}_{\widetilde{C}}C$.
    \item[(ii)] We say $B$ is a \emph{miniversal} deformation if it is versal and the Kodaira--Spencer map $\textrm{KS}_B$ is an isomorphism of $k$-vector spaces.
\end{enumerate}
    
\end{Def}
Note that in (i) we do not require lift $\psi$ to be unique, as that would make $B$ a universal deformation. Note also that one can take the surjections in (i) to be in $\textrm{Comp}_{\Lambda,k}$ by \cite[Prop.\ 6.21]{talpo2010deformation}, so any formal local deformation over a complete base $C$ is the pullback of versal deformation $B$ along a homomorphism $A \to C$.

\begin{Prop}[\cite{talpo2010deformation}, Prop. 6.36]\label{miniversal-iff-power-series-and-kodaira-spencer-isom} Suppose $A = \Lambda[\![t_1, \cdots, t_r]\!]$ is a power series ring over $\Lambda$. If $B \in \textrm{\emph{Def}}_R(A)$ is a formal local deformation of $R$ over $A$ and its Kodaira--Spencer map $\textrm{KS}_B$ is an isomorphism, then $B$ is a miniversal deformation. 
\end{Prop}

\begin{Prop}[\cite{talpo2010deformation}, Prop. 6.44]\label{prop: explicit miniversal formal rdp deformation} Let $r = \dim_kT^1(R)$ and $A = \Lambda[\![t_1, \cdots, t_r]\!]$. Choose $g_1, \cdots, g_r \in k[\![x,y,z]\!]$ whose classes in $T^1(R)$ form a $k$-basis, and choose lifts $\widetilde{f}, \widetilde{g}_i \in \Lambda[\![x,y,z]\!]$ of $f$ and $g_i$. Then
\[B = \frac{A[\![x,y,z]\!]}{F(x,y,z)}, \;\; F = \widetilde{f} + \sum_{i=1}^rt_i\widetilde{g}_i\]
is a miniversal formal local deformation of $R$ over base $A$.
\end{Prop}
\begin{proof} $B$ is a flat $A$-algebra by the local flatness criterion. We can identify $T_{\Lambda}A \simeq k^r$ with generators $v_i$ dual to (the classes of) $t_i$ in the sense that $v_i: A \to D = k \oplus k\varepsilon$ maps $t_j \mapsto \delta_{ij}\varepsilon$. Then, using the notation of Definition \ref{def: kodaira--spencer} we have
\[B_{v_i} \simeq \frac{D[\![x,y,z]\!]}{f+\varepsilon g_i}\]
so that $\textrm{KS}_B(v_i) = [B(g_i)] = [g_i] \in T^1(R)$ upon identifying $\textrm{Def}_R(D) \simeq T^1(R)$. Since $\{[g_i]\}$ form a basis of $T^1(R)$, we have that $\textrm{KS}_B$ is an isomorphism, whence $B$ is miniversal by Proposition \ref{miniversal-iff-power-series-and-kodaira-spencer-isom}.\end{proof}

\begin{ex}\label{example: miniversal defs for An sings} Let $f(x,y,z) = xy-z^n$ so that $R$ is the completed local ring of an $A_{n-1}$-singularity. We have \vspace{-1em}
\[T^1(R) \simeq \begin{cases}k[\![z]\!]/z^{n-1}, & p \nmid n \\ k[\![z]\!]/z^n, & p \mid n \end{cases}\]
so if $p$ is very good for $A_{n-1}$ we have $r=n-1$ parameters and choosing $g_i = z^{n-i-1}$ yields miniversal deformation\[B = \frac{\Lambda[\![x,y,z,t_1, \cdots, t_{n-1}]\!]}{xy - z^n + t_1z^{n-2} + \cdots + t_{n-2}z + t_{n-1}}\]
over base $\Lambda[\![t_1, \cdots, t_{n-1}]\!]$, while if $p \!\mid\! n$ a miniversal deformation is
\[B = \frac{\Lambda[\![x,y,z,t_1, \cdots, t_n]\!]}{xy - z^n + t_1z^{n-1}+t_2z^{n-2} + \cdots + t_{n-1}z + t_n}\]\end{ex}

Let $X/\Lambda$ be a surface degeneration so that the special fiber $X_k$ has a unique RDP $x \in X(k)$ of type
\[\widehat{\Oring}_{X_k,x} \simeq R = \frac{k[\![x,y,z]\!]}{f(x,y,z)}\]
as in (\ref{eq: presentation for rdp singularity}). We can view $\widehat{\Oring}_{X,x}$ as a miniversal deformation pullback as follows.
\begin{Cor}\label{cor: 1-parameter deformation of rdp} Let $g_1, \cdots, g_r$ be polynomial representatives of a $k$-basis of $T^1(R)$ and choose polynomial lifts $\widetilde{f}, \widetilde{g_i} \in \Lambda[x,y,z]$ of $f, g_i$. Then there exist $w_i \in \mathfrak{m}_{\Lambda}$ so that we have an isomorphism of $\Lambda$-algebras
\begin{equation}\label{eq: 1-parameter deformation of rdp}\widehat{\Oring}_{X,x} \simeq \frac{\Lambda[\![x,y,z]\!]}{\widetilde{f} + \sum_{i=1}^r w_i\widetilde{g_i}}\end{equation}
\end{Cor}
\begin{proof} Note that $\widehat{\Oring}_{X,x}$ is a flat $\Lambda$-algebra and $\widehat{\Oring}_{X,x}\otimes_{\Lambda}k \simeq R $ so it is a formal local deformation of singularity $R$ over base $\Lambda$. Let $A = \Lambda[\![t_1, \cdots, t_r]\!]$ with $B,F(x,y,z)$ as in Proposition \ref{prop: explicit miniversal formal rdp deformation}, then $B$ is miniversal and so there is a map $\phi: A \to \Lambda$ inducing isomorphism $B \widehat{\otimes}_A \Lambda \simeq \widehat{\Oring}_{X,x}$. Setting $w_i = \phi(t_i)$ gives presentation (\ref{eq: 1-parameter deformation of rdp}).\end{proof}

We are able to have both $g_i$ and $\widetilde{g_i}$ polynomials due to $T^1(R)$ being finite-dimensional, so that $x^N=y^N=z^N=0$ in $T^1(R)$ for $N \gg 0$. In fact we can describe algebraic deformations of the singularity up to henselization:

\begin{Cor}\label{cor: henselian nbd of deformation} Suppose $\Lambda$ is henselian and $X,R$ are as above. In the notation of \emph{Corollary \ref{cor: 1-parameter deformation of rdp}} there exist $w_i \in \mathfrak{m}_{\Lambda}$ so that we have an isomorphism of henselian local $\Lambda$-algebras
\[\Oring_{X,x}^h \simeq \left(\frac{\Lambda[x,y,z]}{\widetilde{f} + \sum_{i=1}^r w_i\widetilde{g_i}}\right)^h_{(\mathfrak{m}_{\Lambda},x,y,z)}\]
    
\end{Cor}

\begin{proof} Write $R_{\textrm{alg}} = (k[x,y,z]/f)^h_{(x,y,z)}$. By (\ref{eq: 1-parameter deformation of rdp}) and Artin approximation we have $\Oring_{X_k,x}^h \simeq R_{\textrm{alg}}$. Consider the family
\[Y =\Spec\left(\frac{\Lambda[x,y,z,t_1,\cdots,t_r]}{\widetilde{f} + \sum_i t_i\widetilde{g}_i}\right) \to  S\coloneqq \Spec(\Lambda[t_1,\cdots, t_r]) \]
with points $s = (\mathfrak{m}_{\Lambda}, t_1, \cdots t_r)$ and $o= (\mathfrak{m}_{\Lambda},x,y,z, t_1, \cdots, t_r)$. The family is flat near $o$ and its completed local deformation over $\widehat{\Oring}_{S,s} = \widehat{\Lambda}[\![t_1, \cdots, t_r]\!]$ is miniversal by Proposition \ref{prop: explicit miniversal formal rdp deformation}. By a theorem of Elkik\footnote{Namely, henselian deformations of an isolated singularity over a henselian base are miniversal if and only if their induced formal deformations are miniversal.} (\cite[Expos\'e XV, \S1.1.4]{deligne2006groupes}) $\Spec(\Oring_{Y,o}^h)\to \Spec(\Oring_{S,s}^h)$ is a henselian miniversal deformation of $R_{\textrm{alg}}$ and so by versality we get a local $\Lambda$-homomorphism $\phi: \Oring_{S,s}^h \to \Lambda$. Setting $w_i = \phi(t_i)$ completes the proof.\end{proof}

\subsection{Equivariant deformations of singularities.}\label{section: equivariant deformations of sings} In the rest of the paper it will also be useful to enrich the deformations considered in Section \ref{subsection-forma-deformations-of-sings} with a $\mathbb{G}_m$-equivariant structure, so we make this notion precise here. Once again, we fix a hypersurface singularity
\[ R = \frac{k[\![x,y,z]\!]}{f(x,y,z)}\]

A $\mathbb{G}_m$-action on $A \in \textrm{Comp}_{\Lambda,k}$ is a system of compatible algebraic $\mathbb{G}_m$-actions on $\Spec(A/\mathfrak{m}_A^n)$ for each $n\geq 1$, with trivial induced action on the residue field.

\begin{Def}\label{def: equivariant deformations} For $A \in \textrm{Comp}_{\Lambda,k}$ equipped with a $\mathbb{G}_m$-action, an equivariant deformation of $R$ over $A$ is a deformation $B \in \textrm{Alg}_A$ in the sense of Definition \ref{def: formal local deformation} equipped with a $\mathbb{G}_m$-action such that the structure map $A \to B$ and the fixed isomorphism $B/\mathfrak{m}_AB \overset{\sim}{\to} R$ are $\mathbb{G}_m$-equivariant.  

There are equivariant notions of versality and miniversality which follow from Definitions \ref{def: kodaira--spencer} and \ref{def: versal and miniversal deformations} by declaring the lifting and the associated Kodaira--Spencer map to be compatible with the equipped $\mathbb{G}_m$-actions as above.
\end{Def}

In this setting, $f(x,y,z)$ is quasihomogeneous with weights $(w_x, w_y, w_z)$ and the natural $\mathbb{G}_m$-action on $R$ is $(x,y,z) \mapsto (s^{w_x}x, s^{w_y}y, s^{w_z}z)$.

\begin{Prop}[\cite{Slodowy1980}, \S2.5 (ii)]\label{prop: equivariantly miniversal rdp deformations} Suppose $f(x,y,z)$ is quasihomogeneous with weights $(w_x, w_y, w_z)$ and degree $d$. Then in the setting of \emph{Proposition \ref{prop: explicit miniversal formal rdp deformation}} with $A = \Lambda[\![t_1, \cdots, t_r]\!]$, $\{g_i\}$ a homogeneous basis of $T^1(R)$ and $\widetilde{f}, \widetilde{g}_i$ homogeneous lifts to $\Lambda$, the formal local deformation
\[B = \frac{A[\![x,y,z]\!]}{F(x,y,z)}, \;\; F = \widetilde{f} + \sum_{i=1}^rt_i\widetilde{g}_i\]
is equivariantly miniversal with respect to the $\mathbb{G}_m$-action \[s \cdot (x,y,z, t_1, \cdots, t_r) = (s^{w_x}x, s^{w_y}y, s^{w_z}z, s^{d_1}t_1, \cdots, s^{d_r}t_r), \;\;\; d_i = d- \deg(g_i)\]
\end{Prop}
Note that taking weight components for the $\mathbb{G}_m$-action is an exact functor, so the argument of Slodowy here extends from $k$ to $\Lambda$. For a description of the (doubled) weights $w_x, w_y, w_z$ and $d_i$ in the case of $f(x,y,z)$ parametrizing rational double points, see (\cite[\S7.4 Prop.\! 2]{Slodowy1980}).

We also briefly discuss the problem of deforming minimal resolutions of a singularity. In what follows, let $\pi: X \to \Spec(R)$ be the minimal resolution of the isolated hypersurface singularity $R$ with reduced exceptional divisor $E$, and set $\mathcal{X} = \widehat{X}_E$ for the formal completion of $X$ along $E$; it is a smooth formal scheme with underlying space $E$.

\begin{Def}\label{def: deformations of resolutions} Let $A \in \textrm{Art}_{\Lambda, k}$. A \emph{deformation of $\mathcal{X}$ over base $A$} is a flat morphism $\mathcal{X}_A \to \Spec(A)$ from a noetherian formal scheme, along with an isomorphism $\mathcal{X}_A \times_A \Spec(k) \simeq \mathcal{X}$ of formal schemes. We put \[\textrm{Def}_{\mathcal{X}}(A) = \{\textrm{deformations} \; \mathcal{X}_A \to \Spec(A)\}/\sim\]

\end{Def}

Similarly, if $A \in \textrm{Comp}_{\Lambda,k}$, then a deformation of $\mathcal{X}$ over $A$ consists of a compatible system of deformations of $\mathcal{X}$ over $A_n:= A/\mathfrak{m}^nA$.

Definitions \ref{def: deformations over dual numbers}, \ref{def: kodaira--spencer} and \ref{def: versal and miniversal deformations} extend naturally to the setting of deformations of formal schemes. Namely for $D = k[\varepsilon]/\varepsilon^2$ and $\mathcal{X}_A \in \textrm{Def}_{\mathcal{X}}(A)$ we have a Kodaira--Spencer map

\[\textrm{KS}_{\mathcal{X}_A}: \textrm{T}_{\Lambda}(A)\to \textrm{Def}_{\mathcal{X}}(D), \;\;\; v \mapsto [\mathcal{X}_v]\]
where $\mathcal{X}_v$ is the pullback of $\mathcal{X}_A$ along $A \to D$. Then versality of a formal deformation is defined in terms of the usual lifting criterion, and miniversality is versality along with the Kodaira--Spencer map being an isomorphism.

\begin{Prop}\label{prop: tangent obstruction theory for formal scheme deformations} When $p$ is very good for the RDP singularity $R$, we have
\begin{equation}\label{eq: cohom isom rdp resol}\coh^1(\mathcal{X}, \mathcal{T}_{\mathcal{X}}) \simeq \bigoplus_{i \leq r} \coh^1(E_i, \mathcal{N}_{E_i/X}), \;\;\; \coh^2(\mathcal{X}, \mathcal{T}_{\mathcal{X}}) \simeq 0\end{equation}
    
\end{Prop}

\begin{proof} By \cite[Prop.\! 2.3]{hirokado2019further}, for $j_i: E_i \hookrightarrow \mathcal{X}$ we have an exact sequence
\begin{equation}\label{eq:cohom les for rdp resolution} 0 \to T_{\mathcal{X}}(-\log E) \to T_{\mathcal{X}} \to \bigoplus_{i \leq r} (j_i)_*\mathcal{N}_{E_i/\mathcal{X}} \to 0\end{equation}
with $\mathcal{N}_{E_i/\mathcal{X}} \simeq \mathcal{O}_{\mathbb{P}^1}(-2)$ as $E_i$ are rational  $(-2)$-curves (see Definition \ref{def: rational sings and rdps}). Moreover $\coh^2(\mathcal{X}, T_\mathcal{X}(-\log E)) = \coh^2(\mathcal{X}, T_\mathcal{X}) =0$ by Grothendieck vanishing. Thus taking cohomology in (\ref{eq:cohom les for rdp resolution}) gives $r \leq \dim \coh^1(\mathcal{X}, T_X)$ and \cite[Prop.\! 2.2]{hirokado2019further} gives $\dim \coh^1(\mathcal{X}, T_\mathcal{X}) \leq \dim T^1(R) = r$ where $T^1(R)$ is the Tjurina algebra of the singularity (see Definition \ref{def: deformations over dual numbers}). We thus get both isomorphisms in (\ref{eq: cohom isom rdp resol}).\end{proof}

As a result we have that a deformation of $\mathcal{X}$ over $\Lambda[\![s_1, \cdots, s_r]\!]$ is miniversal if and only if its associated Kodaira--Spencer map is an isomorphism.

\subsection{Artin components.}\label{section: artin components} We briefly discuss special components of the deformation space when $R$ is the completed local ring of an arbitrary rational surface singularity.

Let $(U,x)$ be an affine open neighborhood of the singularity $R$ in that $R = \widehat{\Oring}_{U,x}$. Isolated singularities admit formal miniversal deformations over base $\Lambda$, and in particular such deformations have an algebraic versal representative (\cite[Part IV, Thms.\ 7, 8]{elkik1973solutions}). Thus we may choose a flat family $f: \mathcal{U} \to S$ of finite-type $\Lambda$-schemes, with relative singular locus finite over $S$, such that $\mathcal{U}_s \simeq U$, $A\coloneqq\widehat{\Oring}_{S,s} \in \textrm{Comp}_{\Lambda,k}$ and $B\coloneqq \widehat{\Oring}_{\mathcal{U},x} \in \textrm{Def}_R(A)$ is a versal deformation of $R$ (see Definitions \ref{def: formal local deformation} and \ref{def: versal and miniversal deformations} (i)). After pullback to $\Spec(A)$, morphism $f$ is an algebraic realization of $B$ in the sense of Definition \ref{def: formal local deformation} (ii).

\begin{Def}[\cite{artin1974algebraic}, \S1] The simultaneous resolution functor 
\[\textrm{Res}_{\mathcal{U}/S}: (\textrm{Sch}/S)^{\textrm{op}} \to \textrm{Set}, \;\;\;\textrm{Res}_{\mathcal{U}/S}(T) = \{\pi: \mathcal{X} \to U\times_S T\}/\simeq\] parametrizes simultaneous resolutions of the family $\mathcal{U} \to S$ in the sense that $\mathcal{X}$ is an algebraic space smooth over $T$ and $\pi$ is a simultaneous resolution (Definition \ref{definition-of-simultaneous-resolution}).
\end{Def}

By (\cite[Thm.\! 1]{artin1974algebraic}), $\textrm{Res}_{\mathcal{U}/S}$ is representable by an algebraic space $\mathcal{R}$ over $S$. Let $a \in \mathcal{R}(k)$ correspond to the minimal resolution $\pi: \widetilde{U} \to U$ of the rational singularity $R$ and set $C\coloneqq \widehat{\Oring}_{\mathcal{R},a}$.

\begin{Def}\label{def: artin component} The \emph{Artin component} of the versal deformation $\mathcal{U} \to S$ of $R$ is the scheme-theoretic image of $\Spf(C) \to \Spf(A) = \Spf(\widehat{\Oring}_{S,s})$. Equivalently it is $\Spf(A/I)$ where $I = \ker(A \to C)$ is induced from the structure map.  
\end{Def}

\begin{Prop}\label{prop: properties of the artin component} Assume $R$ is a rational singularity and let $f: \mathcal{U} \to S$, $A$ and $C$ be as above. Then $C$ is a finite $A$-algebra and a formally smooth $\Lambda$-algebra, the Artin component is an irreducible component of $\emph{\Spf(A)}$ and the induced family $f\lvert_{A/I}$ over  $\Spec(A/I)$ admits a simultaneous resolution after finite base-change $\Spec(C) \to \Spec(A/I)$.
    
\end{Prop}

\begin{proof} By (\cite[Thm.\ 3]{artin1974algebraic}) $\Spf(C) \to \Spf(A)$ is a finite map onto an irreducible component of $\Spf(A)$ i.e.\ its image is $\Spf(A/I)$ with $I$ a minimal prime. Since $\coh^2(\widetilde{U}, T_{\widetilde{U}}) = 0$ for the minimal resolution of the rational singularity of $U$, deformations of $\widetilde{U}$ over $\Lambda$ are unobstructed and so $C$ is a formally smooth $\Lambda$-algebra by (\cite[Lem.\ 3.3]{artin1974algebraic}). By construction $\Spec(C) \to \mathcal{R}$, so the representing property of $\mathcal{R}$ yields a simultaneous resolution of the base-changed family $\mathcal{U}\times_S \Spec(C)$.\end{proof}

\begin{Rem} When $R$ is a rational double point (i.e.\ a rational hypersurface singularity), its miniversal base $A$ is a domain by Proposition \ref{prop: explicit miniversal formal rdp deformation} so minimal primes are zero and the Artin component equals $\Spf(A)$. We will see that versal deformations of rational double points admit a simultaneous resolution after completion and base-change, without shrinking the base (Proposition \ref{prop: slice simultaneous rdp resolution}).
\end{Rem}

\section{The geometry of the Grothendieck alteration}\label{chevalley-section}

We now collect the necessary Lie-theoretic prerequisites in order to define miniversal deformations of rational double points in terms of the adjoint quotient of Lie algebras. The main goal is to define and study the classical Grothendieck--Springer resolution over mixed characteristic integral bases (Section \ref{relative-grothendieck-alteration}). To this end, we first extend the relevant notions of nilpotent, semisimple and (sub)regular elements to the mixed characteristic setting and construct an integral version of the adjoint quotient in preparation for Theorem \ref{thm: all the properties of simultaneous resolution diagrams}.

\subsection{Notations and conventions.}\label{section: notations and conventions for lie stuff} 

We work over a mixed characteristic complete discrete valuation ring $\Lambda$ with algebraically closed residue field $k$; as a trait, its closed point is denoted $s$ and its generic point is denoted $\eta$.

Fix a split, simply connected, simple group $G/\Lambda$ along with a choice of split maximal torus and Borel $T \subset B \subset G$, unipotent radical $U = \mathcal{R}_u(B)$ so that $B = U \rtimes T$, and set
\[ \g = \textrm{Lie}(G), \; \h = \textrm{Lie}(T), \borel = \textrm{Lie}(B), \; \mathfrak{n} = \textrm{Lie}(U)\]
Write $r:=\textrm{rk}(G) = \textrm{rk}_{\Lambda}(\h)$ for the rank of Lie algebra $\g$\footnote{Not to be confused with $\textrm{rk}_{\Lambda}(\g)$, which is the dimension of $G/\Lambda$.}. Identification $B = U \rtimes T$ yields a decomposition of $\Lambda$-modules
\begin{equation}\label{eq: borel quotient is cartan}\borel = \h \oplus \mathfrak{n}, \;\; \borel/\mathfrak{n} \simeq \h\end{equation}

We will sometimes identify $\mathfrak{g}$ with its affine scheme $\Spec(\textrm{Sym}_{\Lambda}(\g^{\vee}))$, so that for $R \in \textrm{Alg}_{\Lambda}$ its $R$-points are $R\otimes_{\Lambda} \g$ and its coordinate ring is $\Lambda[\g]$; similar notations will be used for the Lie subalgebras of $\g$.

Write $W = N_G(T)/T$ for the Weyl group of $(G,T)$; it is a finite \'etale (in fact constant) group scheme. It comes naturally equipped with an adjoint action on $\h$ via $\textrm{Ad}: G \to \textrm{Aut}_{\Lambda}(\g)$; a priori it is an action of $N_G(T)$ but it descends to $W$ as $T$ acts trivially on $\h$.

Let $\Phi = \Phi(G,T)$ be the associated root lattice of $(G,T)$. We make the following note once and for all:

\begin{Ass}\label{assumptions on G g etc} $G$ is of $ADE$ type so that $\Phi$ is a semisimple simply-laced root lattice, and residue characteristic $p = \textrm{char}(k)$ is \emph{very good} for $\Phi$ (see Definition \ref{definition-of-good-primes}).  
\end{Ass}

Setting $\Phi^{\vee} = \Hom_{\mathbb{Z}}(\Phi,\mathbb{Z})$ we may identify it with the character lattice $X^*(T)$ via Assumption \ref{assumptions on G g etc}, and moreover $\Phi\otimes\mathbb{Q}_{\ell} \simeq \Phi^{\vee}\otimes\mathbb{Q}_{\ell}$ in good characteristic; the restrictions on $p$ ensure $G/\Lambda$ is root-smooth (\cite[\S4.1.1]{bouthier2019torsors}) and $p$ is a non-torsion prime for the root datum of $(G,T)$ (\cite[\S4.1.12]{bouthier2019torsors}). We remark that $p$ good is not a sufficient restriction as e.g. $G = \textrm{SL}_2$ is not root-smooth when $p=2$.

\subsection{Nilpotent and regular loci.}\label{section: nilpotent and reg loci} 
We retain the assumptions of Section \ref{section: notations and conventions for lie stuff}. Over the generic point $\eta \to \Spec(\Lambda)$ we have the usual reduced nilpotent variety $\mathcal{N}_{\eta} \subset \g_{\eta}$, where $\g_{\eta}$ is a characteristic zero Lie $k(\eta)$-algebra.

\begin{Def} The nilpotent scheme $\mathcal{N} \subset \g$ is the schematic closure $\overline{\mathcal{N}_{\eta}}^{\textrm{\;sch}}$ of $\mathcal{N}_{\eta}$ inside $\g$. If $R = \textrm{Sym}_{\Lambda}(\g^{\vee})$ and $I\subset R\otimes_{\Lambda}K$ is the ideal of $\mathcal{N}_{\eta}$, then $\mathcal{N} = \Spec(R/I \cap R)$.
\end{Def}

The nilpotent scheme is $\Lambda$-flat and, under our very good assumptions on $p$, its special fiber $\mathcal{N}_s$ is the usual normal nilpotent variety of $\g_s$ (\cite[Thm.\ 4.12]{cotner2022springer}). Note for $x \in \g(\Lambda)$, we have $x \in \mathcal{N}(\Lambda)$ if and only if $x_{\overline{\eta}} \in \mathcal{N}(\overline{k(\eta)})$; one direction follows from $\Lambda \hookrightarrow \overline{k(\eta)}$ while the converse follows by base-change.

For the next set of definitions, we recall that for a geometric point $\overline{t} \to \Spec(\Lambda)$ (i.e. the closed point $s$ and the geometric generic point $\overline{\eta}$), the centralizer of an element $x \in \g_{\overline{t}}(k(\overline{t}))$ under the adjoint action is denoted $C_{G_{\overline{t}}}(x)$.

\begin{Def}\label{def: reg subreg and regsemisimp}
    Let $r$ be the rank of Lie algebra $\g/\Lambda$. A section $x \in \g(\Lambda)$ will be called:
    \begin{itemize}
        \item[(i)] \emph{regular} if $\dim C_{G_{\overline{t}}}(x_{\overline{t}}) = r$ for all $\overline{t} \to \Spec(\Lambda)$ (\cite[\S4.2.1]{bouthier2019torsors}).
        \item[(ii)] \emph{subregular nilpotent}\footnote{Of course one can omit $x \in \mathcal{N}(\Lambda)$ and define subregular sections in general, but we will only work with subregular nilpotent elements. Such subregular sections always exist by \cite[\S5.4, pp.\ 159-160]{McN08}.} if $\dim C_{G_{\overline{t}}}(x_{\overline{t}}) = r+2$ for all $\overline{t} \to \Spec(\Lambda)$ and $x \in \mathcal{N}(\Lambda)$ (\cite[\S5.4]{Slodowy1980}).
        \item[(iii)] \emph{regular semisimple} if $C_{G_{\overline{t}}}(x_{\overline{t}})$ is a maximal torus in $G_{\overline{t}}$ for all $\overline{t} \to \Spec(\Lambda)$ (\cite[Lemma 4.2.7]{bouthier2019torsors}).
    \end{itemize}
\end{Def}

Since the residue characteristic is very good for $\Phi(G,T)$ (in particular $G/\Lambda$ is root-smooth), the loci of regular semisimple and regular sections in $\g(\Lambda)$ form respectively the $\Lambda$-points of $G$-stable, $\Spec(\Lambda)$-fiberwise dense open subschemes $\g^{\textrm{rs}} \subset \g^{\textrm{reg}} \subset \g$ whose formation commutes with base-change (cf. \cite[\S4.1.5, \S4.2.1-2]{bouthier2019torsors}). We can accordingly say $x \in \g(\Lambda)$ is a regular semisimple (resp. regular) section if and only if $x \in \g^{\textrm{rs}}(\Lambda)$ (resp. $x \in \g^{\textrm{reg}}(\Lambda)$).

We also define locus $\h^{\textrm{rs}} := \h \cap \g^{\textrm{rs}}$. For a root $\alpha: T \to \mathbb{G}_m$ we have $d\alpha: \h \to \Lambda$ i.e. $d\alpha \in \h^{\vee}$. By \cite[\S4.1.5]{bouthier2019torsors}, one can view $\h^{\textrm{rs}}$ as the complement of root hyperplanes in $\h$, i.e. ranging over all roots $\alpha \in \Phi$ we have\[\h^{\textrm{rs}} = \h\setminus \cup_{\alpha \;\textrm{root}}V(d\alpha)\]

\begin{ex} As already remarked, for an integral section $x\in \g(\Lambda)$, it suffices to have the generic fiber $x_{\overline{\eta}} \in \mathcal{N}(k(\overline{\eta}))$ in order to get $x \in \mathcal{N}(\Lambda)$ because $\mathcal{N}$ is a closed subscheme. On the other hand, as $\g^{\textrm{rs}}, \g^{\textrm{reg}}$ are open, $x_s$ regular (resp. regular semisimple) implies $x_{\overline{\eta}}$ regular (resp. regular semisimple). The converse implications are generally false, while for subregular nilpotent elements neither implication holds and we require the fiberwise constancy of $\dim C_{G_{\overline{t}}}(x_{\overline{t}})$: For $x = E_{12} + pE_{23} \in \mathfrak{sl}_3(\Lambda)$, $x_s$ is subregular nilpotent but $x_{\overline{\eta}}$ is regular, and for $y= pE_{12}$, $y_{\overline{\eta}}$ is subregular nilpotent but $y_s = 0$ is not subregular.
\end{ex}

More broadly, since $G/\Lambda$ is split, $G_s$ and $G_{\bar{\eta}}$ have the same root datum and so since $p$ is (very) good, the two sets of nilpotent orbits in $\mathcal{N}(k(\overline{\eta})), \mathcal{N}(k)$ (i.e. the orbits under the adjoint $G$-action) have the same classification. Then purity of a section $x \in \mathcal{N}(\Lambda)$ in the sense that $\dim C_{G_{\overline{t}}}(x_{\overline{t}})$ is constant for $\overline{t} \to \Spec(\Lambda)$ ensures $x_{\overline{t}}$ lie in the same type of nilpotent orbit (\cite[Lemma 5.7]{cotner2022centralizers}).

\subsection{The adjoint quotient.}\label{the-adjoint-quotient-subsection} 

We turn to explaining the construction and properties of the adjoint quotient for algebraic groups over $\Lambda$. A classical comprehensive reference for the usual construction over algebraically closed fields (including positive characteristic) is \cite[\S3.10-3.14]{Slodowy1980}. 

The adjoint action of $G$ on $\g$ and the induced adjoint action of $W$ on $\h$ (see Section \ref{section: notations and conventions for lie stuff}) yield categorical quotient maps \[ \chi: \g \longrightarrow \g/\!\!/G:= \Spec\left(\Lambda[\g]^G\right), \;\;\; \psi: \h \longrightarrow \h/\!\!/W:= \Spec\left(\Lambda[\h]^W\right)\]
We call $\chi$ the \emph{adjoint quotient}; $\psi$ is a finite branched cover. Restricting invariant functions along $\h \subset \g$ yields in addition a natural morphism $\rho: \h/\!\!/W \to \g/\!\!/G$.

\begin{Thm}\label{thm: properties of adjoint quotient and branched cover psi} With the above notations and the hypotheses of \emph{Section \ref{section: notations and conventions for lie stuff}}:

\begin{enumerate}
    \item[\emph{(i)}] Map $\rho$ is an isomorphism $\h/\!\!/W \simeq \g/\!\!/G \simeq \mathbb{A}_{\Lambda}^r$, where the latter isomorphism is non-canonical. Moreover formation of both quotients commutes with base-change.
    \item[\emph{(ii)}] The adjoint quotient $\chi$ is faithfully flat with geometrically normal fibers of dimension $\dim(G_s)-r$ and the restriction $\chi^{\textrm{reg}}: \g^{\textrm{reg}} \to \h/\!\!/W$ is smooth surjective.
    \item[\emph{(iii)}] We have $\mathcal{N} \simeq \chi^{-1}(0)$ where $0 \in (\h/\!\!/W)(\Lambda)$ is the zero section.
    \item[\emph{(iv)}] Map $\psi$ is finite flat, locally free of degree $|W|$. Denoting $(\h/\!\!/W)^{\textrm{rs}} = \psi(\h^{\textrm{rs}})$ we have $\chi^{-1}((\h/\!\!/W)^{\textrm{rs}}) = \g^{\textrm{rs}}$ and $\psi^{\textrm{rs}}: \h^{\textrm{rs}} \to (\h/\!\!/W)^{\textrm{rs}}$ is a finite \'etale $W$-torsor.
\end{enumerate}
    
\end{Thm}

\begin{proof} (i) is \cite[Thm.\ 4.1.10, Prop.\ 4.1.14, Rem. 4.1.15]{bouthier2019torsors}, the non-canonical isomorphism holds because $G/\Lambda$ is split. (ii) is \cite[Cor.\ 4.1.18, Prop.\ 4.2.6]{bouthier2019torsors} and normality of the fibers follows from \cite[\S3.10 Thm.\ (v), \S3.14]{Slodowy1980}. For (iii), note that $\chi^{-1}(0)$ is $\Lambda$-flat by (ii) and generic fiber $\chi^{-1}(0)_{\overline{\eta}}$ is $\mathcal{N}_{\eta}$ by construction. As flat closed schemes over a discrete valuation ring are the schematic closure of their generic fiber we have $\mathcal{N} \simeq \chi^{-1}(0)$. (iv) is \cite[Cor.\ 4.1.18, Lem.\ 4.2.12]{bouthier2019torsors} as geometric points in $\g^{\textrm{rs}}$ are conjugate to elements in $\h^{\textrm{rs}}$; the torsor is finite \'etale because $W$ is a finite \'etale $\Lambda$-scheme.\end{proof}

\begin{Rem} Over algebraically closed fields, the adjoint quotient maps $x$ to the conjugacy class of its semisimple part via the Jordan decomposition. This interpretation however does not hold over $\Lambda$ as there need not be a Jordan decomposition specializing to the Jordan decompositions in the fibers.
\end{Rem}

\subsection{Relative Grothendieck--Springer resolutions.}\label{relative-grothendieck-alteration} In this section we aim to construct an analogue of the Grothendieck--Springer resolution over $\Lambda$. As before, a classical reference for the construction over algebraically closed fields is \cite[\S4.4, \S4.7]{Slodowy1980}.

Recall that the free module $\mathfrak{b}$ is a $B$-representation through the adjoint action of $B$ on $\mathfrak{b}$ (see \S\ref{section: notations and conventions for lie stuff}), and so via the $B$-torsor $G \to G/B$ we can construct the associated \emph{adjoint} bundle
\begin{equation}\label{grothendieck-alteration-equation}
    \gres \coloneqq G \times^B \borel := \left(G\times_{\Lambda} \borel\right)/B, \; \; \; b \cdot (g,x) = (gb^{-1}, \textrm{Ad}_b(x)) \; \; \textrm{for} \; \; b\in B, g \in G, x \in \borel
\end{equation}
with the induced $B$-action on $G\times_{\Lambda} \borel$ indicated on the right. Note we have a closed immersion $\widetilde{g} \hookrightarrow G \times^B \g$ and the latter is identified with $G/B \times_{\Lambda} \g$ via $[g,x] \mapsto (gB, \textrm{Ad}(g)x)$. Together with the $B$-invariant projection $G\times_{\Lambda}\borel \to G/B$ we obtain two maps
\begin{center}
    \begin{tikzcd}
        \widetilde{\g} \arrow[d, "p"] \arrow[r, "\pi"] & \g \\ G/B
    \end{tikzcd}
\end{center}
where $p$ is the Zariski-locally trivial vector bundle associated with the $B$-representation $\borel$ by descent along $B$-torsor $G \to G/B$, and $\pi$ is the projection $\gres \hookrightarrow G/B\times_{\Lambda} \g \to \g$, sending $[g,x] \mapsto \textrm{Ad}(g)x$. We call projective morphism $\pi$ the \emph{Grothendieck alteration}.

\begin{Rem}\label{remark: gres as incidence variety} An equivalent description is $\gres = \{ (B', x) \in G/B \times_{\Lambda} \g \mid x \in \textrm{Lie}(B')\}$. Moreover $\pi$ is dominant and the formation of $\gres$ commutes with base-change on $\Lambda$; see \cite[\S4.1.7]{bouthier2019torsors} for details.
\end{Rem}

We next construct a morphism $\widetilde{\chi}: \gres \to \h$. The quotient map $B \to T \simeq B/U$ yields an exact sequence of $B$-modules
\[ 0 \to \mathfrak{n} \to \borel \to \h \to 0\]
and since $B$ acts trivially on $\h$ (as its conjugation action on $B/U$ is trivial), we get an induced exact sequence of vector bundles
\begin{equation}\label{eq: bundle sequence defining widetilde chi}
    0 \to G\times^B \mathfrak{n} \to \gres \to G\times^B\h \simeq G/B \times_{\Lambda} \h \to 0
\end{equation}
Then we define $\widetilde{\chi}: \gres \to G/B\times_{\Lambda} \h \to \h$ by sending $[g,x] \mapsto x \mod \mathfrak{n}$.

\begin{Thm}\label{thm: all the properties of simultaneous resolution diagrams} With the above notations, we have a commutative diagram 

\vspace{-2em}\begin{center}
    \begin{equation}\label{relative-resolution-diagram}
        \begin{tikzcd}
\widetilde{\mathfrak{g}}\arrow[d,"\widetilde{\chi}"] \arrow[r, "\pi"] & \mathfrak{g} \arrow[d, "\chi"] \\ \mathfrak{h}\arrow[r, "\psi"] & \mathfrak{h}/\!\!/W
        \end{tikzcd}
       \end{equation}
    \end{center}
whose construction commutes with base-change on $\Lambda$. Moreover:
\begin{itemize}
    \item[\emph{(i)}] $\widetilde{\chi}$ is smooth surjective and $\pi$ is finite locally free of degree $|W|$ over $\g^{\textrm{reg}}$. Moreover, the diagram induces a map $q: \gres \to \g \times_{\h/\!\!/W}\h$ which is projective birational.
    \item[\emph{(ii)}] Restricting to regular semisimple loci $\gres^{\textrm{rs}}:=\pi^{-1}(\g^{\textrm{rs}}) \to \g^{\textrm{rs}}$ yields a cartesian diagram
    \vspace{-3em}   \begin{center}
    \begin{equation}\label{eq: regular semisimp relative-resolution-diagram}
        \begin{tikzcd}
\widetilde{\mathfrak{g}}^{\textrm{rs}}\arrow[d,"\widetilde{\chi}^{\textrm{rs}}"] \arrow[r, "\pi^{\textrm{rs}}"] & \mathfrak{g}^{\textrm{rs}} \arrow[d, "\chi^{\textrm{rs}}"] \\ \mathfrak{h}^{\textrm{rs}}\arrow[r, "\psi^{\textrm{rs}}"] & (\mathfrak{h}/\!\!/W)^{\textrm{rs}}
        \end{tikzcd}
       \end{equation}
    \end{center}
    where $\pi^{\textrm{rs}}, \psi^{\textrm{rs}}$ are finite \'etale $W$-torsors; here $(-)^{\textrm{rs}}$ denotes the restriction of the relevant maps to the regular semisimple loci. 

    \item[\emph{(iii)}] The map $G/T \times_{\Lambda} \h^{\textrm{rs}} \to \gres^{\textrm{rs}}$, $(gT,h) \mapsto [g,h]$ is an isomorphism and the map $\widetilde{\chi}^{\textrm{rs}}$ is $W$-equivariant with respect to the following action on the source:
    \[a_w(gT,h) = (g\widetilde{w}^{-1}T, \textrm{Ad}(\widetilde{w})h), \;\;\; \widetilde{w} \in N_G(T) \;\textrm{lift of} \; w \in W\]
    \item[\emph{(iv)}] Base-changing diagram \emph{(\ref{relative-resolution-diagram})} along any geometric point $\overline{t} \to \Spec(\Lambda)$ yields a simultaneous resolution diagram in the sense that for all $h \in \h_{\overline{t}}(k(\overline{t}))$ the induced map
    \begin{equation}\label{eq: geometric point resolutions of each fiber}
        \pi(\overline{t},h): \widetilde{\chi}_{\overline{t}}^{-1}(h) \to \chi_{\overline{t}}^{-1}(\psi_{\overline{t}}(h))
    \end{equation}
    is a resolution of singularities. Here $(-)_{\overline{t}}$ denotes the relevant pulled-back maps and spaces from \emph{(\ref{relative-resolution-diagram})} along $\overline{t} \to \Spec(\Lambda)$.
\end{itemize}
\end{Thm}

\begin{proof} Commutativity and base-change compatibility of (\ref{relative-resolution-diagram}) follow from \cite[\S4.1.7, (4.1.7.4)]{bouthier2019torsors}. For (i), $\widetilde{\chi}$ is a composition of smooth surjections, the first map being locally a projection of vector bundles, hence it is smooth surjective. By \cite[Lem. 4.2.12]{bouthier2019torsors} we have $\gres^{\textrm{reg}}:= \pi^{-1}(\g^{\textrm{reg}})\simeq \g^{\textrm{reg}}\times_{\h/\!\!/W}\h$ and $\pi\lvert_{\gres^{\textrm{reg}}}$ is the pullback of $\psi$ which is finite locally free of degree $|W|$ by Theorem \ref{thm: properties of adjoint quotient and branched cover psi}(iv), so we get the same property for $\pi$ over $\g^{\textrm{reg}}$. For the ``moreover'' part of (i), $q$ is induced by the universal property of fiber products and factors as a closed immersion and $\gres \times_{\h/\!\!/W}\h \to \g\times_{\h/\!\!/W}\h$ which is a base-change of $\pi$, hence $q$ is projective and an isomorphism over open dense locus $\gres^{\textrm{reg}}$, i.e. it is also birational.

For (ii), by \cite[Lemma 4.2.12]{bouthier2019torsors} we have an identification \[\gres^{\textrm{rs}} = \gres^{\textrm{reg}}\times_{\g^{\textrm{reg}}} \g^{\textrm{rs}}  \simeq \h^{\textrm{rs}}\times_{(\h/\!\!/W)^{\textrm{rs}}}\g^{\textrm{rs}}\]
so (\ref{eq: regular semisimp relative-resolution-diagram}) is a cartesian diagram and $\pi^{\textrm{rs}}, \psi^{\textrm{rs}}$ are finite \'etale $W$-torsors by \cite[Prop.\ 4.1.8]{bouthier2019torsors}. (iii) also follows from \cite[Prop.\ 4.1.8, (4.1.8.1)]{bouthier2019torsors}. 
Finally (iv) is the content of the Grothendieck--Springer simultaneous resolution theorem over algebraically closed fields of very good characteristic, which is proven in \cite[\S4.7]{Slodowy1980}; the normality of the target of $\pi(\overline{t}, h)$ is supplied by Theorem \ref{thm: properties of adjoint quotient and branched cover psi}(ii).\end{proof}

\begin{Rem} While $\pi: \gres \to \g$ is commonly referred to as the Grothendieck--Springer resolution, we insist on calling $\pi$ the Grothendieck alteration as it is a map generically of degree $|W|$ and thus not birational. The factoring map $q: \gres \to \g\times_{\h/\!\!/W}\h$ is proper birational and furnishes the simultaneous resolution of $\chi: \g \to \h/\!\!/W$ after base-change along $\psi$.
    
\end{Rem}

For later purposes, we also record here a version of the Grothendieck--Springer resolution when the fibers are $\Lambda$-schemes.

\begin{Cor}\label{cor: grospri resol over dvr} Let $x \in \h(\Lambda)$ be an integral section, viewed as a map $x: \Spec(\Lambda) \to \h$. The induced map
\[\pi(x): \gres \times_{\h}\Spec(\Lambda) \to \g \times_{\h/\!\!/W}\Spec(\Lambda)\]
is a resolution of singularities of $\Lambda$-schemes so that the $\Lambda$-fiberwise induced maps are \emph{(\ref{eq: geometric point resolutions of each fiber})}. Here the base-changes to $\Spec(\Lambda)$ are understood to be via maps $x$ on the source and $\psi(x)$ on the target.
\end{Cor}
\begin{proof} Note the source and target of $\pi(x)$ are closed $\Lambda$-subschemes of $\gres$ and $\g$, respectively. Then $\gres\times_{\h}\Spec(\Lambda)$ is smooth over $\Spec(\Lambda)$ by base-change from $\widetilde{\chi}$ (Theorem \ref{thm: all the properties of simultaneous resolution diagrams}(i)), $\pi(x)$ is the base-change of $q: \gres \to \g \times_{\h/\!\!/W}\h$ hence it is projective, and it is an isomorphism over $(\g \times_{\h/\!\!/W}\Spec(\Lambda)) \cap \g^{\textrm{reg}}$, so it is also birational (Theorem \ref{thm: all the properties of simultaneous resolution diagrams}(i) again) and $\pi(x)$ is a resolution. Lastly, base-change along $\overline{t} \to \Spec(\Lambda)$ yields
\[\pi(x)_{\overline{t}}: \widetilde{\chi}_{\overline{t}}^{-1}(x_{\overline{t}}) \to \chi_{\overline{t}}^{-1}(\psi_{\overline{t}}(x_{\overline{t}}))\]
and so $\pi(x)_{\overline{t}}$ is resolution $\pi(\overline{t}, x_{\overline{t}})$ in the notation of Theorem \ref{thm: all the properties of simultaneous resolution diagrams}(iv).\end{proof}

In particular, choosing the zero section $x = 0 \in \h(\Lambda)$ in Corollary \ref{cor: grospri resol over dvr} yields the relative Springer resolution $G\times^B \mathfrak{n} \to \mathcal{N}$; further base-change to a geometric point $\overline{t} \to \Spec(\Lambda)$ yields the classical Springer resolution over fields (see \cite[\S4.1]{Slodowy1980} for the unipotent version). 

\section{Integral transverse slices}\label{section-integral-slices} The purpose of this section is to define a homogeneous transverse slice to subregular nilpotent sections, whose construction works in all very good characteristics. The end results are Proposition \ref{prop: slice simultaneous rdp resolution} and Theorem \ref{thm: equivariant miniversality of slice deformation and resolution}, which use the geometry of the Grothendieck alteration to show that restricting the adjoint quotient to this slice yields an equivariantly miniversal deformation of a rational double point singularity, and the deformations of its minimal resolution are also captured by this construction.

\subsection{Homogeneous transverse slices.}\label{slodowy-slice-subsection} Throughout we work over a mixed characteristic complete discrete valuation ring $\Lambda$ with algebraically closed residue field $k$. Let $G/\Lambda$ be a split simple simply-connected group of ADE type, $\g = \textrm{Lie}(G)$ its Lie algebra of rank $r$ and assume the residue characteristic $p$ is very good (see Definition \ref{definition-of-good-primes}) for the associated root lattice $\Phi$. We view $\g$ both as a finite free $\Lambda$-module and as its associated affine $\Lambda$-scheme.

\begin{Def} Fix a subregular nilpotent section $x \in \g(\Lambda)$ as in Definition \ref{def: reg subreg and regsemisimp} (ii). For a smooth locally closed $\Lambda$-subscheme $\Sigma \subset \g$ containing section $x$, we say $\Sigma$ is a \emph{relative transverse slice} at $x$ if for every geometric point $\overline{t} \to \Spec(\Lambda)$ we have decomposition
\[T_{x_{\overline{t}}} \g_{\overline{t}} = [\g_{\overline{t}}, x_{\overline{t}}]\oplus T_{x_{\overline{t}}}\Sigma_{\overline{t}}\]
    
\end{Def}

Here $T_{x_{\overline{t}}}\Sigma_{\overline{t}} = \Hom_{k(t)}(x_{\overline{t}}^*\Omega^1_{\Sigma/\Lambda}, \overline{k(t)})$. We note that the associated adjoint action map $G \times \Sigma \to \g$ is smooth in a neighborhood of $G \times \{x\}$. 

By a simple base-change argument we have the following preliminary lemma:

\begin{Lemma}\label{lemma: basechange of slice action map} Let $f: X \to \g$ be a $G$-equivariant morphism of $\Lambda$-schemes and set $\Sigma_X := X \times_{\g} \Sigma$. Then we have a cartesian diagram

\begin{center}
    \begin{tikzcd}
        G \times \Sigma_X \arrow[d] \arrow[r] & X \arrow[d, "f"] \\ G \times \Sigma \arrow[r] & \g
    \end{tikzcd}
\end{center}
and so the action map $G \times \Sigma_X \to X$  is smooth if  the action map $G \times \Sigma \to \g$ is smooth.
    
\end{Lemma}

Since $x \in \g(\Lambda)$ is fiberwise subregular and $\Lambda$ is strictly henselian, we can associate to $x$ a cocharacter $\lambda: \mathbb{G}_m \to G$ such that $\lambda_{\overline{t}}$ is a cocharacter associated to $x_{\overline{t}}$ (see \cite[Def.\ 2.9, Lemma 5.7 (3)]{cotner2022centralizers}). In other words we have a $\mathbb{G}_m$-action on $\g$ and a compatible $\mathbb{Z}$-grading
\begin{equation}\label{eq: z-grading of adjoint action} \begin{gathered}\textrm{Ad}_{\lambda(t)}x = t^2x, \;\;\; t \in \mathbb{G}_m(\Lambda) \\
\g = \bigoplus_{i \in \mathbb{Z}}\g(i), \;\;\; \g(i) := \{y \in \g \mid \textrm{Ad}_{\lambda(t)}y = t^iy\}\end{gathered}\end{equation}

\begin{Prop}\label{proposition: construction of relative homogeneous transverse slice}
    There is a graded free summand $\mathfrak{a} \subset \bigoplus_{i \leq 0} \g(i)$ such that $\g = [x, \g]\oplus \mathfrak{a}$ and the affine scheme $\Sigma = x + \mathfrak{a}$ is a relative transverse slice of (relative) dimension $r+2$ with smooth action map $G \times \Sigma \to \g$.
\end{Prop}

\begin{proof} Since we are in very good characteristic, the orbit maps for the adjoint action on $\g$ are separable (\cite[\S5.1, Rem.\! (2)]{Slodowy1980}) and so $\dim \ker (\textrm{ad}(x_{\overline{t}})) = r+2$ on geometric fibers, by subregularity. Thus $M = \textrm{coker}(\textrm{ad}(x))$ is a free $\Lambda$-module of rank $r+2$; we may write $M = \g/[x,\g]$.

Now $x \in \g(2)$ so $[x, \g(i)] \subseteq \g(i+2)$ and we have a grading on $M = \bigoplus_i M(i)$ with $M(i) = \g(i)/[x,\g(i-2)]$ also a free $\Lambda$-module (as a summand). By passing to the characteristic-zero geometric generic fiber $\overline{t} = \Spec(\overline{K}) \to \Lambda$ we have that $M(i) \otimes_{\Lambda} \overline{K} = 0$ for $i \geq 1$ since $\g_{\overline{t}}$ is a completely reducible $\mathfrak{sl}_2$-representation in characteristic zero (see e.g. \cite[\S7.4]{Slodowy1980}), thus $M(i) = 0$ for $i \geq 1$. By choosing a (graded) splitting of $\g(i) \twoheadrightarrow M(i)$ and denoting its image by $\mathfrak{a}(i)$, we set
\[ \mathfrak{a} = \bigoplus_{i \in \mathbb{Z}}\mathfrak{a}(i) = \bigoplus_{i \leq 0 }\mathfrak{a}(i), \;\;\; \mathfrak{a}(i) \simeq M(i)\]
and by construction $\g(i) = [x,\g(i-2)]\oplus \mathfrak{a}(i)$ so $\g = [x, \g] \oplus \mathfrak{a}$ and $\Sigma = x + \mathfrak{a}$ is a relative transverse slice at $x$.

It remains to check that the action map is smooth everywhere. Define a $\mathbb{G}_m$-action on $\g$ via \begin{equation}\label{eq: contracting action on slice}\mu_t(y) = t^2\textrm{Ad}_{\lambda(t)^{-1}}y, \;\;\; t \in \mathbb{G}_m(\Lambda)\end{equation}

By construction it preserves $\Sigma$, fixes $x$ and the weights on each nontrivial $\mathfrak{a}(i)$ are $2-i>0$ so the action contracts $\Sigma$ to $x$; concretely, this means that $\mu: \mathbb{G}_m \times_{\Lambda} \Sigma \to \Sigma$ extends to an action $\mathbb{A}_{\Lambda}^1\times_{\Lambda} \Sigma \to \Sigma$ and $\mu_0(y) = x$. The locus of $\Sigma$ where the action map at $(1,y) \in G \times \Sigma$ is not smooth is closed and stable under this contracting action, hence must contain $x$, contradicting local smoothness at $x$. Translation by $G$ makes $G \times \Sigma \to \g$ fiberwise smooth, hence smooth as $G \times \Sigma$ is flat over $\Lambda$.\end{proof}

\begin{Rem} If we replace $\Lambda$ by an algebraically closed field of characteristic zero or $p> 4\textrm{Cox}(\g)+2$ (where $\textrm{Cox}(\g)$ denotes the Coxeter number of the simple Lie algebra $\g$), $\Sigma$ can be directly constructed via an $\mathfrak{sl}_2$-triple $(x,y,h)$ and $\mathfrak{a} = \ker(\textrm{ad}y)$; $\Sigma$ is known as a Slodowy slice (cf.\! \cite[\S7.4]{Slodowy1980}). Proposition \ref{proposition: construction of relative homogeneous transverse slice} gives essentially a relative-to-$\Lambda$ version of the construction in \cite{spaltenstein1984existence}.  
\end{Rem}

We now make precise the sense in which $\Sigma$ is \emph{homogeneous}. We call a regular function $f \in \Lambda[\Sigma]$ homogeneous of degree $m$ under action $\mu: \mathbb{G}_m \times \Sigma \to \Sigma$ if \[\mu^*f = t^m \otimes f \in \Lambda[t^{\pm}]\otimes_{\Lambda} \Lambda[\Sigma]\] This notion makes sense for regular functions on any affine scheme with a $\mathbb{G}_m$-action. For $\Sigma$ specifically write $\Lambda[\Sigma] = \Lambda[u_1, \cdots, u_{r+2}]$ where $u_i$ are coordinate functions on $\Sigma$; for a homogeneous $\Lambda$-basis $\{e_j \in \mathfrak{a}(i_j)\}$ of $\mathfrak{a}$ we have $u_i(x + \sum_j a_je_j) = a_i$. One directly checks that the $\mathbb{G}_m$-action on $\Sigma$ defined in (\ref{eq: contracting action on slice}) yields $u_j$ of weight $2-i_j > 0$; here, $i_j$ is the degree of element $e_j$ under the $\mathbb{Z}$-grading in (\ref{eq: z-grading of adjoint action}).

\subsection{Grothendieck alterations for transverse slices.}\label{section: groth alts for slices} We recall the setting of Section \ref{section: notations and conventions for lie stuff}, and in particular retain Assumption \ref{assumptions on G g etc} for $G$ and root lattice $\Phi(G,T)$. Setting $\g = \textrm{Lie}(G)$ and $\h = \textrm{Lie}(T)$ again, we also recall from Section \ref{relative-grothendieck-alteration} the construction of the adjoint bundle $\gres$, the adjoint quotient $\chi\colon \g \to \h/\!\!/W$ and $\widetilde{\chi}\colon \gres \to \h$. For a homogeneous (relative) transverse slice $\Sigma \subset \g$ we set: \[\widetilde{\Sigma} = \gres \times_g \Sigma, \;\;\chi_{\Sigma} = \chi\lvert_{\Sigma}, \;\;\widetilde{\chi}_\Sigma = \widetilde{\chi}\lvert_{\widetilde{\Sigma}}\]

\begin{Prop}\label{prop: slice simultaneous rdp resolution} We have a commutative diagram with the following properties \vspace{-1.5em}
\begin{center}
\begin{equation}\label{tikzcd: restricted slice alteration diagram}
    \begin{tikzcd}
        \widetilde{\Sigma} \arrow[dr, "\widetilde{\chi}_{\Sigma}"] \arrow[r, "q"] & \Sigma \times_{\h/\!\!/W} \h \arrow[d] \arrow[r] & \Sigma \arrow[d, "\chi_{\Sigma}"] \\ & \h \arrow[r, "\psi"] & \h/\!\!/W
    \end{tikzcd}
    \end{equation}
\end{center}
\begin{itemize}
    \item[\emph{(i)}] $\chi_{\Sigma}$ is faithfully flat with geometrically normal surface fibers.

    \item[\emph{(ii)}] Let $Z = \chi_{\Sigma}^{-1}(0) = \Sigma \cap \mathcal{N}$, then for any geometric point $\overline{t} \to \Spec(\Lambda)$ we have that $Z_{\overline{t}} = \Sigma_{\overline{t}} \cap \mathcal{N}_{\overline{t}}$ has a unique rational double point singularity with associated root lattice $\Phi = \Phi(G,T)$.
    \item[\emph{(iii)}] $\widetilde{\chi}_{\Sigma}$ is smooth of relative dimension 2  and $q$ is projective and induces the minimal resolution at every geometric fiber over $\h$.
\end{itemize}
    
\end{Prop}

\begin{proof} The diagram (\ref{tikzcd: restricted slice alteration diagram}) is obtained from restricting the Grothendieck alteration diagram (see Section \ref{relative-grothendieck-alteration}) to $\Sigma \subset \g$. For (i), note that $\alpha: G \times \Sigma \to \g$ is smooth by Proposition \ref{proposition: construction of relative homogeneous transverse slice} and since $\chi$ is $G$-invariant we have $\chi \circ \alpha = \chi_{\Sigma} \circ \textrm{pr}_2$ where $\textrm{pr}_2: G\times \Sigma \to \Sigma$. By faithfully flat descent along $\textrm{pr}_2$ and Theorem \ref{thm: properties of adjoint quotient and branched cover psi} (ii) we get that $\chi_{\Sigma}$ is flat; faithfulness follows by noting the image of $\chi_\Sigma$ is open, invariant under the contraction action on $\h/\!\!/W$ and contains the origin in every geometric fiber, so its complement is empty. As normality is preserved by smooth morphisms we get that the geometric fibers are normal of dimension $\dim_{\Lambda}(\Sigma)-r = 2$.

For (ii), we have that $x_{\overline{t}}$ is subregular nilpotent and $\Sigma_{\overline{t}}$ is a transverse slice (\cite[\S5.1]{Slodowy1980}). Since $\textrm{char}(k(\overline{t})) = 0$ or $p$ very good for $\Phi$, we have by \cite[\S6.4 Thm.]{Slodowy1980} that $\Sigma_{\overline{t}} \cap \mathcal{N}_{\overline{t}}$ has a rational double point of type $\Phi$ at $x_{\overline{t}}$, and moreover by the contracting action $\mu$ defined in (\ref{eq: contracting action on slice}) it is the unique singularity of $Z_{\overline{t}}$.

For (iii), we have by Lemma \ref{lemma: basechange of slice action map} applied to the Grothendieck alteration $\pi: \gres \to \g$ that $\widetilde{\alpha}: G \times_{\Lambda} \widetilde{\Sigma} \to \gres$ is smooth and moreover $\widetilde{\chi} \circ \widetilde{\alpha} = \widetilde{\chi}_{\Sigma} \circ \textrm{pr}_2$ as before. By faithfully flat descent we conclude $\widetilde{\chi}_{\Sigma}$ is smooth. Projectivity of $q$ follows from base-changing $\gres \to \g \times_{\h/\!\!/W} \h$ along $\Sigma \hookrightarrow \g$. For the final assertion, let $\overline{t} \to \h$ be a geometric point and form
\[q_{\overline{t}}: Y_{\overline{t}}:= \widetilde{\chi}_{\Sigma}^{-1}(\overline{t}) \to X_{\overline{t}} := \chi_{\Sigma}^{-1}(\psi(\overline{t}))\]
$Y_{\overline{t}}$ is a smooth surface, and for $G_{\overline{t}} = G\times_{\Lambda} \overline{t}$, $\g_{\overline{t}} = \g \times_{\Lambda} \overline{t}$ and $\gres_{\overline{t}} = \gres \times_{\Lambda} \overline{t}$ we have $G_{\overline{t}}$-equivariant maps $\chi^{-1}(\psi(\overline{t})) \to \g_{\overline{t}}$ and $\widetilde{\chi}^{-1}(\overline{t}) \to \gres_{\overline{t}}$. Applying Lemma \ref{lemma: basechange of slice action map} we obtain smooth morphisms $G_{\overline{t}} \times X_{\overline{t}} \to \chi^{-1}(\psi(\overline{t}))$ and $G_{\overline{t}} \times Y_{\overline{t}} \to \widetilde{\chi}^{-1}(t)$. Thus $\textrm{id}_{G_{\overline{t}}} \times q_{\overline{t}}$ is a smooth base-change of the resolution $\widetilde{\chi}^{-1}(\overline{t}) \to \chi^{-1}(\psi(\overline{t}))$ from Theorem \ref{thm: all the properties of simultaneous resolution diagrams} (iv), and so $q_{\overline{t}}$ is itself a resolution.

It remains to check minimality of $q_{\overline{t}}$. As both $\gres/\Lambda$ and $\h/\Lambda$ have trivial canonical bundle, we have $\omega_{\gres/\h} \simeq \Oring_{\gres}$. By adjunction along the codimension $c$ smooth subvariety $Y_{\overline{t}} \subset \gres_{\overline{t}}$ we get \[\omega_{Y_{\overline{t}}} \simeq \omega_{\gres}\lvert_{Y_{\overline{t}}}\otimes \det\mathcal{N}_{Y_{\overline{t}}/\gres} \simeq \Oring_{Y_{\overline{t}}} \otimes \det \Oring_{Y_{\overline{t}}}^{\oplus c} \simeq \Oring_{Y_{\overline{t}}}\] Thus for any exceptional rational curve $C$ in $Y_{\overline{t}}$ we have $-2 = K_{Y_{\overline{t}}}\cdot C + C^2 = C^2$ and so there are no exceptional $(-1)$-curves, i.e. $q_{\overline{t}}$ is a minimal resolution.\end{proof}

Proposition \ref{prop: slice simultaneous rdp resolution} says that $q: \widetilde{\Sigma} \to \Sigma \times_{h/\!\!/W} \h$ is a simultaneous resolution of the rational double point surface family $\chi_{\Sigma}$, after base-changing along finite morphism $\psi$.

In order to view the above as $\mathbb{G}_m$-equivariant deformations of singularities and their minimal resolutions, we need to equip the relevant spaces with compatible $\mathbb{G}_m$-actions as follows:

Recall the action $\mu$ on $\g$ from (\ref{eq: contracting action on slice}), which preserves slice $\Sigma$. On $\gres$ we define an action
\begin{equation}\label{eq: mu tilde action on resolved slice}\widetilde{\mu}: t \cdot [g,v] = [\lambda(t^{-1})g, t^2v], \; \; [g,v] \in (G\times^B \mathfrak{b})(\Lambda)\end{equation}
where $\lambda: \mathbb{G}_m \to G$ is the associated cocharacter defining $\Sigma$; this action preserves $\widetilde{\Sigma}$. Finally on $\h$ we have the contracting action $t \cdot h = t^2h$ for $h \in \h(\Lambda)$ which commutes with $W$ and thus induces an action on $\h/\!\!/W$; explicitly, one can choose homogeneous generators $p_1, \cdots, p_r$ of $\Lambda[\h]^W$ by Theorem \ref{thm: properties of adjoint quotient and branched cover psi} (i) and the induced weights of the $\mathbb{G}_m$-action are $2\deg(p_i)$.

We next give a Lie-theoretic description of the exceptional divisors which appear in the minimal resolutions of Proposition \ref{prop: slice simultaneous rdp resolution} (iii), which will then be used in Theorem \ref{thm: equivariant miniversality of slice deformation and resolution}. For ease of notation, we work over the residue field $k$. Let $E \subset \widetilde{Z} \to Z$ be the reduced exceptional divisor of the minimal resolution of $Z = \chi_{\Sigma}^{-1}(0)$, which has a unique RDP by Proposition \ref{prop: slice simultaneous rdp resolution} (ii). Decompose $E = \cup_i E_i$ into its irreducible components.

Let $\{\alpha_i\}_{i\leq r}$ be the simple roots determined by $B$ with associated minimal parabolics $P_i$ such that $P_i/B \simeq \mathbb{P}^1_k$. Via the natural projection $\widetilde{\chi}^{-1}(0) = G\times^B\mathfrak{n} \to G/B$ we can identify $E_i \subset \widetilde{Z} \subset G \times^B \mathfrak{n}$ with its image $P_i/B$ in $G/B$, up to translating by $G$ (see \cite[\S6.3-6.4]{Slodowy1980} for details).

On the other hand $\widetilde{Z} \subset \widetilde{\Sigma}$, and the normal bundle sequence (\cite[Prop.\! 19.1.5]{grothendieck1967elements})
\[0 \to \mathcal{N}_{E_i/\widetilde{Z}} \to \mathcal{N}_{E_i/\widetilde{\Sigma}} \to \h \otimes \Oring_{E_i} \to 0\]
yields a connecting homomorphism $\delta: \h \to \coh^1(E_i, \mathcal{N}_{E_i/\widetilde{Z}}) \simeq \coh^1(\mathbb{P}^1_k, \Oring(-2)) \simeq k$.

\begin{Lemma}\label{lemma: technicalities on conn homs} There exist generators $e_i\in \coh^1(E_i, \mathcal{N}_{E_i/\widetilde{Z}})$ such that $\delta(v) = d\alpha_i(v)e_i$ $ \forall v \in \h$.   
\end{Lemma}
\begin{proof} Let us write $E=E_i, P_i=P$ and suppress the index $i$. We have identified $E$ with $P/B \subset G/B$. Restricting the vector bundle sequence (\ref{eq: bundle sequence defining widetilde chi}) from Section \ref{relative-grothendieck-alteration} to $E$ yields
\begin{equation}\label{eq: min parabolic ses} 0 \to P\times^B \mathfrak{n} \to P \times^B \mathfrak{b} \to \h \otimes \Oring_E \to 0\end{equation}
Let $U_P$ and $L$ be the corresponding unipotent radical and Levi subgroup of $P$, $B_L = B \cap L$, $\borel_L = \textrm{Lie}(B_L)$, $\mathfrak{n}_P = \textrm{Lie}(U_P)$ and $\mathfrak{n}_L = \textrm{Lie}(\mathcal{R}_u(B_L))$. We have identifications
\[E \simeq P/B \simeq L/B_L, \; P \times^B (\mathfrak{n}/\mathfrak{n}_P) \simeq L \times^{B_L} \mathfrak{n}_L, \;\; P\times^B (\borel/\mathfrak{n}_P) \simeq L \times^{B_L}\borel_L\]
via the quotient $P \to L$. Thus quotiening (\ref{eq: min parabolic ses}) by the trivial sub-bundle $P \times^B \mathfrak{n}_P \simeq E \times \mathfrak{n}_P$ yields exact sequence
\begin{equation}\label{eq: levi ses} 0 \to L \times^{B_L} \mathfrak{n}_L \to L \times^{B_L}\borel_L \to \h \otimes \Oring_E \to 0\end{equation}
Moreover $\coh^1(E, P \times^B \mathfrak{n}) \simeq \coh^1(E, L\times^{B_L} \mathfrak{n}_L)$ so the connecting homomorphisms of (\ref{eq: min parabolic ses}) and (\ref{eq: levi ses}) are identified. 

Now the natural quotient $L \to L/Z(L)\simeq \textrm{PGL}_2$ induces a differential map on the associated Cartan algebras $\h = \borel_L/\mathfrak{n}_L \to k$ which is identified with $d\alpha$. The associated bundle sequence of $\mathfrak{pgl}_2$ is the Euler sequence \[0 \to \Oring_{\mathbb{P}^1}(-2) \to \Oring_{\mathbb{P}^1}(-1)^{\oplus 2} \to \Oring_{\mathbb{P}^1} \to 0\] and so the quotient $L \to \textrm{PGL}_2$ induces a map of associated bundle sequences
\begin{center}
    \begin{tikzcd}
        0 \arrow[r] & L \times^{B_L}\mathfrak{n}_L \arrow[d, "\simeq"] \arrow[r] & L\times^{B_L} \borel_L \arrow[d] \arrow[r] & \h \otimes \Oring_E \arrow[r] \arrow[d, "d\alpha"] & 0 \\ 0 \arrow[r] & \Oring(-2) \arrow[r] & \Oring(-1)^{\oplus 2} \arrow[r] & \Oring \arrow[r] & 0
    \end{tikzcd}
\end{center}
where on the left we identify $L\times^{B_L} \mathfrak{n}_L$ with the cotangent bundle of $E \simeq \mathbb{P}^1$. Since the connecting homomorphism of the bottom row is isomorphism $k \overset{\sim}{\to} \coh^1(E, \Oring_E(-2))$, by commutativity of the right side of the diagram we get that the connecting homomorphism $\partial: \h \to \coh^1(E, \Oring(-2))$ in the top row is $d\alpha$; this is also the connecting homomorphism in (\ref{eq: min parabolic ses}). It remains to show that this connecting homomorphism coincides with $\delta: \h \to \coh^1(E, \mathcal{N}_{E/\widetilde{Z}})$. 

Writing $\Sigma = x + \mathfrak{a}$ for the relative transverse slice, we have a normal bundle sequence
\[ 0 \to \mathcal{N}_{E/\widetilde{Z}} \to \mathcal{N}_{E/G\times^B\mathfrak{n}} \to \g/\mathfrak{a} \otimes \Oring_E \to 0\]
which splits since $\g = [\g,x]\oplus \mathfrak{a}$ and hence yields an isomorphism \[\coh^1(E, \mathcal{N}_{E/G\times^B\mathfrak{n}}) \simeq \coh^1(E, \mathcal{N}_{E/\widetilde{Z}})\]
Similarly via the normal bundle sequence associated to vector bundle $G\times^B \mathfrak{n} \to G/B$ we have
\[0 \to P \times^B \mathfrak{n} \to \mathcal{N}_{E/G\times^B\mathfrak{n}} \to \mathcal{N}_{E/(G/B)} \to 0\]
where $E \simeq P/B$; since $\mathcal{N}_{E/(G/B)}$ is trivial, we get an isomorphism $\coh^1(E, P \times^B \mathfrak{n}) \simeq \coh^1(E,\mathcal{N}_{E/G\times^B\mathfrak{n}})$. Altogether, this identifies the connecting homomorphisms $\partial$ and $\delta$ via $\delta(v) = \partial(v) = d\alpha(v)e$ after setting $\coh^1(E, P\times^B\mathfrak{n}) \simeq \coh^1(E, \Oring(-2)) \simeq k\langle e\rangle$.\end{proof}

\begin{Thm}\label{thm: equivariant miniversality of slice deformation and resolution} Assume $p$ is very good for $\Phi(G,T)$ and consider a homogeneous transverse slice $\Sigma \subset \g$ at subregular nilpotent $x$ with maps $\chi_{\Sigma}, \widetilde{\chi}_{\Sigma}$ as above. Let $\overline{t} = \Spec(k) \to \Lambda$ be the closed point of $\Lambda$ and set $Z = \chi_{\Sigma}^{-1}(0)$, $R = \widehat{\Oring}_{Z_k, x_k}$, $A = \widehat{\Oring}_{h/\!\!/W,0}$ and $B = \widehat{\Oring}_{\h,0}$. Suppose $X$ is the minimal resolution of $\Spec(R)$ with exceptional divisor $E$ and $\mathcal{X}:=  \widehat{X_E}$. Denote the completions of $\chi_{\Sigma}$ and $\widetilde{\chi}_{\Sigma}$ along $x_k$ and $E$ respectively by $\widehat{\chi}_{\Sigma}$ and $\widehat{\widetilde{\chi}}_{\Sigma}$.

\begin{itemize}
    \item[\emph{(i)}] $\widehat{\chi}_{\Sigma}$ is an equivariantly miniversal deformation of $R$ over $\emph{\Spf}(A)$.
    \item[\emph{(ii)}] $\widehat{\widetilde{\chi}}_{\Sigma}$ is a miniversal deformation of $\mathcal{X}$ over $\emph{\Spf}(B)$.
\end{itemize}
\end{Thm}

\begin{proof} We have $A \simeq \Lambda[\![p_1, \cdots p_r]\!], B\simeq \Lambda[\![h_1, \cdots, h_r]\!]$ for homogeneous coordinates $p_i, h_i$, thus by Proposition \ref{miniversal-iff-power-series-and-kodaira-spencer-isom} for $\widehat{\chi}_{\Sigma}$ and Proposition \ref{prop: tangent obstruction theory for formal scheme deformations} for $\widehat{\widetilde{\chi}}_{\Sigma}$ it suffices to show that their respective Kodaira--Spencer maps are isomorphisms.

Consider first $\textrm{KS}_{\chi_{\Sigma}}: T_{\Lambda} A \to \textrm{Def}_R(k[\varepsilon]/\varepsilon^2) \simeq T^1(R)$, and write $C = \widehat{\Oring}_{\Sigma, x_k}$ so that $\widehat{\chi}_{\Sigma}: \Spf(C) \to \Spf(A)$. Then 
\[\textrm{KS}_{\chi_\Sigma}(v) = [C\widehat{\otimes}_{A,v} k[\varepsilon]/\varepsilon^2] = [(C/\mathfrak{m}_{\Lambda}C)\otimes_{A/\mathfrak{m}_{\Lambda}A,v} k[\varepsilon]/\varepsilon^2]\]
and consequently it suffices to show that the Kodaira--Spencer map of the characteristic $p$ deformation $\Spf(C/\mathfrak{m}_{\Lambda}C) \to \Spf(A/\mathfrak{m}_{\Lambda}A)\simeq \Spf(k[\![p_1,\cdots, p_r]\!])$ is an isomorphism.

We therefore work over $k$ from now on. Using Proposition \ref{prop: equivariantly miniversal rdp deformations}, we obtain a $\mathbb{G}_m$-equivariantly miniversal deformation as follows: write $R = k[\![x,y,z]\!]/f(x,y,z)$ for $f$ as in Theorem \ref{classification-of-rdps} (ii), choose homogeneous generators $\{g_i\}$ of $T^1(R)$ and then complete the family $F: V(f- \sum t_i g_i) \to \Spec(k[t_1, \cdots, t_r])$ for the desired equivariantly miniversal deformation. By \cite[\S7.4, Prop.\! 2, \S8.1, Lem.\! 2]{Slodowy1980} the weights of the $t_i$ for the $\mathbb{G}_m$-action may be chosen to be $2\deg(p_i)$.

Deformation $\widehat{F}$ induces a classifying map $\widehat{\phi}: \Spf(k[\![p_1, \cdots, p_r]\!]) \to \Spf(k[\![t_1, \cdots, t_r]\!])$ with the same $\mathbb{G}_m$-weights on parameters $p_i, t_i$, and moreover since these weights are positive, the map is algebraic i.e. it is induced by a polynomial map \[\phi: \Spec(k[p_1, \cdots, p_r]) \to \Spec(k[t_1, \cdots, t_r])\] so that $\chi_{\Sigma}\otimes k$ is the pullback of the algebraic family $F$. 

It now suffices to show $\phi^{-1}(0)_{\textrm{red}} = 0$: if $h \in (\h/\!\!/W)_k$ has $\phi(h) = 0$ then $\chi_{\Sigma}^{-1}(h)$ has a rational double point with associated root lattice $\Phi$, and lifting $h$ to $y \in \h_k$ we have that the singularity of $\chi_\Sigma^{-1}(\psi(y))$ has as root lattice a connected component of the $\mathbb{Z}$-lattice $\Phi_y = \langle \alpha \in \Phi \mid d\alpha(y)=0\rangle$ (\cite[\S6.5, Lem.\! 1]{Slodowy1980}). So $\Phi_y = \Phi$, but in very good characteristic the simple root differentials $d\alpha_i$ form a basis of $\h_k^{\vee}$ so $y=h=0$. We conclude by \cite[\S8.1, Lem\! 3]{Slodowy1980} that $\phi$ is an isomorphism and so $\textrm{KS}_{\chi_\Sigma} \otimes_{\Lambda} k =  \textrm{KS}_F \circ d\phi_0$ is an isomorphism, i.e. $\textrm{KS}_{\chi_\Sigma}$ is an isomorphism and we deduce (i).

For (ii), we need to show that the mod $p$ Kodaira--Spencer map is an isomorphism, as in (i). Identifying $\textrm{Def}_{\mathcal{X}}(k[\varepsilon]/\varepsilon^2) \simeq \coh^1(\mathcal{X}, T_{\mathcal{X}})$ as in Section \ref{section: equivariant deformations of sings} and $T_{\Lambda}(B/\mathfrak{m}_{\Lambda}B) \simeq \h_k \simeq k^r$, we have by Proposition \ref{prop: tangent obstruction theory for formal scheme deformations} a composition
\[ k^r \overset{\textrm{KS}}{\longrightarrow} \coh^1(\mathcal{X}, T_{\mathcal{X}}) \overset{\rho}{\longrightarrow} \bigoplus_{i=1}^r H^1(E_i, \mathcal{N}_{E_i/X})\]
where $\rho$ is an isomorphism and $E_i$ are the smooth exceptional curves in the minimal resolution $X \to \Spec(R)$. By construction of the Kodaira--Spencer map, setting $\rho_i$ for the $i$-th component of $\rho$ we have that $\rho_i \circ \textrm{KS}$ is identified with connecting homomorphism $\delta: \h_k \to \coh^1(E_i, \mathcal{N}_{E_i/X})$ of the normal bundle sequence associated to $E_i \subset X \subset \widetilde{\Sigma}$; see the paragraph before Lemma \ref{lemma: technicalities on conn homs}.

Let $\{e_i\}$ generate the $1$-dimensional spaces $\coh^1(E_i, \mathcal{N}_{E_i/X})$ and let $\alpha_i \in \Phi$ be the simple root corresponding to $E_i$. By Lemma \ref{lemma: technicalities on conn homs} we have $\rho \circ \textrm{KS}(v) = (d\alpha_i(v)e_i)_i$, so $\rho \circ \textrm{KS}$ is a linear map $v \mapsto (d\alpha_i(v))_i$, and in the usual simple coroot basis it is given by the Cartan matrix, whose determinant is invertible in very good characteristic. We conclude $\textrm{KS}_{\widehat{\widetilde{\chi}}_\Sigma}$ is an isomorphism and so we get (ii).\end{proof}

As a result, the completed map $\widehat{q}$ is a simultaneous resolution of $\widehat{\chi}_{\Sigma}$ after pullback along $\Spf(B) \to \Spf(A)$. 

\section{The monodromy Weyl action}\label{section-monodromy-weyl-action} The main goal of this section is the proof of Theorem \ref{thm: main thm in intro} (i.e.\ Theorem \ref{thm: main theorem about rdps}). To this end, Sections \ref{subsection-classic-nearby-cycles}-\ref{section: rel perv of grospri sheaf and springer w-action} cover prerequisites like nearby cycles, Springer fiber cohomology and relative perversity, which are then used to define and compare the monodromy and Springer actions in Section \ref{section: monodromy springer w-actions}.

\subsection{Nearby and vanishing cycles.}\label{subsection-classic-nearby-cycles} Throughout we fix a mixed characteristic complete\footnote{For this subsection, $\Lambda$ strictly henselian suffices.} discrete valuation ring $\Lambda$ with algebraically closed point $s = \Spec(k)$ and generic point $\eta = \Spec(K)$, and we write $D^b(X, \mathbb{Q}_\ell), D^b_c(X, \mathbb{Q}_{\ell})$ for the bounded derived category of $\mathbb{Q}_\ell$ sheaves (respectively constructible $\mathbb{Q}_\ell$-sheaves) on noetherian finite type $\Lambda$-schemes $X$ ($\ell \neq p$).

\begin{Def}\label{def: nearby cycles} Let $X$ be a finite-type $\Lambda$-scheme with special fiber $i: X_s \to X$, generic fiber $j: X_{\eta} \to X$ and geometric generic fiber $\overline{j}: X_{\overline{\eta}} \to X$. 
\begin{enumerate}[label=(\roman*)]
    \item The \textit{nearby cycles} functor is $\textrm{R}\Psi_X: \textrm{D}^b(X_{\eta}, \mathbb{Q}_\ell) \to \textrm{D}^b(X_s, \mathbb{Q}_\ell)$, $\mathcal{F} \longmapsto i^*\textrm{R}\overline{j}_*(\mathcal{F}_{\etabar})$, where $\mathcal{F}_{\etabar}$ is the pullback of $\mathcal{F}$ to $X_{\etabar}$.  $\textrm{R}\Psi_X(\mathcal{F})$ is naturally equipped with an action of inertia $I_K = \Gal_K =  \Gal(\etabar/\eta)$ (\cite{deligne2006groupes}, Expos\'e XIII, \S1.3).

    \item For $\mathcal{F} \in \textrm{D}^b(X, \mathbb{Q}_\ell)$, adjunction $\mathcal{F} \to \textrm{R}\overline{j}_*(\mathcal{F}_{\etabar})$ gives an exact $I_K$-equivariant triangle \begin{equation}\label{vanishing-cycles-definition} i^*\mathcal{F} \overset{\phi}{\longrightarrow} \textrm{R}\Psi_X(\mathcal{F}\lvert_{X_{\eta}}) \longrightarrow \textrm{R}\Phi_X(\mathcal{F}) \overset{[1]}{\longrightarrow}\end{equation}
    \noindent defining the \textit{vanishing cycles functor} $\textrm{R}\Phi: \textrm{D}^b(X, \mathbb{Q}_\ell) \to \textrm{D}^b(X_s, \mathbb{Q}_\ell)$ as the cone of map $\phi$. For its functoriality and equivariance see \cite[Expos\'e XIII, 1.4.2, 2.1.1-2.1.2]{deligne2006groupes}.
    \end{enumerate}
\end{Def}
We record two relevant properties of nearby cycles. Given a map $f: X \to Y$ with induced maps $f_{\eta}, f_s$ on the geometric generic, resp. special fibers, we have:

\begin{itemize}
    \item[(i)] For $f$ proper, $\mathcal{F} \in \textrm{D}^b_c(X_{\eta}, \mathbb{Q}_\ell)$, a natural isomorphism $\textrm{R}\Psi_Y(\textrm{R}f_{\eta *}\mathcal{F}) \overset{\sim}{\to}\textrm{R}f_{s*}\textrm{R}\Psi_X(\mathcal{F})$ 
    \item[(ii)] For $f$ smooth, $\mathcal{G} \in D^b_c(Y_{\eta}, \mathbb{Q}_\ell)$, a natural isomorphism $f_s^*\textrm{R}\Psi_Y(\mathcal{G}) \overset{\sim}{\to} \textrm{R}\Psi_X(f_{\eta}^*\mathcal{G})$
\end{itemize}
In particular, when $f: X \to \Spec(\Lambda)$ is smooth we have $\textrm{R}\Psi_X\mathbb{Q}_{\ell} \simeq f_s^*\textrm{R}\Psi_{\Lambda}\mathbb{Q}_{\ell} \simeq \mathbb{Q}_{\ell}$ on $X_s$ so $\textrm{R}^{\bullet \geq 1}\Psi_X \mathbb{Q}_{\ell} = \textrm{R}\Phi_X\mathbb{Q}_{\ell} = 0$; more generally, $\textrm{R}\Phi_X\mathbb{Q}_{\ell}$ and $\textrm{R}^{\bullet \geq 1}\Psi_X\mathbb{Q}_{\ell}$ vanish on the smooth locus of $X/\Lambda$ (see \cite[Expos\'e XIII, 2.1.5, 2.1.7.1-2]{deligne2006groupes}).

Given $\overline{x} \to X_s$ and setting $X_{(\overline{x})}:= \Spec\left(\Oring_{X,\overline{x}}^{\textrm{sh}}\right)$, the nearby cycle stalks are  
\[(\textrm{R}\Psi_X\mathcal{F})_{\overline{x}} \simeq \textrm{R}\Gamma(X_{(\overline{x})} \times_{\Lambda} \etabar, \mathcal{F}_{\eta})\]
by \cite[Expos\'e XIII, Prop.\ 2.1.4]{deligne2006groupes}. We call $X_{(\overline{x})}$ the $\ell$-adic \emph{Milnor ball} around $x$ and $X_{(\overline{x})} \times_{\Lambda} \etabar$ the $\ell$-adic \emph{Milnor fiber} at $x$, which alludes to the topological fact that complex-analytic nearby cycle stalks compute the singular cohomology of Milnor fibers of isolated singularities.

\begin{ex} Let $X$ be a flat proper surface over $\Lambda$, smooth outside an isolated rational double point on the special fiber $x \to X_s$, then  we have (\cite{deligne2006groupes}, Expos\'e XVI) \[(R^0\Psi_X\mathbb{Q}_{\ell})_x \simeq \mathbb{Q}_{\ell}, \;\;  (\textrm{R}^2\Psi_X\mathbb{Q}_{\ell})_x \simeq \mathbb{Q}_{\ell}^r, \;\; (\textrm{R}^{i}\Psi_X\mathbb{Q}_{\ell})_x = 0 \;\textrm{for} \;i\neq 0,2\]
 \noindent where $r$ is the rank of the root lattice associated to RDP $\overline{x}$ (see Theorem \ref{classification-of-rdps}).
\end{ex}

\begin{Prop}\label{prop: isom nearby cycle stalks under formal nbd isom} Assume $\Lambda$ is complete and let $X, Y$ be finite-type $\Lambda$-schemes with closed points $\overline{x} \to X_s, \overline{y} \to Y_s$ on their special fibers. A chosen isomorphism $\widehat{\Oring}_{X,x} \simeq \widehat{\Oring}_{Y,y}$ of complete local $\Lambda$-algebras induces an $I_K$-equivariant isomorphism
\[ (\textrm{R}\Psi_X\mathbb{Q}_{\ell})_x \overset{\sim}{\longrightarrow} (\textrm{R}\Psi_Y\mathbb{Q}_{\ell})_y\]  
\end{Prop}

\begin{proof} Let $(U,u) \to (X,x)$ be a pointed \'etale neighborhood, then by invariance of strict henselizations under \'etale maps we have $U_{(u)} \simeq X_{(x)}$ (see \stacksref{094Z}). Thus we have an isomorphism
\[ (\textrm{R}\Psi_X\mathbb{Q}_{\ell})_x \simeq \textrm{R}\Gamma(X_{(x)}\times_{\Lambda}\etabar, \mathbb{Q}_{\ell}) \simeq \textrm{R}\Gamma(U_{(u)}\times_{\Lambda} \etabar, \mathbb{Q}_{\ell}) \simeq (\textrm{R}\Psi_U\mathbb{Q}_{\ell})_u\]
compatible with the $I_K$-action on the stalks since the strict henselizations are defined over $\Spec(\Lambda)$. As complete noetherian $\Lambda$-algebras are $G$-rings, by Artin approximation (see \stacksref{0CAV}) we can find a common pointed \'etale neighborhood
\[ (X,x) \leftarrow (U,u) \rightarrow (Y,y)\]
and consequently we get an $I_K$-equivariant isomorphism $(\textrm{R}\Psi_X\mathbb{Q}_{\ell})_x \simeq (\textrm{R}\Psi_Y\mathbb{Q}_{\ell})_y$. By \cite[Thm.\ 3.1]{berkovich1996vanishing} this isomorphism is induced by $\widehat{\Oring}_{X,x} \simeq \widehat{\Oring}_{Y,y}$.\end{proof}

We remark that Proposition \ref{prop: isom nearby cycle stalks under formal nbd isom} also applies to stalks of vanishing cycles as the isomorphism above is compatible with specialization $\mathbb{Q}_{\ell} \to \textrm{R}\Psi_X\mathbb{Q}_{\ell}$, and moreover it suffices to have a common \'etale neighborhood of $(X,x)$ and $(Y,y)$ (cf.\ Corollary \ref{cor: henselian nbd of deformation}). The underlying point of the functoriality of the nearby cycles stalk isomorphism is that there is a canonical isomorphism
\[(\textrm{R}\Psi_X\mathbb{Q}_{\ell})_x \simeq \textrm{R}\Gamma(\mathcal{M}_x, \mathbb{Q}_{\ell})\]
where $\mathcal{M}_x:= (\Spf(\widehat{\Oring}_{X,x}))_{\eta} \widehat{\times}_K \widehat{\overline{K}}$ is the analytic Milnor fiber (in the sense of \cite[\S1]{berkovich1996vanishing}).

\subsection{$\ell$-adic Springer correspondences.}\label{section: ladic springer correspondence} As in Section \ref{section: notations and conventions for lie stuff} we fix a split, simply connected, simple group $G/\Lambda$ and a split maximal torus $T \subset G$, and set $\g = \textrm{Lie}(G), \h = \textrm{Lie}(T)$. In particular we retain Assumption \ref{assumptions on G g etc} regarding the residue characteristic. We also recall from Section \ref{relative-grothendieck-alteration} the construction of the adjoint bundle $\gres$ and the Grothendieck alteration $\pi: \gres \to \g$.

For a geometric point $\overline{t} \to \Lambda$ consider the pullbacks $\g_{\overline{t}} = \g \times_{\Lambda} \overline{t}, \gres_{\overline{t}} = \gres \times_{\Lambda} \overline{t}$ and $\pi_{\overline{t}}: \gres_{\overline{t}} \to \g_{\overline{t}}$ the classical Grothendieck alteration; $\g_{\overline{t}}$ is a Lie algebra over a field of zero or very good characteristic, as per our assumptions. We omit the subscript $(\cdot)_{\overline{t}}$ in this subsection for ease of notation, and assume $\g$ is a Lie algebra over $k=\overline{k}$.

\begin{Def} For $x \in \mathcal{N}$, the \emph{Springer fiber} over $x$ is $\mathcal{B}_x = \pi^{-1}(x)_{\textrm{red}} \subset \gres$.   
\end{Def}

By the incidence variety description of $\gres$ (see Remark \ref{remark: gres as incidence variety}) we have that $\mathcal{B}_x$ is the projective variety of Borels $B'$ so that $x \in \textrm{Lie}(B')$; in particular $\mathcal{B}_0 = G/B$.

Let $W$ be the Weyl group associated to the root lattice $\Phi = \Phi(G,T)$. There is a natural construction of Springer (\cite{springer1976trigonometric}, see also \cite[\S6]{shoji1988geometry}) which gives a correspondence between irreducible $\overline{\mathbb{Q}_{\ell}}$-representations of $W$ and nilpotent orbits in $\mathcal{N}$ equipped with an equivariant irreducible local system; under this correspondence, the cohomology of Springer fibers $\coh^i(\mathcal{B}_x, \mathbb{Q}_{\ell})$ is endowed with a $W$-action. The following sheaf-theoretic construction of this action is due to Borho--MacPherson (\cite[\S2.5-2.6]{borho}, see also \cite[\S4-\S6]{shoji1988geometry}):

Over the regular semisimple locus $\g^{\textrm{rs}}$ (Section \ref{section: nilpotent and reg loci}) the Grothendieck alteration is a $W$-torsor $\pi^{\textrm{rs}}: \gres^{\textrm{rs}} \to \g^{\textrm{rs}}$. More explicitly, isomorphism $\gres^{\textrm{rs}} \simeq G/T \times_{\Lambda} \h^{\textrm{rs}}$ equips the left hand side with a $W$-action by picking a representative $\widetilde{w} \in N_G(T)$ of $w$ and setting $a_w(gT,h) = (g\widetilde{w}^{-1}T, \textrm{Ad}(\widetilde{w})h)$ (Theorem \ref{thm: all the properties of simultaneous resolution diagrams} (ii) and (iii)). We remark that this formula does not depend on the choice of $\widetilde{w}$. Since $\textrm{Ad}(g\widetilde{w}^{-1})\textrm{Ad}(\widetilde{w})h = \textrm{Ad}(g)h$, the $W$-torsor structure on $\pi^{\textrm{rs}}$ is clear by $\pi^{\textrm{rs}}\circ a_w = \pi^{\textrm{rs}}$. Setting $\mathcal{L} := \pi_*^{\textrm{rs}}\mathbb{Q}_{\ell}$ we therefore have a natural map $\mathbb{Q}_{\ell}[W] \to \textrm{End}(\mathcal{L})$ via $w \cdot v(x) = v(a_w^{-1}x)$ on functions $v$ on its (finite) stalks.

Now $\pi: \gres \to \g$ is a small morphism (in the sense of \cite[Rem. 4.4]{shoji1988geometry}) and so setting $d = \dim\g$ we have (Prop. 4.2 \emph{loc.\! cit.})
\begin{equation}\label{eq: groth alt pushforward is ic sheaf}
\textrm{R}\pi_*\mathbb{Q}_{\ell}[d] \simeq j_{!*}(\pi_*^{\textrm{rs}}\mathbb{Q}_{\ell}[d]) =: \textrm{IC}_{\g}(\mathcal{L})  
\end{equation}
where $j: \g^{\textrm{rs}} \hookrightarrow \g$. In other words $\textrm{R}\pi_*\mathbb{Q}_{\ell}[d]$ is an IC (perverse) sheaf on $\g$, and full-faithfullness of the intermediate extension $j_{!*}$ extends every deck transformation of $W$ uniquely to $\textrm{R}\pi_*\mathbb{Q}_{\ell}[d]$. Proper base-change along $i_x: x \to \mathcal{N} \subset \g$ yields
\[(\textrm{R}\pi_*\mathbb{Q}_{\ell}[d])_x[-d] \simeq \textrm{R}\Gamma(\mathcal{B}_x, \mathbb{Q}_{\ell})\]
and the right-hand side is naturally equipped with the \emph{Springer} $W$-action; the convention here is that when $x$ is regular (whence $\mathcal{B}_x$ is discrete) the associated $W$-representation $\coh^0(\mathcal{B}_x, \mathbb{Q}_{\ell})$ is the trivial one.

We specialize to studying $W$-representations for $\mathcal{B}_x$ when $x$ is subregular: denote by $\alpha_1, \cdots, \alpha_r$ the simple roots of $\h$ (here $r = \textrm{rk}(\g)$). $W$ has a natural action on the character lattice $X^*(T)$ and its $\mathbb{Q}_{\ell}$-linear extension $V_{\ell}:= X^*(T)\otimes_{\mathbb{Z}}\mathbb{Q}_{\ell}$ via the reflecton representation: explicitly, the simple reflections $s_{\alpha_i}\in W$ yield $s_{\alpha_i}(\lambda) = \lambda - \langle \lambda, \alpha_i^{\vee}\rangle \alpha_i$ for $\lambda \in X^*(T)$.

Viewing $\lambda$ as a character $B \to T \to \mathbb{G}_m$, we set $k_{\-\lambda} :=k$ equipped with a $B$-action so that $b \in B$ acts as multiplication by $\lambda(b)^{-1}$. We form the associated (line) bundle
\begin{equation}\label{eq: assoc line bundle L lambda}
\mathcal{L}_{\lambda} := G \times^B k_{\lambda}
\end{equation}
It is a line bundle over $G/B$. If $P_i/B$ is the minimal parabolic associated to $\alpha_i$, we have $\mathcal{L}_{\lambda} \lvert_{P_i/B} \simeq \Oring_{\mathbb{P}^1}(\langle \lambda, \alpha_i^{\vee}\rangle)$ (cf.\! \cite[\S2.1]{brion2005frobenius}). Moreover  we have a natural map \[X^*(T) \to \Pic(\mathcal{B}_x), \;\; \lambda  \mapsto [\mathcal{L}_{\lambda}\lvert_{\mathcal{B}_x}]\] and thus by $\mathbb{Q}_{\ell}$-linear extension an $\ell$-adic Chern class map $c: V_{\ell} \to \coh^2(\mathcal{B}_x, \mathbb{Q}_{\ell}(1))$ sending $\lambda \otimes 1 \mapsto c_1(\mathcal{L}_{\lambda}\lvert_{\mathcal{B}_x})$.

\begin{Prop}\label{prop: springer w-action on cohom of B_x} Let $x$ be subregular, then the Chern class map $c: V_{\ell} \to \coh^2(\mathcal{B}_x, \mathbb{Q}_{\ell}(1))$ is a $W$-equivariant isomorphism, where the $W$-action on the source is the reflection representation.
\end{Prop}
\begin{proof} The Chern class map $\widetilde{c}: V_{\ell} \to \coh^2(G/B, \mathbb{Q}_{\ell}(1))$, $\lambda \mapsto [\mathcal{L}_{\lambda}]$ is $W$-equivariant with respect to the `classical' $W$-action on the target as defined in \cite[\S5.1]{shoji1988geometry}, and this action is equivalent to the Springer $W$-action by \cite[Prop.\! 5.4]{shoji1988geometry}. Moreover the proof of \emph{loc.\! cit} shows that $\coh^2(G/B, \mathbb{Q}_{\ell}) \to \coh^2(\mathcal{B}_x, \mathbb{Q}_{\ell})$ is $W$-equivariant; the composite is $c$ after a Tate twist, and $W$-equivariance of this map follows.

Now $\mathcal{B}_x = (\pi^{-1}(x))_{\textrm{red}} = \cup_{i \leq r} E_i$ is a union of projective lines $E_i$: we can choose a suitable transverse slice $\Sigma = x + \mathfrak{a}$ at $x$ (Section \ref{slodowy-slice-subsection}) so that $x \in \Sigma \cap \mathcal{N}$ is a rational double point of Dynkin type $\Phi$, and $\pi^{-1}(x)$ is the exceptional divisor of its minimal resolution (Proposition \ref{prop: slice simultaneous rdp resolution} (ii) and (iii)). Up to translating by $G$ we can view $E_i$ as $P_i/B$ (see Section \ref{section: groth alts for slices}). Then restriction to the irreducible components yields
\[ \coh^2(\mathcal{B}_x, \mathbb{Q}_{\ell}(1)) \overset{\sim}{\to} \bigoplus_{i=1}^r \coh^2(E_i, \mathbb{Q}_{\ell}(1)) \overset{\sim}{\to} \mathbb{Q}_{\ell}^r\]
where the maps on the right are trace maps. Since $\mathcal{L}_{\lambda} \lvert_{P_i/B} \simeq \Oring_{\mathbb{P}^1}(\langle \lambda, \alpha_i^{\vee}\rangle)$, map $c$ sends $\lambda \in V_{\ell} \mapsto \oplus_i c_1\left(\mathcal{L}_{\lambda}\lvert_{P_i/B}\right) = (\langle \lambda, \alpha_i^{\vee}\rangle)_{i \leq r}$ and by picking $\lambda$ generating $X^*(T)$ we see that $c$ maps linear basis to linear basis, hence is an isomorphism.\end{proof}

\subsection{Relative perversity of the Grothendieck--Springer sheaf.}\label{section: rel perv of grospri sheaf and springer w-action} We retain the setting of Sections \ref{subsection-classic-nearby-cycles}, \ref{section: ladic springer correspondence} from here and onwards, and consider constructible $\ell$-adic sheaves on $\gres$ and $\g$ viewed as $\Lambda$-schemes of relative dimension $d$. To interpolate between points $\overline{t} \to \Lambda$ of different characteristic, it will be useful to consider perverse sheaves in families over $\Lambda$ following Hansen--Scholze (\cite{hansen2021relative}).

\begin{Def} Let $X$ be a separated finite-type $\Lambda$-scheme. A sheaf $\mathcal{F} \in D_c^b(X, \mathbb{Q}_{\ell})$ is 

\begin{itemize}
    
\item[(i)] \emph{$\Lambda$-perverse} if for every geometric point $\overline{t} \to \Spec(\Lambda)$ we have $\mathcal{F}\lvert_{X_{\overline{t}}} \in \textbf{Perv}(X_{\overline{t}}, \mathbb{Q}_{\ell})$.

\item[(ii)] $\Lambda$-\emph{universally locally acyclic} (or $\Lambda$-ULA) if for every specialization\footnote{By specialization we mean a choice of map $\overline{t} \to \Spec\left(\Oring_{\Spec(\Lambda),\overline{s}}^{\textrm{sh}}\right)$, see \stacksref{0GJ2}.} $\overline{t} \rightsquigarrow \overline{s}$ in $\Spec(\Lambda)$ we have 
\begin{equation}\label{eq: ULA}\mathcal{F}_{\overline{x}} \overset{\sim}{\longrightarrow} \textrm{R}\Gamma\left(X_{(x)} \times_{\Spec(\Lambda)_{(s)}} \overline{t}, \mathcal{F}\right)\end{equation}
where $\overline{x}$ lives over $\overline{s}$, and moreover (\ref{eq: ULA}) holds after any base-change $T \to \Spec(\Lambda)$.

\end{itemize}
    
\end{Def}

The objects in (i) form the heart $\Perv(X/\Lambda)$ of the relative perverse $t$-structure on $D_c^b(X, \mathbb{Q}_{\ell})$; see \cite[Thm.\! 1.1]{hansen2021relative} for the definition of this category and this $t$-structure, which interpolates between the standard $t$-structure (when $X = \Lambda$) and the middle perverse $t$-structure (when $\Lambda = \Spec(\overline{k})$). In particular $\Perv(X_{\overline{t}}, \mathbb{Q}_{\ell})$ for schemes over field $k(\overline{t})$ is always defined in terms of the middle perverse $t$-structure.

The objects in (ii) are also well-defined and those satisfying both (i) and (ii) form an abelian subcategory $\Perv^{\textrm{ULA}}(X/\Lambda) \subset \Perv(X/\Lambda)$ (\cite[Thm.\! 1.6 (iii), Thm.\! 1.9]{hansen2021relative}).

\begin{Prop}\label{prop: grospri sheaf is rel perverse ic etc} Let $\pi: \gres \to \g$ be the Grothendieck alteration, $j: \g^{\textrm{rs}} \hookrightarrow \g$ the inclusion of the regular semisimple locus and $\mathcal{G}:= \textrm{\emph{R}}\pi_*\mathbb{Q}_{\ell}[d]$ the \emph{Grothendieck--Springer sheaf}.

\begin{itemize}
    \item[\emph{(i)}] $\mathcal{G}$ is $\Lambda$-perverse and $\Lambda$-universally locally acyclic.
    \item[\emph{(ii)}] If $\mathcal{L} = \pi^{\textrm{rs}}_*\mathbb{Q}_{\ell}$ then $\mathcal{G}$ is a relative IC extension of $\mathcal{L}[d]$ in the sense that for every $\overline{t} \to \Spec(\Lambda)$ we have $\mathcal{G}\lvert_{\g_{\overline{t}}} = (j_{\overline{t}})_{!*}(\mathcal{L}\lvert_{\g^{\textrm{rs}}_{\overline{t}}}[d])$.
    \item[\emph{(iii)}] We have an isomorphism $\textrm{\emph{End}}_{D_c^b(\g, \mathbb{Q}_{\ell})}(\mathcal{G}) \overset{\sim}{\longrightarrow} \textrm{\emph{End}}_{D_c^b(\g^{\textrm{rs}}, \mathbb{Q}_{\ell})}(j^*\mathcal{G})$.
\end{itemize}
    
\end{Prop}

\begin{proof} Since  $\gres / \Lambda$ is smooth, constant sheaf $\mathbb{Q}_{\ell} \in D_c^b(\gres)$ is $\Lambda$-ULA, and so by properness of $\pi$ we have that $\mathcal{G}$ is $\Lambda$-ULA as well (\cite[Prop.\! 3.3]{hansen2021relative}). 

For a geometric point $\overline{t} \to \Spec(\Lambda)$, by proper base-change we have $\mathcal{G}\lvert_{\g_{\overline{t}}} \simeq \textrm{R}\pi_{\overline{t}*}\mathbb{Q}_{\ell}[d]$, the usual Grothendieck--Springer sheaf associated with Grothendieck alterations over fields (see Section \ref{section: ladic springer correspondence}). By the construction in Section \ref{section: ladic springer correspondence} we have
\[\textrm{R}\pi_{\overline{t}*}\mathbb{Q}_{\ell}[d] \simeq (j_{\overline{t}})_{!*}(\mathcal{L}\lvert{\g_{\overline{t}}^{\textrm{rs}}}[d]), \;\;\; j_{\overline{t}}: \g_{\overline{t}}^{\textrm{rs}} \hookrightarrow \g_{\overline{t}}\]
This automatically makes $\mathcal{G}\lvert_{\g_{\overline{t}}}$ perverse (in fact an IC sheaf) so we get $\Lambda$-perversity of $\mathcal{G}$ and (ii).

For (iii), $\Spec(\Lambda)$ is integral and geometrically unibranch so by \cite[Thm.\!\!\! 1.10 (ii)]{hansen2021relative} the restriction functor $\Perv^{\textrm{ULA}}(\g/\Lambda) \to \Perv(\g_{\eta})$ is exact and fully-faithful, where $\eta = \Spec(\textrm{Frac}(\Lambda))$ is the generic point. Applying this fact to $\g$ and $\g^{\textrm{rs}}$ we get
\[\textrm{End}_{\g}(\mathcal{G}) \simeq \textrm{End}_{\g_{\eta}}(\mathcal{G}\lvert_{\g_{\eta}}) \overset{(*)}{\simeq} \textrm{End}_{\g_{\eta}^{\textrm{rs}}}(\mathcal{L}\lvert_{\g_{\eta}^{\textrm{rs}}}) \simeq \textrm{End}_{\g^{\textrm{rs}}}(\mathcal{L})\]

where $(*)$ stems from the fully-faithfulness of the usual intermediate extension functor $(j_\eta)_{!*}: \textrm{\textbf{Perv}}(\g_{\eta}^{\textrm{rs}}) \to \textrm{\textbf{Perv}}(\g_{\eta})$ (see Section \ref{section: ladic springer correspondence}); note $\mathcal{L}[d]$ is trivially $\Lambda$-ULA as a local system on the $\Lambda$-smooth scheme $\g^{\textrm{rs}}$. The composite is restriction along $j: \g^{\textrm{rs}} \hookrightarrow \g$ and $\mathcal{L}[d] = j^*\mathcal{G}$, whence we get (iii).
\end{proof}

Now let $e \in \mathcal{N}(\Lambda)$ be fiberwise subregular. As we are in very good residue characteristic, we may construct a relative homogeneous transverse slice $\Sigma$ at $e$ (Section \ref{slodowy-slice-subsection}) of relative dimension $r+2$. Following the notation of Section \ref{section: groth alts for slices} we set $i_{\Sigma}: \Sigma \hookrightarrow \g$, $j_{\Sigma}: \Sigma^{\textrm{rs}} = \Sigma \cap \g^{\textrm{rs}} \hookrightarrow \Sigma$ and $\pi_{\Sigma}: \widetilde{\Sigma} \overset{q}{\to} \Sigma \times_{\h/\!\!/W}\h \to \Sigma$ the slice-restricted Grothendieck alteration, with $\pi_{\Sigma}^{\textrm{rs}}$ its restriction to $\Sigma^{\textrm{rs}}$ (see Proposition \ref{prop: slice simultaneous rdp resolution}).

\begin{Cor}\label{cor: restricted grospri sheaf on slice perversity} With the above notations:

\begin{itemize}
    \item[\emph{(i)}] The restricted Grothendieck--Springer sheaf $\mathcal{G}_{\Sigma}:=\textrm{\emph{R}}(\pi_{\Sigma})_*\mathbb{Q}_{\ell}[r+2]$ is $\Lambda$-perverse, $\Lambda$-universally locally acyclic and a relative IC extension of $\mathcal{L}_{\Sigma}[r+2]:=(\pi_{\Sigma}^{\textrm{rs}})_*\mathbb{Q}_{\ell}[r+2]$.
    \item[\emph{(ii)}] We have an isomorphism $\textrm{\emph{End}}_{\Sigma}(\mathcal{G}_{\Sigma}) \overset{\sim}{\longrightarrow} \textrm{\emph{End}}_{\Sigma^{\textrm{rs}}}(\mathcal{L}_\Sigma[r+2])$.
   
\end{itemize}
    
\end{Cor}

\begin{proof} Note by proper base-change along $i_\Sigma$ we have $\mathcal{G}_\Sigma = i_\Sigma^*\mathcal{G}[r+2-d]$ where $\mathcal{G}$ is the Grothendieck--Springer sheaf. By the construction of $\Sigma$ we also have a smooth action map $\alpha: G \times \Sigma \to \g$ that restricts to smooth action maps $\alpha_{\overline{t}}: G_{\overline{t}} \times \Sigma_{\overline{t}} \to \g_{\overline{t}}$ along geometric points $\overline{t} \to \Lambda$ (cf.\! Lemma \ref{lemma: basechange of slice action map} and Proposition \ref{proposition: construction of relative homogeneous transverse slice}). 

As in Proposition \ref{prop: grospri sheaf is rel perverse ic etc}, $\mathcal{G}_{\Sigma}\lvert_{\Sigma_{\overline{t}}} \simeq \textrm{R}(\pi_{\Sigma})_{\overline{t}*}\mathbb{Q}_{\ell}[r+2]$ where $(\pi_{\Sigma})_{\overline{t}}: \widetilde{\Sigma}_{\overline{t}} \to \Sigma_{\overline{t}}$ is the classical Grothendieck alteration $\pi_{\overline{t}}$ restricted to slice $\Sigma_{\overline{t}}$. We check the smallness of $(\pi_{\Sigma})_{\overline{t}}$: put $S_n = \{g \in \g_{\overline{t}} \mid \dim \pi_{\overline{t}}^{-1}(g) \geq n\}$ then since $\pi_{\overline{t}}$ is a small morphism we have for $n\geq 1$ that $\textrm{codim}_{\g_{\overline{t}}}(S_n) > 2n$. As $S_n$ is closed and $G_{\overline{t}}$-invariant and $\alpha_{\overline{t}}$ is smooth, we have
\[\begin{aligned}\textrm{codim}_{\Sigma_{\overline{t}}}(\Sigma_{\overline{t}} \cap S_n) = \textrm{codim}_{G_{\overline{t}} \times \Sigma_{\overline{t}}}(G_{\overline{t}} \times (\Sigma_{\overline{t}} \cap S_n)) \\ = \textrm{codim}_{G_{\overline{t}} \times \Sigma_{\overline{t}}}(\alpha_{\overline{t}}^{-1}(S_n)) \geq \textrm{codim}_{\g_{\overline{t}}}(S_n) > 2n\end{aligned}\]
and we deduce that $(\pi_\Sigma)_{\overline{t}}$ is a small morphism as well. Moreover $\widetilde{\Sigma}$ is a smooth $\Lambda$-scheme as the composite of smooth maps $\widetilde{\Sigma} \overset{\widetilde{\chi}_{\Sigma}}{\to} \h \to \Lambda$ (see Proposition \ref{prop: slice simultaneous rdp resolution}).

Thus the proof of Proposition \ref{prop: grospri sheaf is rel perverse ic etc} carries verbatim to this setting as well: $\mathcal{G}_{\Sigma}$ is $\Lambda$-ULA, $\Lambda$-perverse and $\mathcal{L}_{\Sigma}[r+2] = j_{\Sigma}^*\mathcal{G}_{\Sigma}$ yields fiberwise $\mathcal{G}_{\Sigma}\lvert_{\Sigma_{\overline{t}}} = \textrm{IC}(\mathcal{L}_{\Sigma}\lvert_{\Sigma_{\overline{t}}})$ so we get (i) and (ii) as before.\end{proof}

The upshot of Proposition \ref{prop: grospri sheaf is rel perverse ic etc} and Corollary \ref{cor: restricted grospri sheaf on slice perversity} is that the $W$-action via deck transformations on $(\pi_{\overline{t}}^{\textrm{rs}})_*\mathbb{Q}_{\ell}$ (for $\overline{t} \to \Lambda$ a geometric point) naturally and uniquely extends to a $W$-action on the Grothendieck--Springer sheaf $\mathcal{G}$, which induces the Springer $W$-action on its stalks at geometric points (cf.\! Section \ref{section: ladic springer correspondence}). Thus:

\begin{Cor}\label{cor: springer w-action on cohom of springer fibers} The $W$-action on the Grothendieck--Springer sheaf $\mathcal{G}$ coming from the $W$-torsor $\pi^{\textrm{\emph{rs}}}: \gres^{\textrm{\emph{rs}}} \to \g^{\textrm{\emph{rs}}}$ induces the Springer $W$-action on $\coh^2(\mathcal{B}_{e_t}, \mathbb{Q}_{\ell}(1))$ as in \emph{Proposition \ref{prop: springer w-action on cohom of B_x}}, where $e \to \mathcal{N}$ is fiberwise subregular and $\overline{t} \to \Lambda$ is any geometric point.  
\end{Cor}

\subsection{Monodromy--Springer $W$-actions.}\label{section: monodromy springer w-actions} In this section we consider the monodromy $I_K$-action on the stalks of nearby cycles at a rational double point, and compare it to the Springer $W$-action as defined in Section \ref{section: rel perv of grospri sheaf and springer w-action}. 

Let $X/\Lambda$ be a surface degeneration with $X_k$ containing a unique rational double point $x$. For the corresponding $\g/\Lambda$ and the fiberwise subregular section $e\colon \Spec(\Lambda) \to \mathcal{N}$, construct transverse slice $\Sigma$ as in Section \ref{slodowy-slice-subsection} and identify the singularity germ $(X_k,x)$ with germ $(\Sigma_k\cap \mathcal{N}_k,e_k)$. By miniversality of the \emph{completed} family $\widehat{\chi}_\Sigma$ (at $e_k$ in the source and at the origin in the target, see Theorem \ref{thm: equivariant miniversality of slice deformation and resolution}), there exists a classifying map with the properties\footnote{A priori, formal miniversality gives a map $\widehat{\phi}: \Spf(\Lambda) \to \widehat{\h/\!\!/W}\simeq \Spf(\Lambda[\![t_1,\cdots, t_r]\!])$, which by restriction along $\Lambda[t_1,\cdots, t_r]$ induces the algebraic map $\phi$. Since $X_K$ is smooth, $\phi(\Spec(K))$ lies outside the discriminant locus in $\h/\!\!/W$, i.e.\ it lies in $(\h/\!\!/W)^{\textrm{rs}}$. }
\[\phi: \Spec(\Lambda) \to \h/\!\!/W, \;\;\; \phi(\Spec(k)) = 0, \; \phi(\Spec(K)) \in (\h/\!\!/W)^{\textrm{rs}}\]

Write $Y\coloneqq \Sigma \times_{\phi} \Spec(\Lambda)$, then for $y \in Y_k$ mapping to $e_k \in \Sigma$ we have by construction $\widehat{\Oring}_{X,x} \simeq \widehat{\Oring}_{Y,y}$, so Proposition \ref{prop: isom nearby cycle stalks under formal nbd isom} yields
\begin{equation}\label{eq: formal nbd isom between X and pullback of RDP model}
(\textrm{R}\Psi_X\mathbb{Q}_{\ell})_x \overset{\sim}{\longrightarrow} (\textrm{R}\Psi_Y\mathbb{Q}_{\ell})_y   
\end{equation}
$I_K$-equivariantly. Note $Y_K$ is smooth, and $Y/\Lambda$ is nonsmooth only at $y$.

Set also $Z\coloneqq \Spec(K) \times_{\h/\!\!/W}\h^{\textrm{rs}} \simeq \Spec(A)$; $A$ is a finite \'etale $K$-algebra of dimension $|W|$ since $\psi: \h^{\textrm{rs}}\to (\h/\!\!/W)^{\textrm{rs}}$ is a $W$-torsor (Theorem \ref{thm: all the properties of simultaneous resolution diagrams} (ii)). Note $Z_{\overline{K}} \simeq \coprod_{|W|} \Spec(\overline{K})$ and so a point $a \in Z(\overline{K})$ identifies this set with $W$.

\begin{Constr}\label{def: rho hom}We define the following homomorphisms $\rho_a: I_K \to W$ for $a \in Z(\overline{K})$: there is a freely transitive $W$-action on $Z(\overline{K})$ induced by the geometric Weyl action on $\h^{\textrm{rs}}$ via $w \cdot h = \textrm{Ad}(\widetilde{w})(h)$ for $h \in \h^{\textrm{rs}}(\overline{K}), \widetilde{w} \in N_G(T)$ lift of $w$ (Theorem \ref{thm: all the properties of simultaneous resolution diagrams} (iii)). Denote this action by $w \cdot_1 a \in Z(\overline{K})$. There is another natural action on $Z(\overline{K})$ by $I_K$, the usual continuous Galois action on the $\overline{K}$-points of a finite \'etale $K$-scheme, which we denote by $g \cdot_2 a$.  Since $G_K$ and $T_K$ are split, the Weyl action is defined over $K$, so the two actions commute: $w \cdot_1 (g \cdot_2 a) = g \cdot_2(w \cdot_1 a)$. Set
\[\rho_a: I_K \to W, \;\;\; g \cdot_2 a = \rho_a(g)^{-1} \cdot_1 a\]
It is straightforward to check $\rho_a$ is a well-defined homomorphism and a different choice of $a \in Z(\overline{K})$ gives a conjugate of $\rho_a$.

\end{Constr}

\begin{Prop}\label{prop: simult resol over Lambda-L} Let $Y, Z$ and $\rho_a$ be as above. Set $L\coloneqq (\overline{K})^{\textrm{\emph{ker}}\rho_a}$ and let $\Lambda_L$ be the normalization of $\Lambda$ in $L/K$. Then $a \in Z(\overline{K})$ extends to a section $a: \Spec(\Lambda_L) \to \h$ and the pullback of the restricted Grothendieck--Springer resolution $q: \widetilde{\Sigma} \to \Sigma \times_{\h/\!\!/W}\h$ along $a$ yields a simultaneous resolution \[q_a: \widetilde{Y}_a \to Y\times_{\Lambda} \Lambda_L\]
over $\Lambda_L$. On the special fibers $q_{a,k}$ is a resolution of $Y_k$ with exceptional divisor $E \simeq \mathcal{B}_{e_k}$.
    
\end{Prop}

\begin{proof} Note that $L/K$ is a finite Galois extension of degree $\leq |W|$ since $\ker(\rho_a)$ is an open normal subgroup, and by conjugation $L$ does not depend on $a$. Since $g \in \ker(\rho_a)$ if and only if $g\cdot_2 a = a$, we have that $a: \Spec(\overline{K}) \to \h^{\textrm{rs}}$ factors through $a_L: \Spec(L) \to \h^{\textrm{rs}}$, and by the valuative criterion of properness $a_L$ extends to a section $a_{\Lambda_L}: \Spec(\Lambda_L) \to \h$ with $a_{\Lambda_L}(\Spec(k)) = 0$. Abusing notation, we replace $a_{\Lambda_L}$ with $a$.

Set $\widetilde{Y}_a = \widetilde{\Sigma}\times_a \Spec(\Lambda_L)$ and note that the generic fiber $(\widetilde{Y}_a)_L$ is smooth as we base-changed the smooth morphism $\widetilde{\chi}_{\Sigma}: \widetilde{\Sigma} \to \h$. By naturality, there is also a canonical identification $\widetilde{Y}_a \simeq \widetilde{\Sigma} \times_{\Sigma\times_{\h/\!\!/W}h}(Y\times_{\Lambda}{\Lambda_L})$ and so the universal property of cartesian diagrams yields a map $q_a: \widetilde{Y}_a \to Y\times_{\Lambda} \Lambda_L$ that is the base-change of map 
\[q = (\pi_{\Sigma}, \widetilde{\chi}_{\Sigma}): \widetilde{\Sigma} \to \Sigma\times_{\h/\!\!/W} \h\]

Proposition \ref{prop: slice simultaneous rdp resolution} (iii) then yields that $q_a$ is a simultaneous resolution over $\Spec(\Lambda_L)$ (cf.\ Corollary \ref{cor: grospri resol over dvr} too), i.e.\ over the generic fiber it is an isomorphism and over the special fiber it is the minimal resolution \[q_{a,k}\colon (\widetilde{Y}_a)_k \to Y_k \simeq \chi^{-1}_{\Sigma_k}(0) = \Sigma_k \cap \mathcal{N}_k\] In other words, the exceptional divisor is $E \simeq \mathcal{B}_{e_k}$ (cf.\ proof of Proposition \ref{prop: springer w-action on cohom of B_x}).\end{proof}

Together with the $I_K$-equivariant isomorphism in (\ref{eq: formal nbd isom between X and pullback of RDP model}) we have:

\begin{Cor}\label{cor: gamma isom of nearby cycle stalks} The chosen lift $a: \Spec(\overline{K}) \to \h^{\textrm{rs}}$ yields an $I_L$-equivariant isomorphism
\begin{equation}\label{eq: IL-equiv isom of nearby cycle stalks}\gamma_a:(\textrm{\emph{R}}\Psi_X\mathbb{Q}_{\ell})_x \overset{\sim}{\longrightarrow} \textrm{\emph{R}}\Gamma(\mathcal{B}_{e_k}, \mathbb{Q}_{\ell}) \simeq \mathbb{Q}_{\ell}[0] \oplus V_\ell(-1)[-2]\end{equation}
where $V_{\ell} = X^*(T)\otimes_{\mathbb{Z}}\mathbb{Q}_{\ell}$ and the $I_L$-action on the source and target is trivial.
\end{Cor}

\begin{proof} Retaining the notation from Proposition \ref{prop: simult resol over Lambda-L}, set $Y' = Y\times_{\Lambda} \Lambda_L$ so that $Y_k' = Y_k$. Invariance of geometric nearby cycles under change of trait (\cite[Th.\ 3.7]{deligne1977cohomologie}) and Proposition \ref{prop: isom nearby cycle stalks under formal nbd isom} give
\[\textrm{Res}^{I_K}_{I_L}((\textrm{R}\Psi_X\mathbb{Q}_{\ell})_x) \simeq (\textrm{R}\Psi_{Y'}\mathbb{Q}_{\ell})_y\]
By proper base-change and smoothness of $\widetilde{Y}_a/\Lambda_L$, Proposition \ref{prop: simult resol over Lambda-L} gives $I_L$-equivariant isomorphisms
\[(\textrm{R}\Psi_{Y'}\mathbb{Q}_{\ell})_y \simeq (\textrm{R}(q_{a,k})_*\textrm{R}\Psi_{\widetilde{Y}_a}\mathbb{Q}_{\ell})_y \simeq (\textrm{R}(q_{a,k})_*\mathbb{Q}_{\ell})_y \simeq \textrm{R}\Gamma(\mathcal{B}_{e_k}, \mathbb{Q}_{\ell})\]
where $q_{a,k}: (\widetilde{Y}_a)_k \to Y_k$ is the minimal resolution and we identify the exceptional divisor with Springer fiber $\mathcal{B}_{e_k}$. The right isomorphism in (\ref{eq: IL-equiv isom of nearby cycle stalks}) follows from Proposition \ref{prop: springer w-action on cohom of B_x}.\end{proof}

To derive a description of the full $I_K$-action on the right hand side of (\ref{eq: IL-equiv isom of nearby cycle stalks}), we proceed as follows. Write $\phi_Y: Y \to \Sigma$ for the map induced by $\phi$ and set $\widetilde{Y}\coloneqq Y\times_{\Sigma} \widetilde{\Sigma}$ for the base-change of the restricted Grothendieck alteration $\pi_{\Sigma}$ along $\phi_Y$. By the universal property there is a map $\pi_Y: \widetilde{Y} \to Y$ of $\Lambda$-schemes, so that over $\overline{K}$ it is $\widetilde{Y}_{\overline{K}}\simeq \coprod_{a \in Z(\overline{K})} Y_{\overline{K}} \to Y_{\overline{K}}$ and over $k$ it is the minimal resolution of $Y_k$ after taking $(\widetilde{Y}_k)_{\textrm{red}}$ on the source. We have
\[ \mathcal{F} \coloneqq \textrm{R}(\pi_Y)_*\mathbb{Q}_{\ell} \simeq \phi_Y^*\textrm{R}(\pi_\Sigma)_*\mathbb{Q}_{\ell} = \phi_Y^*\mathcal{G}_{\Sigma}[-r-2]\]
by proper base-change. Restricting to the special fiber $Y_k$ we have specialization map (Definition \ref{def: nearby cycles} (ii))
\[\mathcal{F}\lvert_{Y_k} \longrightarrow \textrm{R}\Psi_Y(\mathcal{F}\lvert_{Y_{K}}) \]
Taking the stalk at $y\in Y$ and noting that $\mathcal{F}_y \simeq \textrm{R}\Gamma(\mathcal{B}_{e_k}, \mathbb{Q}_{\ell})$ (Proposition \ref{prop: simult resol over Lambda-L}) and $\mathcal{F}\lvert_{Y_{\overline{K}}} \simeq \mathbb{Q}_{\ell, Y_{\overline{K}}} \otimes \mathbb{Q}_{\ell}^{Z(\overline{K})}$ (where $Z(\overline{K})$ is a discrete set of cardinality $|W|$) we get a map
\[\textrm{sp}\colon \textrm{R}\Gamma(\mathcal{B}_{e_k}, \mathbb{Q}_{\ell}) \longrightarrow (\textrm{R}\Psi_Y\mathbb{Q}_{\ell})_y \otimes_{\mathbb{Q}_{\ell}} \mathbb{Q}_{\ell}^{Z(\overline{K})}\]

For the following lemma, we recall the actions of $W$ and $I_K$ on $Z(\overline{K})$ as in Construction \ref{def: rho hom}. They respectively induce a $W$-action and an $I_K$-action on $(\textrm{R}\Psi_Y\mathbb{Q}_{\ell})_y \otimes_{\mathbb{Q}_{\ell}} \mathbb{Q}_{\ell}^{Z(\overline{K})}$ as follows: after identifying the latter with $\prod_{a \in Z(\overline{K})} (\textrm{R}\Psi_Y\mathbb{Q}_{\ell})_y$ and indexing elements by $v_a$, the $W$-action is $w\cdot (v_a)_a \coloneqq (v_{w^{-1}\cdot_1 a})_a$ and the $I_K$-action is $g\cdot (v_a)_a\coloneqq (g(v_{g^{-1}\cdot_2 a}))_a$ where $g \in I_K$ acts as the monodromy on the right hand side.

\begin{Lemma}\label{lemma: specialization is w and IK equivariant also delta inverse to gamma} With notations as above:

\begin{itemize} \item[\emph{(i)}] Specialization map $\textrm{\emph{sp}}$ is $W$-equivariant and $I_K$-equivariant, where the source is equipped with the Springer $W$-action from \emph{Corollary \ref{cor: springer w-action on cohom of springer fibers}} and the trivial $I_K$-action, and the target is equipped with the $W$-action and the $I_K$-action specified above.
\item[\emph{(ii)}] Choose $a \in Z(\overline{K})$ and set \[\delta_a = (\textrm{\emph{id}}_{\textrm{\emph{R}}\Psi} \otimes \textrm{\emph{ev}}_a) \circ \textrm{\emph{sp}} \colon \textrm{\emph{R}}\Gamma(\mathcal{B}_{e_k}, \mathbb{Q}_{\ell}) \longrightarrow (\textrm{\emph{R}}\Psi_Y\mathbb{Q}_{\ell})_y \otimes_{\mathbb{Q}_{\ell}} \mathbb{Q}_{\ell}^{Z(\overline{K})} \longrightarrow (\textrm{\emph{R}}\Psi_Y\mathbb{Q}_{\ell})_y\]
where $\textrm{\emph{ev}}_a$ specifies a component of $\mathbb{Q}_{\ell}^{Z(\overline{K})} \simeq \mathbb{Q}_{\ell}^{|W|}$. Then $\delta_a$ is the inverse of isomorphism $\gamma_a$ from \emph{Corollary \ref{cor: gamma isom of nearby cycle stalks}}. 
\end{itemize}
\end{Lemma}
\begin{proof} For (i), $I_K$-equivariance follows from the naturality of the specialization (Definition \ref{def: nearby cycles}) and $W$-equivariance follows from the $W$-action on $\mathcal{F}$ induced from Corollary \ref{cor: restricted grospri sheaf on slice perversity}.

For (ii) we implicitly identify $(\textrm{R}\Psi_X\mathbb{Q}_{\ell})_x \simeq (\textrm{R}\Psi_Y\mathbb{Q}_{\ell})_y$ as in the proof of Corollary \ref{cor: gamma isom of nearby cycle stalks}. By Proposition \ref{prop: simult resol over Lambda-L}, $a \in Z(\overline{K})$ extends to a section $a: \Spec(\Lambda_L) \to \h$ where $L$ is the fixed field of $\ker\rho_a \subset I_K$. Over $Y\times_{\Lambda} \Lambda_L$ we have an inclusion $\widetilde{Y}_a\coloneqq \widetilde{\Sigma} \times_a \Spec(\Lambda_L) \hookrightarrow \widetilde{Y}\times_\Lambda \Lambda_L$ and adjunction gives
\[\mathcal{F}\lvert_{Y\times_{\Lambda}\Lambda_L} \longrightarrow \textrm{R}(q_a)_*\mathbb{Q}_{\ell}, \;\;\; q_a: \widetilde{Y}_a \to Y\times_{\Lambda} \Lambda_L\]
where $q_a$ is the simultaneous resolution of Proposition \ref{prop: simult resol over Lambda-L}. Over $Y_{\overline{K}}$ this map is just $\textrm{ev}_a$ and over $Y_k$ it is an isomorphism as $\pi_{Y,k}$ and $q_{a,k}$ are the same minimal resolutions after passing to the reduced subscheme structure. The specialization map for $\textrm{R}(q_a)_*\mathbb{Q}_{\ell}\lvert_{Y_k} \simeq \textrm{R}(q_{a,k})_*\mathbb{Q}_{\ell}$ is
\[\textrm{sp}': \textrm{R}(q_{a,k})_*\mathbb{Q}_{\ell} \longrightarrow \textrm{R}\Psi_{Y'}(\textrm{R}(q_a)_*\mathbb{Q}_{\ell}\lvert_{Y_{\overline{K}}}) \simeq \textrm{R}\Psi_{Y'}\mathbb{Q}_{\ell} \simeq \textrm{Res}_{I_L}^{I_K}\textrm{R}\Psi_Y\mathbb{Q}_{\ell} \]
as $q_a$ is an isomorphism over $\Spec(\overline{K})$. By naturality of specialization maps, after taking stalks at $y$ we have
\begin{center}
    \begin{tikzcd}
        \textrm{R}\Gamma(\mathcal{B}_{e_k},\mathbb{Q}_{\ell}) \arrow[r, "\textrm{sp}"] \arrow[d, "\simeq"] &  (\textrm{R}\Psi_{Y}\mathbb{Q}_{\ell})_y \otimes_{\mathbb{Q}_\ell} \mathbb{Q}_{\ell}^{Z(\overline{K})}\arrow[d, "\textrm{id} \otimes \textrm{ev}_a"] \\ (\textrm{R}(q_a)_*\mathbb{Q}_{\ell})_y \arrow[r, "\textrm{sp}_y'"]& (\textrm{R}\Psi_Y\mathbb{Q}_{\ell})_y
    \end{tikzcd}
\end{center}
The left column and bottom row of the diagram define $\gamma_a^{-1}$ while the top row and right column define $\delta_a$, which are then the same by commutativity.\end{proof}

\begin{Prop}\label{prop: springer and monodromy IK actions coincide} Let $\coh^2(\gamma_a): (\textrm{\emph{R}}^2\Psi_Y\mathbb{Q}_{\ell})_y \to \coh^2(\mathcal{B}_{e_k},\mathbb{Q}_{\ell})$ be the isomorphism induced from $\gamma_a$, then it identifies the monodromy $I_K$-action with the Springer $W$-action via $\rho_a: I_K \to W$ in the sense that
\begin{equation}\label{eq: springer and monodromy action on elts}\coh^2(\gamma_a)(g \cdot v) = \rho_a(g)\cdot \coh^2(\gamma_a)(v), \;\; g\in I_K, v \in (\textrm{R}^2\Psi_Y\mathbb{Q}_{\ell})_y\end{equation}
where $g\cdot(-)$ denotes the monodromy action and $\rho_a(g)\cdot(-)$ denotes the Springer action.
\end{Prop}
\begin{proof} We first explain two compatibility relations on maps \[\coh^2(\delta_b): \coh^2(\mathcal{B}_{e_k},\mathbb{Q}_{\ell}) \to (\textrm{R}^2\Psi_Y\mathbb{Q}_{\ell})_y, \;\;\;b \in Z(\overline{K})\]
with the understanding that $\delta_b$ maps to the copy of $(\textrm{R}\Psi_Y\mathbb{Q}_{\ell})_y$ indexed by $b$. By $I_K$-equivariance of the specialization map (where $I_K$ acts trivially on the source, see Lemma \ref{lemma: specialization is w and IK equivariant also delta inverse to gamma}), \[g \cdot \coh^2(\delta_{g^{-1}\cdot_2 b})(v) = \coh^2(\delta_b)(v),\;\;\;v \in \coh^2(\mathcal{B}_{e_k},\mathbb{Q}_{\ell})\] so setting $a = g^{-1}\cdot_2 b$ yields $g \circ \delta_a = \delta_{g\cdot_2 a}$ for the monodromy $I_K$-action on $(\textrm{R}\Psi_Y\mathbb{Q}_{\ell})_y$. Similarly, $W$-equivariance of the specialization map from Lemma \ref{lemma: specialization is w and IK equivariant also delta inverse to gamma} yields relation $\delta_a \circ w = \delta_{w^{-1}\cdot_1 a}$. Thus
\[g \circ \delta_a = \delta_{g\cdot_2 a} = \delta_{\rho_a(g)^{-1}\cdot_1 a} = \delta_a \circ \rho_a(g)\]
by definition of $\rho_a$. Since $\delta_a = \gamma_a^{-1}$ by Lemma \ref{lemma: specialization is w and IK equivariant also delta inverse to gamma}, we obtain $\coh^2(\gamma_a)\circ g = \rho_a(g) \circ \coh^2(\gamma_a)$, which yields (\ref{eq: springer and monodromy action on elts}).\end{proof}

Along with the $I_K$-equivariant isomorphism $(\textrm{R}^2\Psi_X\mathbb{Q}_{\ell})_x \simeq (\textrm{R}^2\Psi_Y\mathbb{Q}_{\ell})_y$, the upshot is that the natural monodromy $I_K$-action on the left-hand side is exactly the restriction of the Springer $W$-action along a map $\rho_a: I_K \to W$. By Corollary \ref{cor: gamma isom of nearby cycle stalks} $I_L=\ker \rho_a$ acts trivially on $(\textrm{R}\Psi_X\mathbb{Q}_{\ell})_x$ so $\rho_a$ factors through \[\overline{\rho}_a: I_K/I_L\simeq \Gal(L/K) \hookrightarrow W\]
By definition of $\rho_a$, picking another element $b = w\cdot_1 a \in Z(\overline{K})$ yields $\overline{\rho}_b(g) = w \overline{\rho}_a(g) w^{-1}$ and so $\rho_a, \rho_b$ realize $\Gal(L/K)$ as conjugate subgroups in $W$. In other words:

\begin{Cor}\label{cor: what the monodromy depends on} The monodromy action on $(\textrm{R}^2\Psi_X\mathbb{Q}_{\ell})_x$ factors through an embedding $\overline{\rho}_a: \Gal(L/K) \to W$ and is the restriction of Springer's reflection $W$-representation to a subgroup whose conjugacy class depends only on $\phi: \Spec(\Lambda) \to \h/\!\!/W$.
    
\end{Cor}

\subsection{Proof of the main theorem.} We can finally put the pieces together in order to prove Theorem \ref{thm: main thm in intro}, re-phrased below. We retain the notations of Section \ref{subsection-classic-nearby-cycles}, \ref{section: ladic springer correspondence} and \ref{section: monodromy springer w-actions}.

\begin{Thm}\label{thm: main theorem about rdps}Let $X /\Lambda$ be a surface degeneration with $X_k$ having a unique rational double point $x$ so that $p$ is good for $x$. Let $\Phi, W$ be the root lattice associated to $x$ and its Weyl group, and set $V_{\ell} = \Phi^{\vee}\otimes_{\mathbb{Z}}\mathbb{Q}_\ell$.
\begin{itemize}
    \item[\emph{(i)}] The monodromy $I_K$-action on $\coh^i(X_{\overline{K}},\mathbb{Q}_{\ell})$ is trivial for $i\neq 2$. For $i=2$ let $L$ be the fixed field of the kernel of the $I_K$-action, then we have an $I_K$-equivariant splitting
    \begin{equation}\label{eq: equivariant splitting}\coh^2(X_{\overline{K}},\mathbb{Q}_{\ell}) \simeq \coh^2(X_k,\mathbb{Q}_{\ell}) \oplus V_{\ell}(-1)\end{equation}
    The monodromy $I_K$-action is trivial on the first summand, and on the second it factors through $\Gal(L/K)$, where it is identified with the restriction of the reflection $W$-representation along an embedding $\Gal(L/K) \hookrightarrow W$.

    \item[\emph{(ii)}] $L$ is the minimal Galois extension over which $X_K$ has good reduction in the category of algebraic spaces. Moreover the isomorphism class of the $I_K$-representation $\coh^2(X_{\overline{K}},\mathbb{Q}_{\ell})$ depends, up to trivial summands, only on a formal $\Lambda$-isomorphism class of $(X,x)$. In particular $L$ depends only on this class. 
    \end{itemize}   
\end{Thm}

\begin{Rem}\label{remark: variants of main rdp theorem} We make several comments on the formulation of Theorem \ref{thm: main theorem about rdps}.

\begin{itemize}

\item[(i)]A uniform characteristic that works is $p>5$; the bound can be lower if we know the Dynkin type of $x$. Moreover, the local monodromy $I_K$-action on $V_{\ell}$ and extension $L$ are invariant under passage to an \'etale neighborhood of $(X,x)$.

\item[(ii)] The `up to trivial summands' condition is important:\ blowing up a section of $X$ disjoint from $x$ contributes a trivial $I_K$-representation $\mathbb{Q}_{\ell}(-1)$ as a summand in $\coh^2(X_{\overline{K}}, \mathbb{Q}_{\ell})$ without changing the nontrivial part of the monodromy; in particular, if we fix $\dim \coh^2(X_{\overline{K}},\mathbb{Q}_{\ell})$ we have that a formal $\Lambda$-germ of $x$ determines the monodromy representation up to isomorphism. 

\item[(iii)] The statement generalizes to $X_k$ having rational double points $x_1, \cdots, x_m$. If $\Phi_i$ is the Dynkin type of $x_i$ with Weyl group $W_i$ and $V_i$ is the $\mathbb{Q}_{\ell}$-vector space reflection representation of $W_i$, set $W = \prod W_i$ and $V_{\ell} = \bigoplus V_i$. Assuming $p$ is good for each $\Phi_i$, we have an $I_K$-equivariant splitting $\coh^2(X_{\overline{K}},\mathbb{Q}_{\ell})\simeq \coh^2(X_k,\mathbb{Q}_{\ell})\oplus V_{\ell}(-1)$ with the monodromy action being trivial on the first summand, while on the second summand it is the restriction of the reflection representation of $W$ along \[I_K \to \Gal(L/K)\hookrightarrow W\] Here $L$ is the compositum of $L_i$ where $L_i$ is the monodromy extension associated to $x_i$ as in (i).

\end{itemize}
\end{Rem}

Before the proof, we construct a smooth model for $X$ after base extension. Recall the homomorphism $\rho_a: I_K \to W$ (Construction \ref{def: rho hom}) which yields a finite Galois extension $L/K$ as in Proposition \ref{prop: simult resol over Lambda-L}.

\begin{Lemma}\label{lemma: global simult resol of X as algebraic space} There exists a proper morphism of algebraic spaces $\widetilde{X} \to X\times_{\Lambda}\Lambda_L$ with $\widetilde{X}$ smooth over $\Lambda_L$. The map is an isomorphism outside of $x$ and over the special fiber $\widetilde{X}_k \to X_k$ is a minimal resolution of $X_k$.
\end{Lemma}
\begin{proof} Take $(Y,y)$ as in the setup of Section \ref{section: monodromy springer w-actions}. By Proposition \ref{prop: isom nearby cycle stalks under formal nbd isom} there is a common pointed \'etale neighborhood \[(X,x) \leftarrow (U,u) \rightarrow (Y,y)\] and up to shrinking we can assume the preimage of $x$ and $y$ is just $u$. Pull back the simultaneous resolution $q_a: \widetilde{Y}_a \to Y \times_{\Lambda} \Lambda_L$ of Proposition \ref{prop: simult resol over Lambda-L} to $\widetilde{U} \to U\times_{\Lambda} \Lambda_L$; $\widetilde{U}$ is \'etale over $\widetilde{Y}_a$ hence smooth over $\Lambda_L$, and this modification is an isomorphism outside $u$. By \stacksref{0BGY} we can descend this modification along $U\times_{\Lambda}\Lambda_L \to X\times_{\Lambda}\Lambda_L$ and glue it with the identity modification outside $x$ to obtain a proper modification $\widetilde{X} \to X\times_{\Lambda} \Lambda_L$. By construction $\widetilde{X}$ is a proper smooth algebraic space, and since the special fiber is the minimal resolution over the pointed \'etale neighborhood $(U,u)$ and the modification is the identity otherwise, $\widetilde{X}_k$ is the minimal resolution of $X_k$.\end{proof}

We also need to consider separately the case of $p$ good but not very good, i.e. $p\!\mid\!n$ when $x$ is an $A_{n-1}$-singularity; the upshot here is that we can repeat the slice arguments with an appropriate slice inside $\mathfrak{gl}_n$ instead of $\mathfrak{sl}_n$.

\begin{Lemma}\label{lemma: p not very good for An} In the setting of \emph{Theorem \ref{thm: main theorem about rdps}}, suppose $p\!\mid\!n$ and $x \in X_k$ is an $A_{n-1}$ singularity. Let $\g = \mathfrak{gl}_n = \textrm{\emph{Lie}}(\textrm{\emph{GL}}_{n,\Lambda})$, $\h$ the diagonal Cartan algebra, $W= S_n$ and $V_{\ell} = \mathbb{Q}_{\ell}^n/\mathbb{Q}_{\ell}(1,\cdots,1)$ the quotient by the diagonal embedding $\mathbb{Q}_{\ell}\hookrightarrow \mathbb{Q}_{\ell}^n$.

\begin{itemize}
\item[\emph{(i)}] There exists a homogeneous transverse slice $\Sigma$ at a fiberwise subregular section such that $\chi_{\Sigma}: \Sigma \to \h/\!\!/W$ induces a formal miniversal deformation of the $A_{n-1}$ singularity, and moreover $\widetilde{\Sigma} \to \h$ is smooth and $\widetilde{\Sigma} \to \Sigma \times_{\h/\!\!/W}\h$ is an isomorphism over the regular semisimple locus and yields a minimal resolution of $\chi_{\Sigma}^{-1}(0)$. 

\item[\emph{(ii)}] We have an isomorphism $(\textrm{\emph{R}}^2\Psi_X\mathbb{Q}_{\ell})_x \simeq V_{\ell}(-1)$ where the $I_K$-action on the right is a restriction of the reflection representation along a map $\rho_a: I_K \to W$.
\end{itemize}  
\end{Lemma}

\begin{proof} We first explain (i). Let $e = B_{n-1} \oplus 0$ where $B_{n-1}$ is the $(n-1)\times(n-1)$ matrix with $1$ along its superdiagonal and $0$ in all other entries. Fiberwise its Jordan type is $(n-1,1)$ so $e$ is fiberwise subregular nilpotent. Let
\[\Sigma\coloneqq \{e + \sum_{i=1}^{n-1}a_i E_{n-1,i} + xE_{n-1,n}-y E_{n,1}+ z E_{n,n}\} \simeq \mathbb{A}_{\Lambda}^{n+2}\]
Computing the characteristic polynomial of matrices in $\Sigma$, we get
\begin{equation}\label{eq: slice identification for gln}\Sigma \simeq \Spec\left(\frac{\Lambda[x,y,z, t_1, \cdots, t_n]}{xy-z^n+t_1z^{n-1} + \cdots + t_n}\right)\end{equation}
Slice $\Sigma$ is the relative version of the matrix slice constructed in \cite[\S3.2.4]{Mirkovic2022Comparison} for Jordan type $(n-1,1)$, and after completing the map $\Sigma \to \h/\!\!/S_n\simeq \Spec(\Lambda[t_1,\cdots,t_n])$ it yields a formal miniversal deformation for singularity $A_{n-1}$ (see Proposition \ref{prop: explicit miniversal formal rdp deformation} and algebra $B$ in Example \ref{example: miniversal defs for An sings}).

Put $\mathfrak{a}=\Sigma-e$. A matrix computation shows that the trace pairing $(A,B)\mapsto \textrm{tr}(AB)$ restricts to a perfect pairing between $\mathfrak{a}$ and centralizer $Z_{\g}(e)$ so $\g = [\g,e]\oplus \mathfrak{a}$. The proof of Proposition \ref{proposition: construction of relative homogeneous transverse slice} therefore applies to show $\Sigma$ is a homogeneous transverse slice at $e$.

By the proof of Theorem \ref{thm: all the properties of simultaneous resolution diagrams} (i) which applies to $\mathfrak{g} = \mathfrak{gl}_n$, we have that $\widetilde{\chi}:\widetilde{\g} \to \h$ is smooth and so restricting the Grothendieck alteration to $\Sigma$ as usual shows that $\widetilde{\Sigma} \to \h$ is smooth. The last statement follows from resolution $(\widetilde{\Sigma}_k)_0 \to \chi_{\Sigma_k}^{-1}(0)$ having a chain of $n-1$ rational $(-2)$ curves as exceptional divisor, proving minimality.

For (ii), we first identify $V_{\ell}$: the Chern class argument of Proposition \ref{prop: springer w-action on cohom of B_x} gives an $S_n$-equivariant map $\mathbb{Q}_{\ell}^n \to \coh^2(\mathcal{B}_{e_k}, \mathbb{Q}_{\ell}(1))$ which has coordinates $(\lambda_i-\lambda_{i+1})_i$. The diagonal embedding then gives $V_{\ell} \simeq \mathbb{Q}_{\ell}^n/\mathbb{Q}_{\ell}(1,\cdots,1)$.

Now $\widehat{\Oring}_{X,x}$ is a formal deformation of the $A_{n-1}$-singularity over $\Spec(\Lambda)$ so there exist $\phi, Y,y$ as in Section \ref{section: monodromy springer w-actions} so that $(\textrm{R}^2\Psi_X\mathbb{Q}_{\ell})_x \simeq (\textrm{R}^2\Psi_Y\mathbb{Q}_{\ell})_y$ by Proposition \ref{prop: isom nearby cycle stalks under formal nbd isom}. The arguments of Proposition \ref{prop: simult resol over Lambda-L}, Corollary \ref{cor: gamma isom of nearby cycle stalks}, Lemma \ref{lemma: specialization is w and IK equivariant also delta inverse to gamma} and Proposition \ref{prop: springer and monodromy IK actions coincide} apply verbatim to show $(\textrm{R}^2\Psi_Y\mathbb{Q}_{\ell})_y \simeq V_{\ell}(-1)$.\end{proof}

\begin{proof}[Proof of Theorem \ref{thm: main theorem about rdps}] We first assume $p$ is very good. Since $X/\Lambda$ is proper we have $\coh^2(X_k, \textrm{R}\Psi_X\mathbb{Q}_{\ell})\simeq H^2(X_{\overline{K}},\mathbb{Q}_{\ell})$, so the $I_K$-equivariant triangle in Definition \ref{def: nearby cycles} (ii) gives
\begin{equation}\label{eq: cohom les from triangle}0 \to \coh^2(X_k,\mathbb{Q}_{\ell}) \to \coh^2(X_{\overline{K}},\mathbb{Q}_{\ell}) \overset{(*)}{\to} (\textrm{R}^2\Psi_X\mathbb{Q}_{\ell})_x \to \coh^3(X_k, \mathbb{Q}_{\ell}) \to \coh^3(X_{\overline{K}},\mathbb{Q}_{\ell}) \to 0\end{equation}
since $(\textrm{R}^i\Psi_X\mathbb{Q}_{\ell})_x = 0$ for $i\neq0,2$. To check surjectivity of $(*)$, let $\widetilde{X} \to X\times_{\Lambda}\Lambda_L$ be a simultaneous resolution in the category of algebraic spaces over $\Lambda_L$ as in Lemma \ref{lemma: global simult resol of X as algebraic space}. Then $(*)$ is identified with restriction $i^*:\coh^2(\widetilde{X}_k,\mathbb{Q}_{\ell}) \to \coh^2(\mathcal{B}_{e_k},\mathbb{Q}_{\ell})$ where $i: \mathcal{B}_{e_k} \hookrightarrow \widetilde{X}_k$ is identified with the exceptional divisor of minimal resolution $\widetilde{X}_k$. Note $\coh^2(X_{\overline{K}},\mathbb{Q}_{\ell})\simeq \coh^2(\widetilde{X}_{\overline{K}},\mathbb{Q}_{\ell})\simeq \coh^2(\widetilde{X}_k,\mathbb{Q}_{\ell})$, the latter because $\widetilde{X}$ is smooth.

Label the exceptional curves in $\widetilde{X}_k$ by $E_1, \cdots, E_r$ where $r=\textrm{rk}(\Phi)$. By Proposition \ref{prop: springer w-action on cohom of B_x}, restricting $\mathcal{B}_{e_k} = \cup_i E_i$ to its irreducible components gives \begin{equation}\label{eq: identifying stuff}\coh^2(\mathcal{B}_{e_k}, \mathbb{Q}_{\ell}(1)) \simeq \bigoplus_{i=1} ^r\coh^2(E_i, \mathbb{Q}_{\ell}(1)) \simeq V_{\ell}\end{equation}
Let $e_i = c_1(\Oring(E_i)) \in \coh^2(\widetilde{X}_k,\mathbb{Q}_{\ell}(1))$ denote the Chern classes of $E_i$, then the identification in (\ref{eq: identifying stuff}) sends $i^*e_j \mapsto (\deg(\Oring(E_j)\lvert_{E_i}))_i = (E_j\cdot E_i)_i$ so overall we have a linear map
\[V_{\ell} \to \coh^2(\widetilde{X}_k,\mathbb{Q}_{\ell}(1)) \overset{i^*}{\to} \coh^2(\mathcal{B}_{e_k},\mathbb{Q}_{\ell}(1)) \simeq V_{\ell}\]
where the leftmost map is $(a_i)\mapsto (\sum a_ie_i)$. The matrix of the composite map is the (negative of the) Cartan matrix, whose determinant is nonzero and so $i^*$ is surjective. We deduce from this that $\coh^i(X_{\overline{K}},\mathbb{Q}_{\ell})\simeq \coh^i(X_k,\mathbb{Q}_{\ell})$ is a trivial $I_K$-module for $i \neq 2$.

Smooth proper base-change for $\widetilde{X}$ shows that $\coh^2(X_{\overline{K}},\mathbb{Q}_{\ell})$ is a trivial $I_L$-module, so (\ref{eq: cohom les from triangle}) is a sequence of $\Gal(L/K)$-representations i.e.\ $(*)$ splits and we get (\ref{eq: equivariant splitting}) $I_K$-equivariantly. This yields (i) for $p$ very good. The remaining case is when $x$ is an $A_n$-singularity and $p \mid n+1$, which is treated in Lemma \ref{lemma: p not very good for An}.

For (ii), note that by Proposition \ref{prop: isom nearby cycle stalks under formal nbd isom} the $I_K$-representation $(\textrm{R}^2\Psi_X\mathbb{Q}_{\ell})_x$ depends only on a formal affine neighborhood of $(X,x)$ over $\Lambda$, hence so does $\ker\rho_a \subset I_K$ and $L = \overline{K}^{\textrm{ker}\rho_a}$.

It remains to show the minimality of $L$.  Let $M/K$ be a finite extension and suppose $X_M$ admits a smooth proper model over $\Lambda_M$ (the normalization of $\Lambda$ inside $M$). Smooth proper base-change implies $I_M$ acts trivially on $\coh^2(X_{\overline{K}},\mathbb{Q}_{\ell})$ so $I_M \subset \ker \rho_a = I_L$; as $k$ is algebraically closed this implies $L \subset M$. Conversely smooth proper model $\widetilde{X}$ over $\Lambda_L$ base-changes to a smooth proper model over $\Lambda_M$. We conclude $L/K$ is the minimal extension over which $X_K$ has good reduction in the category of algebraic spaces.\end{proof}

\begin{Rem}\label{remark on algebraic space for good reduction} Lemma \ref{lemma: global simult resol of X as algebraic space} produces a smooth proper algebraic space model for $X$ after a finite base extension. We do \emph{not} in general get that $X$ has potential good reduction in the category of schemes, so it is necessary to enlarge to the category of algebraic spaces. In \cite[\S5.2]{matsumoto2015good}, Matsumoto constructs a surface degeneration \[X=\{w^2=f(x,y,z)\} \subset \mathbb{P}_{\Lambda}(1,1,1,3)\] where $f(x,y,z)$ is a homogeneous sextic polynomial, $\Lambda=W(k)$ and $\textrm{char}(k)=p\geq 7$, so that $X_K$ is a smooth K3 surface of geometric Picard rank $1$ and special fiber $X_k$ has an ordinary double point. By \cite[Prop.\ 5.1]{matsumoto2015good} $X_K$ does not admit a smooth proper scheme model over any finite extension.
\end{Rem}

When $p$ is sufficiently good (see Definition \ref{definition-of-good-primes} (iii)) we have $p \nmid |\Gal(L/K)|$ so pro-$p$ wild inertia $P_K$ acts trivially, $L/K$ is a tame extension and $\Gal(L/K)=\langle T \rangle$ up to conjugacy. So as a special case of Theorem \ref{thm: main theorem about rdps} we have:

\begin{Cor} For $p$ sufficiently good, the $\ell$-adic monodromy is generated by a single element $T \in W$ of order $|\Gal(L/K)|$, determined up to conjugacy and acting on the $V_{\ell} = \Phi^{\vee}\otimes_{\mathbb{Z}}\mathbb{Q}_{\ell}$ summand of $\coh^2(X_{\overline{K}},\mathbb{Q}_{\ell}(1))$ through the reflection representation. \end{Cor}

\section{The case of general rational singularities}\label{section: extension to rational sings}
In this section, we explain how to extend Theorem \ref{thm: main thm in intro} to the case where $X/\Lambda$ is a surface degeneration with $X_k$ having a unique rational singularity $x \in X(k)$. We retain the hypotheses of Section \ref{subsection-classic-nearby-cycles} regarding $\Lambda$.

Write $R = \widehat{\Oring}_{X_k,x}$ and recall that $R$ admits formal miniversal deformations which admit algebraic representatives (see Section \ref{section: artin components}). Let $A \in \textrm{Comp}_{\Lambda,k}$ denote the base of a formal miniversal deformation of $R$, then there exists an algebraic classifying map $\phi: \Spec(\Lambda) \to \Spec(A)$ induced from deformation $\widehat{\Oring}_{X,x}$ of $R$. We will throughout assume the following:

\begin{assumptionstar} The classifying map $\phi$ factors as $\Spec(\Lambda) \to \Spec(A/I) \to \Spec(A)$ where $\Spec(A/I)$ is the Artin component of $\Spec(A)$ (Definition \ref{def: artin component}).
    
\end{assumptionstar}

\begin{Lemma}\label{lemma: existence of simultaneous resol for rational sings with artin comp} Suppose $X/\Lambda$ induces a formal deformation of $R$ with classifying map $\phi$ satisfying \emph{$(*)$}. Then there exists a finite extension $\Lambda_L$ of $\Lambda$ such that the base-change $X'\coloneqq X\times_{\Spec(\Lambda)}\Spec(\Lambda_L)$ admits a simultaneous resolution $\widetilde{X} \to X'$ over $\Lambda_L$.

\begin{proof} By Proposition \ref{prop: properties of the artin component} there exists a finite $A$-algebra $C$ that is formally smooth over $\Lambda$ and such that the algebraic realization $\mathcal{U} \to \Spec(A)$ of the miniversal deformation of $R$ admits a simultaneous resolution after base-change $\Spec(C) \to \Spec(A)$. 

Let $\pi: \widetilde{\mathcal{U}} \to \mathcal{U}_C\coloneqq\mathcal{U}\times_{\Spec(A)}\Spec(C)$ be a simultaneous resolution and put $\Lambda' = C\otimes_{A/I}\Lambda$, so $\Spec(\Lambda') \to \Spec(\Lambda)$ is finite surjective and in particular $\Lambda'\otimes_{\Lambda} K \neq 0$. Choose a closed point of $\Spec(\Lambda'\otimes_{\Lambda}K)$ which corresponds to a finite extension $L/K$, and let $\Lambda_L$ be the normalization of $\Lambda$ in $L$, which is also a complete discrete valuation ring with residue field $k$.

Write $\mathcal{U}_{\Lambda}\coloneqq \mathcal{U}\times_{\Spec(A)}\Spec(\Lambda)$. By construction there is a point $u \in (\mathcal{U}_{\Lambda})_k$ such that $\widehat{\Oring}_{X,x}\simeq \widehat{\Oring}_{\mathcal{U}_{\Lambda},u}$ and by the proof of Proposition \ref{prop: isom nearby cycle stalks under formal nbd isom} we can choose a pointed \'etale neighborhood $(T,t)$ of both $(X, x)$ and $(\mathcal{U}_{\Lambda},u)$. Pulling back the simultaneous resolution $\pi$ to $\Spec(\Lambda_L)$ we get a simultaneous resolution $\pi_{\Lambda}\colon \widetilde{\mathcal{U}}_{\Lambda_L} \to \mathcal{U}_{\Lambda_L}$ over $\Lambda_L$. The modification gluing argument of Lemma \ref{lemma: global simult resol of X as algebraic space} applied to $\pi_{\Lambda}$ and \'etale neighborhood $(T,t)$ yields a modification $\widetilde{X} \to X'$ that is an isomorphism over $\Spec(L)$ and the minimal resolution of $X_k$ over $\Spec(k)$.\end{proof}
    
\end{Lemma}

\begin{Prop}\label{prop: existence of RDP model} Suppose $X/\Lambda$ admits a simultaneous resolution after a finite extension $L/K$. Then there exists a simultaneous partial resolution $f: X^{\textrm{\emph{RDP}}} \to X$ over $\Lambda$ in the sense that $f$ is proper with $f\lvert_{\Spec(K)}$ an isomorphism, $X^{\textrm{\emph{RDP}}}$ is flat over $\Lambda$ and $f\lvert_{\Spec(k)}$ is a partial resolution of $X_k$, obtained by contracting the $(-2)$-curves in its minimal resolution.
\end{Prop}

\begin{proof} Set $\mathcal{L}\coloneqq \omega_{\widetilde{X}/\Lambda_L}$ for the relative dualizing sheaf of the simultaneous resolution $\widetilde{X}$ of $X'=X\times_{\Spec(\Lambda)}\Spec(\Lambda_L)$ and denote the exceptional irreducible curves in the minimal resolution $\widetilde{X}_k$ by $E_1, \cdots E_r$. By adjunction $\deg(\mathcal{L}\lvert_{E_i})= \deg(\omega_{E_i})-E_i^2 = -2-E_i^2 \geq 0$, the last inequality following from $\widetilde{X}_k$ being a minimal resolution of a rational singularity, and thus $E_i^2 \leq -2$ (\cite[Cor.\ 27.3]{Lipman1969}). 

Write $\pi_k: \widetilde{X}_k \to X_k$ for the minimal resolution and note $\textrm{R}^1(\pi_k)_*\Oring_{\widetilde{X}_k} = 0$ as $x \in X(k)$ is a rational singularity. We localize by taking $S = \Spec(\Oring_{X_k,x})$, $g\colon T = S\times_{X_k}\widetilde{X}_k \to S $ so that $\textrm{R}^1g_*\Oring_T = 0$ by flat base-change. As the fibers of $g$ have dimension $\leq 1$ and $S$ is affine, by the Leray spectral sequence we get $\coh^1(T, \Oring_T) = 0$, so applying \cite[Thm.\ 12.1 (ii)]{Lipman1969} we get that $\mathcal{L}^n\lvert_T$ is globally generated and $\coh^1(T, \mathcal{L}^n\lvert_T) = 0$ for all $n\geq 1$. Since $\pi_k$ is an isomorphism outside of $x$, we get that $\mathcal{L}^n\lvert_{\widetilde{X}_k}$ is relatively globally generated in the sense that $(\pi_k)^*(\pi_k)_*\mathcal{L}^n\lvert_{\widetilde{X}_k} \twoheadrightarrow \mathcal{L}^n\lvert_{\widetilde{X}_k}$ and $\textrm{R}^1(\pi_k)_*\mathcal{L}^n\lvert_{\widetilde{X}_k} = 0$. 

Write $\pi:\widetilde{X} \to X'$ for the simultaneous resolution and let $i: \widetilde{X}_k \hookrightarrow \widetilde{X}$. If $\varpi \in \Lambda_L$ is a uniformizer, applying $\pi_*$ to
\[0 \to \mathcal{L}^n \overset{\cdot \varpi}{\to} \mathcal{L}^n \to i_*\left(\mathcal{L}^n\lvert_{\widetilde{X}_k}\right) \to 0\]
yields exact sequence
\begin{equation}\label{eq: les for rdp model}\cdots \to \textrm{R}^1\pi_*\mathcal{L}^n \overset{\cdot \varpi}{\to}\textrm{R}^1\pi_*\mathcal{L}^n \to j_*\textrm{R}^1(\pi_k)_*\mathcal{L}^n\lvert_{\widetilde{X}_k} = 0, \;\;\;\textrm{where} \; j: X_k \hookrightarrow X'\end{equation}
Note that for any $y\in X_k$, setting $M = (\textrm{R}^1\pi_*\mathcal{L}^n)_y$ yields $M = \varpi M$ by (\ref{eq: les for rdp model}) so $M=0$ by Nakayama. Since $\pi$ is an isomorphism over $X'\!\setminus\! X_k$ we get overall that $\textrm{R}^1\pi_*\mathcal{L}^n=0$ and (\ref{eq: les for rdp model}) yields $\pi_*\mathcal{L}^n/\varpi(\pi_*\mathcal{L}^n) \simeq j_*(\pi_k)_*\mathcal{L}^n\lvert_{\widetilde{X}_k}$. Setting $\mathcal{F}\coloneqq \textrm{coker}(\pi^*\pi_*\mathcal{L}^n \to \mathcal{L}^n)$ for the cokernel of the evaluation map, we have $i^*\mathcal{F} = \textrm{coker}((\pi_k)^*(\pi_k)_*\mathcal{L}^n\lvert_{\widetilde{X}_k} \to \mathcal{L}^n\lvert_{\widetilde{X}_k})$ by proper base-change, and this cokernel is zero as $\mathcal{L}^n\lvert_{\widetilde{X}_k}$ is globally generated. Since $\mathcal{F}\lvert_{\widetilde{X}\setminus\widetilde{X}_k} = 0$ we have $\mathcal{F}=0$, i.e.\ $\pi^*\pi_*\mathcal{L}^n \twoheadrightarrow \mathcal{L}^n$, so $\mathcal{L}^n$ is (relatively) globally generated.

Set $\mathcal{A} = \bigoplus_{n\geq 0} \pi_*\mathcal{L}^n$ and $Y = \textrm{Proj}_{X'}\mathcal{A}$. By the argument above, the graded $\Oring_{\widetilde{X}}$-algebra map $\pi^*\mathcal{A} \to \bigoplus_{n \geq 0} \mathcal{L}^n$ is surjective in degree $1$ and $\mathcal{A}$ is generated in degree $1$, so by \stacksref{085D} the simultaneous resolution $\pi: \widetilde{X} \to X'$ factors as
\[\widetilde{X} \overset{\alpha}{\longrightarrow} Y \overset{\beta}{\longrightarrow} X'\]
with $\alpha$ proper, $\alpha^*\Oring_Y(1)\simeq \mathcal{L}$ and $\beta$ projective. By construction $Y$ is flat over $\Lambda_L$ and its generic and special fibers are respectively
\[Y_L \simeq \textrm{Proj}_{X_L}\bigoplus_{n \geq 0}\omega_{X_L}^n = X_L, \;\;\; Y_k \simeq \textrm{Proj}_{X_k}\bigoplus_{n \geq 0} (\pi_k)_*\omega_{\widetilde{X}_k}^n\]

Let $\widetilde{X}_k \overset{h}{\to} Y'_k \overset{\rho}{\to} X_k$ be a factorization of the minimal resolution so that $h: \widetilde{X}_k \to Y'_k$ is a contraction of the exceptional $(-2)$-curves; $Y_k'$ is a partial resolution of $X_k$ which has only rational double point singularities and $h^*\omega_{Y_k'}\simeq \omega_{\widetilde{X}_k}$ with $\omega_{Y_k'}$ $\rho$-ample (\cite[\S1 and \S3, Lem.\ 2]{Lipman1979Double}). Thus by relative ampleness we get
\[Y_k \simeq \textrm{Proj}_{X_k}\bigoplus_{n \geq 0} (\pi_k)_*\omega_{\widetilde{X}_k}^n \simeq \textrm{Proj}_{X_k}\bigoplus_{n\geq 0} \rho_*\omega_{Y_k'}^n \simeq Y_k'\]
so $Y_k$ is a partial resolution of $X_k$ with at worst RDP singularities.

To obtain a model over $\Lambda$, let $U$ be the relative smooth locus of $X/\Lambda$ and let $P = \textrm{Proj}_X\left(\bigoplus_{n\geq0} \textrm{Sym}^n_{\Oring_X}(\omega_{X/\Lambda})\right)$ so that $P\lvert_U\simeq U$ (as $\omega_{X/\Lambda}$ is invertible over $U$). Write $\overline{U}$ for the scheme-theoretic closure of $U \hookrightarrow P$. Since projectivization commutes with base-change, \[P'\coloneqq P\times_{\Spec(\Lambda)}\Spec(\Lambda_L) \simeq \textrm{Proj}_{X'}\left(\bigoplus_{n\geq 0}\textrm{Sym}_{\Oring_{X'}}^n(\omega_{X'/\Lambda_L})\right)\]
and the schematic closure of $U\times_{\Spec(\Lambda)}\Lambda_L$ in $P'$ is identified with $Y$ by \cite[\S2]{Lipman1979Double}. Set $X^{\textrm{RDP}}\coloneqq \overline{U}$ and $f$ the projective composite map $\overline{U}\hookrightarrow P \to X$, then by the above we have $X^{\textrm{RDP}}\times_{\Spec(\Lambda)}\Spec(\Lambda_L) \simeq Y$. Note $X^{\textrm{RDP}}_K \simeq X_K$, $X^{\textrm{RDP}}_k \simeq Y_k$ and $X^{\textrm{RDP}}$ is flat over $\Lambda$ since $Y$ is flat over $\Lambda_L$ and $\Spec(\Lambda_L)\to\Spec(\Lambda)$ is faithfully flat, so $X^{\textrm{RDP}}/\Lambda$ is the desired model with a partial simultaneous resolution $f$.\end{proof}

Note that $X^{\textrm{RDP}}_K$ is smooth and $X^{\textrm{RDP}}_k$ contains only rational double points. We can therefore apply the arguments of Theorem \ref{thm: main theorem about rdps} and deduce the following.

\begin{Thm}\label{thm: rational sing monodromy} Let $X/\Lambda$ be a surface degeneration with $X_k$ having a rational singularity $x \in X(k)$ so that the induced formal deformation of $(X_k,x)$ satisfies \emph{Assumption }$(*)$. Let \[\Psi = \bigoplus_{E_i^2=-2}\mathbb{Z}[E_i] \subset \Phi = \bigoplus_{i=1}^r\mathbb{Z}[E_i]\]
be an inclusion of exceptional lattices, where $E_i$ are the exceptional curves of the minimal resolution $\widetilde{X}_k \to X_k$. $\Psi$ is a direct sum of finitely many connected $ADE$ root lattices and its Weyl group $W$ acts on $\Phi$ via the standard reflections. If $p$ is good for each $ADE$ root lattice in $\Psi$ then 

\begin{itemize}
    \item[\emph{(i)}] The monodromy $I_K$-action on $\coh^i(X_{\overline{K}},\mathbb{Q}_{\ell})$ is trivial for $i\neq 2$. For $i=2$ let $L$ be the fixed field of the kernel of the $I_K$-action, then there exists a homomorphism $\rho: I_K \to W$ and an $I_K$-equivariant splitting

\begin{equation}\label{eq: equivariant splitting for rational singularities}
\coh^2(X_{\overline{K}}, \mathbb{Q}_{\ell}) \simeq \coh^2(X_k, \mathbb{Q}_{\ell})\oplus V_{\ell} (-1)
\end{equation}
where the $I_K$-action is trivial on the first summand, and on $V_{\ell}=\Phi\otimes_{\mathbb{Z}}\mathbb{Q}_{\ell}$ it factors through $\Gal(L/K)$, where it is identified with the restriction of the reflection $W$-representation along $\rho$.

\item[\emph{(ii)}] $L$ is the minimal Galois extension over which $X_K$ has good reduction in the category of algebraic spaces. Moreover the isomorphism class of the $I_K$-representation $\coh^2(X_{\overline{K}},\mathbb{Q}_{\ell})$ depends, up to trivial summands, only on a formal $\Lambda$-isomorphism class of $(X,x)$. In particular $L$ depends only on this class. 

\end{itemize}
    
\end{Thm}

\begin{proof} By Assumption $(*)$, Lemma \ref{lemma: existence of simultaneous resol for rational sings with artin comp} and Proposition \ref{prop: existence of RDP model} there exists a simultaneous resolution $\pi: \widetilde{X} \to X$ over some finite extension $\Lambda'/\Lambda$ with $\pi_k: \widetilde{X}_k \to X_k$ the minimal resolution, and there exists a flat proper model $X^{\textrm{RDP}}/\Lambda$ such that $\pi$ factors through $\alpha: \widetilde{X} \to X^{\textrm{RDP}}\times \Spec(\Lambda')$, $X^{\textrm{RDP}}_K\simeq X_K$ and $X^{\textrm{RDP}}_k$ is a  partial resolution of $X_k$, obtained by contracting the $\pi_k$-exceptional $(-2)$ curves in $\widetilde{X}_k$; in particular $X^{\textrm{RDP}}_k$ has finitely many rational double points and is otherwise regular.

Let $U_\ell = \Psi\otimes_{\mathbb{Z}}\mathbb{Q}_{\ell}$ and $T_{\ell} = V_{\ell} \cap U_{\ell}^{\perp}$. By the nearby cycles triangle (\ref{vanishing-cycles-definition}), since $X_k$ has a rational singularity and therefore the exceptional divisor in $\widetilde{X}_k$ is a tree of rational curves, we get
\[0 \to \coh^2(X_k, \mathbb{Q}_{\ell}(1)) \overset{\textrm{sp}}{\to} \coh^2(X_{\overline{K}},\mathbb{Q}_{\ell}(1)) \overset{\phi}{\to}(\textrm{R}^2\Psi_X\mathbb{Q}_{\ell}(1))_x \to \coh^3(X_k, \mathbb{Q}_{\ell}(1))\]
Setting $E = \cup_i E_i \subset \widetilde{X}_k$ for the exceptional divisor, we have as in the proof of Theorem \ref{thm: main theorem about rdps} that $\phi$ is identified with  the restriction mapping 
\[\coh^2(\widetilde{X}_k,\mathbb{Q}_{\ell}(1)) \to \coh^2(E, \mathbb{Q}_{\ell}(1)) \simeq \bigoplus_i \coh^2(E_i,\mathbb{Q}_{\ell}(1))\simeq V_{\ell}\]
and moreover it is surjective as the intersection matrix $(E_i\cdot E_j)_{ij}$ is nondegenerate (\cite[Lem.\ 14.1]{Lipman1969}). Consequently $\phi$ is surjective, $\coh^i(X_{\overline{K}},\mathbb{Q}_{\ell})\simeq \coh^i(X_k,\mathbb{Q}_{\ell})$ for $i\neq 2$ has trivial monodromy action and, upon identifying $\textrm{sp}$ with $\pi_k^*$, we have \[\coh^2(X_{\overline{K}},\mathbb{Q}_{\ell}(1))\simeq V_{\ell}^{\perp}\oplus V_{\ell} = \pi_k^*\coh^2(X_k,\mathbb{Q}_{\ell}(1))\oplus V_{\ell}\]
We can similarly apply the nearby cycles triangle (\ref{vanishing-cycles-definition}) to $X^{\textrm{RDP}}$, noting that \[\textrm{R}^2\Psi_{X^{\textrm{RDP}}}\mathbb{Q}_{\ell} \simeq \bigoplus_{x_i} (i_{x_i})_* (\textrm{R}^2\Psi_{X^{\textrm{RDP}}}\mathbb{Q}_{\ell})_{x_i}\]
for $i_{x_i}: x_i\hookrightarrow X_k^{\textrm{RDP}} $ the inclusions of the finitely many rational double points $x_i$. The $\alpha_k$-exceptional curves of $\widetilde{X}_k \to X_k^{\textrm{RDP}}$ are exactly the $(-2)$-curves $E_i$, whence the above argument yields
\[\coh^2(X_{\overline{K}},\mathbb{Q}_{\ell}(1)) \simeq U_{\ell}^{\perp}\oplus U_{\ell} = \alpha_k^*\coh^2(X_k^{\textrm{RDP}},\mathbb{Q}_{\ell}(1)) \oplus U_{\ell}\]
Now the non-degenerate intersection pairings on $\Psi$ and $\Phi$ yield $V_{\ell} = U_{\ell} \oplus T_{\ell}$ and $U_{\ell}^{\perp} = V_{\ell}^{\perp} \oplus T_{\ell}$ so $\coh^2(X_k^{\textrm{RDP}},\mathbb{Q}_{\ell}(1)) \simeq T_{\ell} \oplus \coh^2(X_k, \mathbb{Q}_{\ell}(1))$ equipped with the trivial $I_K$-action.

Applying Theorem \ref{thm: main theorem about rdps} to $X^{\textrm{RDP}}/\Lambda$ (in the form of Remark \ref{remark: variants of main rdp theorem} (iii)) yields an $I_K$-equivariant splitting
\[\coh^2(X_{\overline{K}},\mathbb{Q}_{\ell}(1)) \simeq \coh^2(X_k^{\textrm{RDP}},\mathbb{Q}_{\ell}(1)) \oplus U_{\ell}\]
where $I_K$ acts trivially on the first summand and via restricting the reflection representation of $W$ along  a homomorphism $\rho: I_K \to W$ factoring through $\Gal(L/K)$, where $L$ is the fixed field of $\ker(\rho)$. Thus we also have an $I_K$-equivariant splitting
\[\coh^2(X_{\overline{K}},\mathbb{Q}_{\ell}(1))\simeq \coh^2(X_k,\mathbb{Q}_{\ell}(1)) \oplus T_{\ell} \oplus U_{\ell}\]
with $I_K$ acting trivially on $T_{\ell}$ and via $\rho$ on $U_{\ell}$. Twisting by $(-1)$ yields (i), and (ii) proceeds exactly as Theorem \ref{thm: main theorem about rdps} (ii). Note in particular if $\Psi = \emptyset$, meaning there are no $(-2)$ curves in the minimal resolution, then the monodromy action is trivial and $X$ already has good reduction over $K$.\end{proof}

We note that the statements in Remark \ref{remark: variants of main rdp theorem} also apply to this theorem, namely we can replace `$p$ is good for $\Psi$' with $p>5$, and there is a variant of the theorem that holds when $X_k$ has finitely many isolated rational singularities.

We also remark that Assumption $(*)$ is crucial in producing a simultaneous resolution of $X/\Lambda$ after a finite extension. There are rational singularities and degenerations which deform outside the Artin component, and which don't admit simultaneous resolution after any finite base-change.

\begin{ex} Assume $p\neq 2$, let $\pi\in\Lambda$ be a uniformizer and set \[Y = \{xy-z^2 = \pi w^2\} \subset \mathbb{P}_{\Lambda}^3\] Let $\sigma: [x:y:z:w] \mapsto [x:y:z:-w]$ be the involution action and put $X\coloneqq Y/\langle \sigma \rangle$. Then $X/\Lambda$ is a surface degeneration with $X_K \simeq \mathbb{P}_K^2 $ and $X_k$ has a singularity of type $\frac{1}{4}(1,1)$ at $x=[0:0:0:1]$. Indeed, after change of coordinates and completion one has $\widehat{\Oring}_{X_k,x} \simeq k[\![a,b]\!]^{\mu_4}$ where $\zeta \in \mu_4$ acts as $(a,b) \mapsto (\zeta a, \zeta b)$. 

The minimal resolution of $X_k$ is $\mathbb{F}_4 \to X_k$ where $\mathbb{F}_4$ is a Hirzebruch surface, locally the resolution of the cone over a rational quartic curve. Thus the singularity is rational. However, if $X/\Lambda$ admitted a simultaneous resolution $\widetilde{X}$ after an extension $\Lambda_L/\Lambda$, we would have $\coh^2(\widetilde{X}_{\overline{L}},\mathbb{Q}_{\ell}) \simeq \coh^2(\widetilde{X}_k,\mathbb{Q}_{\ell})$ by smooth proper base-change. Since $\widetilde{X}_L\simeq \mathbb{P}^2$ and $\widetilde{X}_k \simeq \mathbb{F}_4$, this cannot happen as $h^2(\mathbb{P}^2, \mathbb{Q}_{\ell}) = 1 < h^2(\mathbb{F}_4, \mathbb{Q}_{\ell})=2$.

Note in this example that the $\ell$-adic monodromy is trivial as $\coh^2(\mathbb{P}^2_{\overline{K}},\mathbb{Q}_{\ell})\simeq \mathbb{Q}_{\ell}(-1)$ and $I_K$ acts trivially on the Tate twist. What fails here specifically is the existence of a simultaneous resolution, as $X/\Lambda$ deforms the singularity outside its Artin component. In general, we do not know whether every rational surface degeneration $X/\Lambda$ necessarily yields finite $\ell$-adic monodromy on $\coh^2(X_{\overline{K}},\mathbb{Q}_{\ell})$.
    
\end{ex}

\bibliographystyle{alpha-shorthand}
\bibliography{newbib.bib}

\end{document}